\documentclass{amsart}
\usepackage{amssymb}
\usepackage{amscd}
\usepackage{verbatim}
\usepackage{epsfig}
\begin{document}
\newcommand\Mand{\ \text{and}\ }
\newcommand\Mor{\ \text{or}\ }
\newcommand\Mfor{\ \text{for}\ }
\newcommand\Real{\mathbb{R}}
\newcommand\RR{\mathbb{R}}
\newcommand\im{\operatorname{Im}}
\newcommand\re{\operatorname{Re}}
\newcommand\sign{\operatorname{sign}}
\newcommand\sphere{\mathbb{S}}
\newcommand\BB{\mathbb{B}}
\newcommand\TT{\mathbb{T}}
\newcommand\HH{\mathbb{H}}
\newcommand\dS{\mathrm{dS}}
\newcommand\ZZ{\mathbb{Z}}
\newcommand\NN{\mathbb{N}}
\newcommand\codim{\operatorname{codim}}
\newcommand\Sym{\operatorname{Sym}}
\newcommand\End{\operatorname{End}}
\newcommand\Span{\operatorname{span}}
\newcommand\Ran{\operatorname{Ran}}
\newcommand\ep{\epsilon}
\newcommand\Cinf{\cC^\infty}
\newcommand\dCinf{\dot \cC^\infty}
\newcommand\CI{\cC^\infty}
\newcommand\dCI{\dot \cC^\infty}
\newcommand\Cx{\mathbb{C}}
\newcommand\Nat{\mathbb{N}}
\newcommand\dist{\cC^{-\infty}}
\newcommand\ddist{\dot \cC^{-\infty}}
\newcommand\pa{\partial}
\newcommand\Card{\mathrm{Card}}
\renewcommand\Box{{\square}}
\newcommand\Ell{\mathrm{Ell}}
\newcommand\Char{\mathrm{Char}}
\newcommand\WF{\mathrm{WF}}
\newcommand\WFh{\mathrm{WF}_\semi}
\newcommand\WFb{\mathrm{WF}_\bl}
\newcommand\WFsc{\mathrm{WF}_\scl}
\newcommand\WFscb{\mathrm{WF}_{\scl,\bl}}
\newcommand\Vf{\mathcal{V}}
\newcommand\Vb{\mathcal{V}_\bl}
\newcommand\Vsc{\mathcal{V}_\scl}
\newcommand\Vscsus{\mathcal{V}_\scsus}
\newcommand\Vz{\mathcal{V}_0}
\newcommand\Hb{H_{\bl}}
\newcommand\bHb{\bar H_{\bl}}
\newcommand\dHb{\dot H_{\bl}}
\newcommand\Hbb{\tilde H_{\bl}}
\newcommand\Hsc{H_{\scl}}
\newcommand\Hscb{H_{\scbl}}
\newcommand\bHscb{\bar H_{\scbl}}
\newcommand\dHscb{\dot H_{\scbl}}
\newcommand\Hscsus{H_{\scsus}}
\newcommand\Ker{\mathrm{Ker}}
\newcommand\Range{\mathrm{Ran}}
\newcommand\Hom{\mathrm{Hom}}
\newcommand\Id{\mathrm{Id}}
\newcommand\sgn{\operatorname{sgn}}
\newcommand\ff{\mathrm{ff}}
\newcommand\tf{\mathrm{tf}}
\newcommand\esssupp{\operatorname{esssupp}}
\newcommand\supp{\operatorname{supp}}
\newcommand\vol{\mathrm{vol}}
\newcommand\Diff{\mathrm{Diff}}
\newcommand\Diffd{\mathrm{Diff}_{\dagger}}
\newcommand\Diffs{\mathrm{Diff}_{\sharp}}
\newcommand\Diffb{\mathrm{Diff}_\bl}
\newcommand\Diffsc{\mathrm{Diff}_\scl}
\newcommand\Diffscsus{\mathrm{Diff}_\scsus}
\newcommand\DiffbI{\mathrm{Diff}_{\bl,I}}
\newcommand\Diffbeven{\mathrm{Diff}_{\bl,\even}}
\newcommand\Diffz{\mathrm{Diff}_0}
\newcommand\Psih{\Psi_{\semi}}
\newcommand\Psihcl{\Psi_{\semi,\cl}}
\newcommand\Psisc{\Psi_\scl}
\newcommand\Psiscc{\Psi_\sccl}
\newcommand\Psiscb{\Psi_\scbl}
\newcommand\Psib{\Psi_\bl}
\newcommand\Psibc{\Psi_{\mathrm{bc}}}
\newcommand\Psibcdelta{\Psi_{\mathrm{bc},\delta}}
\newcommand\TbC{{}^{\bl,\Cx} T}
\newcommand\Tb{{}^{\bl} T}
\newcommand\Sb{{}^{\bl} S}
\newcommand\Tsc{{}^{\scl} T}
\newcommand\Tscsus{{}^{\scsus} T}
\newcommand\Ssc{{}^{\scl} S}
\newcommand\Sscsus{{}^{\scsus} S}
\newcommand\Lambdab{{}^{\bl} \Lambda}
\newcommand\zT{{}^{0} T}
\newcommand\Tz{{}^{0} T}
\newcommand\zS{{}^{0} S}
\newcommand\dom{\mathcal{D}}
\newcommand\cA{\mathcal{A}}
\newcommand\cB{\mathcal{B}}
\newcommand\cE{\mathcal{E}}
\newcommand\cG{\mathcal{G}}
\newcommand\cH{\mathcal{H}}
\newcommand\cU{\mathcal{U}}
\newcommand\cO{\mathcal{O}}
\newcommand\cF{\mathcal{F}}
\newcommand\cM{\mathcal{M}}
\newcommand\cQ{\mathcal{Q}}
\newcommand\cR{\mathcal{R}}
\newcommand\cI{\mathcal{I}}
\newcommand\cL{\mathcal{L}}
\newcommand\cK{\mathcal{K}}
\newcommand\cC{\mathcal{C}}
\newcommand\cX{\mathcal{X}}
\newcommand\cY{\mathcal{Y}}
\newcommand\cXsus{\mathcal{X}_\sus}
\newcommand\cYsus{\mathcal{Y}_\sus}
\newcommand\cP{\mathcal{P}}
\newcommand\cS{\mathcal{S}}
\newcommand\cZ{\mathcal{Z}}
\newcommand\cW{\mathcal{W}}
\newcommand\Ptil{\tilde P}
\newcommand\ptil{\tilde p}
\newcommand\chit{\tilde \chi}
\newcommand\yt{\tilde y}
\newcommand\zetat{\tilde \zeta}
\newcommand\xit{\tilde \xi}
\newcommand\taut{{\tilde \tau}}
\newcommand\phit{{\tilde \phi}}
\newcommand\mut{{\tilde \mu}}
\newcommand\taubsemi{\tau_{\bl,\hbar}}
\newcommand\lambdasemi{\lambda_{\hbar}}
\newcommand\sigmat{{\tilde \sigma}}
\newcommand\sigmah{\hat\sigma}
\newcommand\zetah{\hat\zeta}
\newcommand\etah{\hat\eta}
\newcommand\taub{\tau_{\bl}}
\newcommand\mub{\mu_{\bl}}
\newcommand\taubh{\hat\tau_{\bl}}
\newcommand\mubh{\hat\mu_{\bl}}
\newcommand\nuh{\hat\nu}
\newcommand\loc{\mathrm{loc}}
\newcommand\compl{\mathrm{comp}}
\newcommand\reg{\mathrm{reg}}
\newcommand\GBB{\textsf{GBB}}
\newcommand\GBBsp{\textsf{GBB}\ }
\newcommand\bl{{\mathrm b}}
\newcommand\scl{{\mathrm{sc}}}
\newcommand\scbl{{\mathrm{sc,b}}}
\newcommand\sccl{{\mathrm{scc}}}
\newcommand\scsus{{\mathrm{sc-sus}}}
\newcommand\sus{{\mathrm{sus}}}
\newcommand{\sH}{\mathsf{H}}
\newcommand{\cte}{\digamma}
\newcommand\cl{\mathrm{cl}}
\newcommand\hsf{\mathcal{S}}
\newcommand\Div{\operatorname{div}}
\newcommand\hilbert{\mathfrak{X}}
\newcommand\smooth{\mathcal{J}}
\newcommand\decay{\ell}
\newcommand\symb{j}

\newcommand\Hh{H_{\semi}}

\newcommand\bM{\bar M}
\newcommand\Xext{X_{-\delta_0}}

\newcommand\xib{{\underline{\xi}}}
\newcommand\etab{{\underline{\eta}}}
\newcommand\zetab{{\underline{\zeta}}}

\newcommand\xibh{{\underline{\hat \xi}}}
\newcommand\etabh{{\underline{\hat \eta}}}
\newcommand\zetabh{{\underline{\hat \zeta}}}

\newcommand\zn{z}
\newcommand\sigman{\sigma}
\newcommand\psit{\tilde\psi}
\newcommand\rhot{{\tilde\rho}}

\newcommand\hM{\hat M}

\newcommand\Op{\operatorname{Op}}
\newcommand\Oph{\operatorname{Op_{\semi}}}

\newcommand\innr{{\mathrm{inner}}}
\newcommand\outr{{\mathrm{outer}}}
\newcommand\full{{\mathrm{full}}}
\newcommand\semi{\hbar}

\newcommand\Feynman{\mathrm{Feynman}}
\newcommand\future{\mathrm{future}}
\newcommand\past{\mathrm{past}}

\newcommand\elliptic{\mathrm{ell}}
\newcommand\diffordgen{k}
\newcommand\difford{2}
\newcommand\diffordm{1}
\newcommand\diffordmpar{1}
\newcommand\even{\mathrm{even}}
\newcommand\dimn{n}
\newcommand\dimnpar{n}
\newcommand\dimnm{n-1}
\newcommand\dimnp{n+1}
\newcommand\dimnppar{(n+1)}
\newcommand\dimnppp{n+3}
\newcommand\dimnppppar{n+3}

\newcommand\sob{s}

\newtheorem{lemma}{Lemma}[section]
\newtheorem{prop}[lemma]{Proposition}
\newtheorem{thm}[lemma]{Theorem}
\newtheorem{cor}[lemma]{Corollary}
\newtheorem{result}[lemma]{Result}
\newtheorem*{thm*}{Theorem}
\newtheorem*{prop*}{Proposition}
\newtheorem*{cor*}{Corollary}
\newtheorem*{conj*}{Conjecture}
\numberwithin{equation}{section}
\theoremstyle{remark}
\newtheorem{rem}[lemma]{Remark}
\newtheorem*{rem*}{Remark}
\theoremstyle{definition}
\newtheorem{Def}[lemma]{Definition}
\newtheorem*{Def*}{Definition}

\newcommand{\mar}[1]{{\marginpar{\sffamily{\scriptsize #1}}}}
\newcommand\av[1]{\mar{AV:#1}}

\renewcommand{\theenumi}{\roman{enumi}}
\renewcommand{\labelenumi}{(\theenumi)}

\title[Resolvent near zero energy]{Resolvent near zero energy on Riemannian scattering (asymptotically conic) spaces}
\author[Andras Vasy]{Andr\'as Vasy}
\address{Department of Mathematics, Stanford University, CA 94305-2125, USA}

\email{andras@math.stanford.edu}


\thanks{The author gratefully
  acknowledges partial support from the NSF under grant numbers
  DMS-1361432 and DMS-1664683 and from a Simons Fellowship.}
\date{August 18, 2018}

\begin{abstract}
We give resolvent estimates near zero energy on Riemannian scattering, i.e.\
asymptotically conic, spaces, and their generalizations, using a
uniform microlocal Fredholm analysis framework.
\end{abstract}

\maketitle

\section{Introduction and results}
In this paper we consider geometric generalizations of Euclidean low energy resolvent
estimates, such as for the resolvent of the Euclidean
Laplacian plus a decaying potential, in a Fredholm framework. Such an analysis is relevant for
instance for the asymptotic behavior of the solutions of the wave
equation, though this connection will be pursued elsewhere as it
requires some additional ingredients. Indeed,
one motivation for the present paper is understanding waves on Kerr spacetimes.
However, another motivation for the present work is even just the
description of
Euclidean phenomena, namely
what is {\em really} happening for low energies. In one sense this has been addressed by Guillarmou and Hassell in a series of
works
\cite{Guillarmou-Hassell:Resolvent-I,Guillarmou-Hassell:Resolvent-II}
via constructing a parametrix for the resolvent family; here we
proceed by {\em directly} obtaining Fredholm estimates, which are
actually less technically involved, and also are not straightforward
to read off from the parametrix result, especially as we need to work
on variable order, or anisotropic, Sobolev spaces. A forthcoming
companion paper will give a slightly different perspective that is
even more amenable towards the study of wave propagation, though
somewhat less so for direct spectral theory applications.

Let us start by recalling the Euclidean results.
For this purpose, we initially let
$g_0$ be the Euclidean metric, 
$g$ metric on $\RR^n$ with $g-g_0\in S^{-\delta}$,\ $\delta>0$ (i.e.\
$g_{ij}-(g_0)_{ij}\in S^{-\delta}$), $g$ positive definite,
$V\in S^{-\delta}$, real.
Recall here that $S^m(\RR^n_z)$ is the  space of symbols of order $m$:
for all $\alpha$
$$
|D_z^\alpha a(z)|\leq
C_\alpha\langle z\rangle^{m-|\alpha|};\qquad\langle z\rangle=(1+|z|^2)^{1/2}.
$$
In order to make a connection with the upcoming generalization, we
also give these estimates in a different form (away from origin): for all $k$
and sequences $i_1,\ldots,i_k$, $j_1,\ldots,j_k$,
$$
|W_1\ldots W_k a(z)|\leq C_k \langle z\rangle^m,
$$
where
$W_r=z_{i_r} D_{z_{j_r}}$; of course, localized to the region where
$|z_n|>c|z_j|$ for $j\neq n$ ($c>0$) and $|z|>1$, it suffices to consider $W_r=z_n D_{z_{j_r}}$.
We note that $V$ non-real, but $\im V\in S^{-1-\delta}$ (so a slightly
stronger constraint), would only cause some possible finite rank issues below.

Under these assumptions
$$
H=\Delta_g+V
$$
is self-adjoint on $L^2(\RR^n)$, so $H-\lambda$,
$\lambda\in\Cx\setminus\RR$ is invertible, e.g.\ as a map
$$
H-\lambda:H^{s+2,l}\to H^{s,l},\ s,l\in\RR;
$$
standard elliptic theory implies that the inverse is independent of
the particular space chosen for the inversion (in the sense that e.g.\
the restriction of the inverse to the dense subspace given by Schwartz
functions is the same).
Moreover, the spectrum in $(-\infty,0)$ is discrete, with $0$ a
possible accumulation point (this happens e.g.\ for negative
Coulomb-like potentials, which are Coulomb-like in terms of decay at infinity);
$[0,\infty)$ the essential spectrum.
Here $H^{s,l}=\langle z\rangle^{-l} H^s$, where $H^s$ is the standard Sobolev space.

While $H-\lambda$ will no longer be invertible between the weighted
Sobolev spaces when $\lambda>0$, the limiting absorption
principle states that
\begin{equation}\label{eq:Euclidean-LAP}
(H-(\lambda\pm i0))^{-1}=\lim_{\ep \to 0}(H-(\lambda\pm i\ep))^{-1}
\end{equation}
exist e.g.\ as strong limits (indeed, norm limits) in $\cL(H^{s,l},H^{s+2,l'})$,
$l>\frac{1}{2}$, $l'<-\frac{1}{2}$ (so $l-l'>1$).
Under stronger assumptions (and $n\geq 3$, which will be assumed from
now), which are necessary in view of Coulomb-like potentials, $V\in S^{-2-\delta}$,
$\delta>0$, $0$ is not an accumulation point of the spectrum, and
under stronger restrictions on $l,l'$, in particular $l-l'>2$,
$(H-(\lambda\pm i0))^{-1}$ is uniformly bounded between the weighted
Sobolev spaces as $\lambda\to 0$ if there are no 0-energy bound states
($L^2$ nullspace of $H$) or half-bound states; the latter are elements
of the nullspace of $H$ on a larger space of distributions that will
be discussed later, see the discussion around \eqref{eq:P-zero-Fredholm}.
Such results go back to Jensen and Kato \cite{Jensen-Kato:Spectral};
the stated results are from recent works of Bony and H\"afner
\cite{Bony-Haefner:Low}, though potentials are not
considered there, just divergence form second order operators, and
Rodnianski and Tao \cite{Rodnianski-Tao:Effective} who do consider
such potentials and indeed more generally asymptotically conic
manifolds. We also refer to the work of M\"uller and Strohmaier
\cite{Muller-Strohmaier:Hahn} for a discussion of the analytic
continuation, near zero energy,
of the resolvent of the Laplacian on spatially compact perturbations
of cones using the theory of Hahn meromorphic
functions they develop.

We remark here that there are also high energy estimates under
assumptions on the geodesic flow, with the best case scenario if all
geodesics are backward and forward non-trapped, i.e.\ tend to infinity
in both directions, see
\cite{Robert-Tamura:Semiclassical,Gerard-Martinez:Principe,Gerard:Semiclassical-CPDE,Wang:Time-decay},
and also \cite{Vasy-Zworski:Semiclassical} for a general
asymptotically conic result.

It is natural to ask what kind of structure of Euclidean space is
involved in these results. One way to
address this is via constructing geometric generalizations. Another natural
question is whether one can make the function spaces more precise. For instance, can one
fit these estimates into a Fredholm (here typically invertible) statement? Such
frameworks are necessarily sharp in a sense: one can only
significantly (in not an essentially finite dimensional way) change the
domain if one changes the target space and vice versa.

Note that the estimates
above are 
{\em lossy}: the actual phenomenon is a version of real principal type 
propagation, with radial points (discussed below), here in terms of decay, so the 
difference between the two decay indices should be $1$ for the
positive energy estimates (i.e\ we have a loss of $\ep>0$) -- it is the 
radial points that prevent the optimal choice {\em if one works with 
  constant powers of $\langle z\rangle$ as the weights}.

In this paper we address these questions.
A natural geometric generalization is asymptotically
conic spaces.
A conic metric, with cross section a Riemannian manifold $(Y,h)$, is
the metric $g_0=dr^2+r^2 h$ on $\RR^+_r\times Y$; Euclidean space is a
cone over the sphere.
In local coordinates, such a metric is a linear combination, with
$\CI(Y)$-coefficients, of $dr$ and $r\,dy_j$, $y_j$ local coordinates
on $Y$, with some restrictions.

We want to generalize these coefficients to relax the warped product
structure of the conic metric.
Identifying a coordinate chart in $Y$ with a coordinate chart on the
sphere $\sphere^{n-1}$, we are working in an open conic subset $O$ of $\RR^n$
near infinity. (So this is a natural generalization of
the process of going from $\RR^n$ to say compact manifolds without boundary.)
From this perspective, below we replace the coefficients $\CI(Y)$ by
a {\em symbolic} statement, namely coefficients in $S^0(O)$, differing
from elements of $\CI(Y)$ by $S^{-\delta}(O)$, $\delta>0$; this
parallels the setting of Rodnianski and Tao \cite{Rodnianski-Tao:Effective}.

Let us rephrase this in a compactified notation introduced by Melrose
\cite{RBMSpec}, where `scattering', or sc-structures were discussed.
So let $x=\frac{1}{r}$, and add $x=0$ as an ideal boundary at infinity,
i.e.\ work with $[0,\infty)_x\times Y$. The metrics considered above
are then a linear
combination of symmetric tensors formed from $\frac{dx}{x^2}$ and $\frac{dy_j}{x}$.
Then being an element of
$\CI([0,\infty)_x\times Y)$ near the ideal boundary means exactly
being a
{\em classical symbol of order $0$} (on $\RR^n$ or under the
local conic identification), i.e.\ having an asymptotic expansion in non-positive
integer powers of $r$ (asymptotic in the symbolic sense), with the expansion being
just Taylor series at the boundary, i.e.\ an asymptotic expansion in non-negative integer
powers of $x$. A straightforward computation shows 
that if $Y$ is the sphere, i.e.\ we are discussing $\RR^n\setminus
\{0\}$, with an ideal boundary at infinity added, then linear combinations
of $\frac{dx}{x^2}$ and $\frac{dy_j}{x}$ with $\CI([0,\infty)_x\times
Y)$, resp.\ symbolic (on the complement of the origin), coefficients, are exactly the same as linear
combinations of the coordinate 1-forms $dz_j$ with the same kind of
coefficients; this property persists for general $Y$ locally on open subsets $O$ of
the kind considered above.

In Melrose's notation, if $X$ is a manifold with
boundary which near $\pa X$ is of the form $[0,x_0)_x\times Y$, one
would say that a metric with such coefficients, namely an {\em sc-metric}, is a $\CI$ section of
$\Tsc^* X\otimes_s\Tsc^* X$, with $\Tsc^* X$ being locally spanned
by $\frac{dx}{x^2}$ and $\frac{dy_j}{x}$. (See
Section~\ref{sec:2-micro} for additional discussion.)
In general $a\in S^m$ becomes $(x\pa_x)^\alpha \pa_y^\beta
a\in x^{-m}L^\infty$ from this perspective (cf.\ the earlier linear
vector field characterization); more invariantly this can be replaced
by $W_1\ldots W_k a\in x^{-m} L^\infty$ for all $k$ and
$W_j\in\Vb(X)$, where $\Vb(X)$ is the set of all {\em smooth vector
  fields tangent to $\pa X$}. (Thus, in the
non-compactified notation, locally and near $\pa X$, i.e.\ in an asymptotically
conic set near infinity in $\RR^n_z$, a spanning set for $\Vb(X)$ is given by linear vector
fields $z_i D_{z_j}$, $i,j=1,\ldots,n$; in the region where $|z_i|<C|z_n|$, $|z|>1$,
this means that one can use the local basis $z_n D_{z_j}$, $j=1,\ldots,n$, again with classical
symbols of order $0$, cf.\ above, as coefficients.) A symbolic section of
$\Tsc^* X^{\otimes_s^2}$ is denoted by
$S^m(X,\Tsc^*X^{\otimes^2_s})$.

So in general
let $X$ be an $n$-dimensional manifold with
boundary, $n\geq 3$, with $X^\circ$ equipped with a metric $g$ which
is asymptotically conic (near $\pa X$) in the following
precise sense. We consider $X$ equipped with a sc-metric $g$, i.e.\ a
symbolic of order $0$
section of $\Tsc^*X\otimes_s\Tsc^*X$ which at
$\pa X$ (i.e.\ `metric infinity') is conic:
$$
g-g_0\in S^{-\delta}(X;\Tsc^*X),\ g_0=x^{-4}\,dx^2+x^{-2} g_{\pa X},\ \delta>0,
$$
$g_{\pa X}$ a metric on $\pa X$,
i.e.\ $g$ is asymptotic to a conic metric $g_0$ on $(0,\infty)\times
\pa X$.
A special case of the operators we then consider is the spectral
family of the Laplacian plus a decaying potential,
$$
P(\sigma)=\Delta_g+V-\sigma^2,
$$
where $V\in S^{-2-\delta}(X)$,
and we are interested in the $\sigma\to 0$ limit. (For $\sigma$
bounded away from $0$, $V\in S^{-\delta}(X)$ real-valued suffices.)

Thus, $P(\sigma)$ is a family of (conormal/symbolic) scattering differential
operators (elements of $S^0\Diffsc(X)$), i.e.\ is a finite sum of products (possibly empty) of scattering vector fields $\Vsc(X)=x\Vb(X)$, with
local basis $x^2 D_x,xD_{y_j}$, dual to the scattering covectors
discussed above, with symbolic of order $0$ coefficients. (See
Section~\ref{sec:2-micro} for more detail.) In a non-compactified
notation, over conic subsets of $\RR^n_z$, this means that it is a linear combination
of $D_z^\alpha$ with coefficients which are symbols of order
$0$. In fact, modulo $S^{-\delta}$ the coefficients are equal to a
{\em classical symbol}. (This holds even for $V\in S^{-\delta}(X)$.) (The
notation $\Diffsc(X)$ is reserved for $\CI(X)$, i.e.\ classical
symbolic of order $0$, coefficients.) This scattering structure is
what is required for the non-zero $\sigma$ limiting absorption
principle considerations, which is the reason $V\in S^{-\delta}(X)$ is
acceptable for that purpose.

However, our $P(\sigma)$ has a stronger structure. Recall that $\Diffb(X)$
is the differential operator algebra generated by $\Vb(X)$ with
$\CI(X)$ coefficients; $S^m\Diffb(X)$ means that the coefficients are
symbols of order $m$. Then
$$
\Delta_g\in x^2 S^0\Diffb^2(X)=S^{-2}\Diffb^2(X),\ V\in
S^{-2-\delta}\Diffb^0(X),
$$
so
$$
P(\sigma)+\sigma^2\in S^{-2}\Diffb^2(X),
$$
with a classical symbolic leading term. This stronger structure plays
a key role in the zero energy limit considered in this paper.

More generally, though the results are interesting even without this
extension, with $\delta>0$, we
consider families of the form
\begin{equation}\label{eq:P-sigma-form}
P(\sigma)=P(0)+\sigma Q-\sigma^2,\qquad P(0)\in x^2\Diffb^2(X),\ Q\in
S^{-2-\delta}\Diffb^1(X)
\end{equation}
and $P(0)$ elliptic, $P(\sigma)$ is symmetric for $\sigma$ real, and
$P(0)-\Delta_g\in S^{-2-\delta}\Diffb^2(X)$, with the b-differential operators
discussed below.
Such a generalization is useful for Kerr-type spaces, and it does not
affect any arguments. For actual Kerr spacetimes, near the event horizon we lose
ellipticity as well, but it is straightforward to modify arguments
using the Kerr-de Sitter results from
\cite{Vasy-Dyatlov:Microlocal-Kerr}, and we do that here in the penultimate
Section~\ref{sec:Kerr}. (We explicitly discuss the asymptotic form
of the Kerr wave operator in that section.) {\em Indeed, $Q$ can be
replaced by any family of operators $Q(\sigma)$ of the same class, depending smoothly on
$\sigma$, for all considerations below.} Note that there is extensive
literature on related aspects of the Kerr setting, including
\cite{Dafermos-Rodnianski:Axi, Dafermos-Rodnianski-Shlapentokh:Decay,
  Shlapentokh-Rothman:Quantitative,
  Dafermos-Holzegel-Rodnianski:Linear-Schwarzschild,Andersson-Blue:Uniform,Hung-Keller-Wang:Linear,Tataru:Local,
  Tataru-Tohaneanu:Local} and many other papers, most of which is for the actual
spacetimes, rather than the Fourier transformed problem; here our
interest is in a robust treatment of a general class of such
operators, so we restrict to brief remarks for the changes necessary
in Section~\ref{sec:Kerr}, and we refer to the introduction of \cite{Hintz-Vasy:Stability}
for a more extensive discussion.

\begin{figure}[ht]
\begin{center}
\includegraphics[width=75mm]{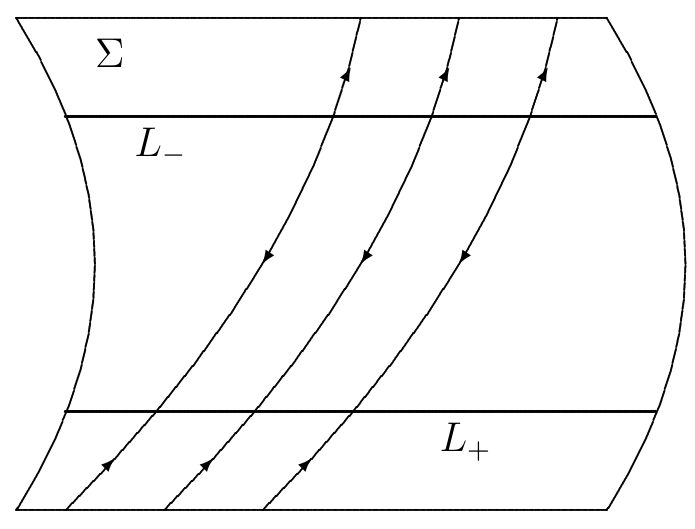}
\end{center}
\caption{The characteristic set $\Sigma$, here a $\TT^2$ (with
  opposite sides on the figure identified), and source/sink manifolds
  $L_-$, resp. $L_+$, of the Hamilton flow.}
\label{fig:char-set-flow}
\end{figure}

We are then interested in limiting absorption resolvent estimates along the spectrum (it
is straightforward to add an imaginary part of the correct sign,
corresponding to $\pm i0$) as the spectral
parameter $\sigma$ tends to $0$. We first recall the basic limiting
absorption principle, for $\sigma\neq 0$, proved in this geometric setting by Melrose
\cite{RBMSpec}, generalizing \eqref{eq:Euclidean-LAP}, but in a {\em
  different, Fredholm setting}, in which it was shown in
\cite[Section~5.4.8]{Vasy:Minicourse}, see also \cite[Section~4]{Vasy:Positive}. This setting involves
variable order (or anisotropic) scattering Sobolev spaces; such spaces in the standard microlocal setting have
been used by Unterberger \cite{Unterberger:Resolution} and Duistermaat
\cite{Duistermaat:Carleman}, but here it is the {\em weight}, i.e.\
the {\em decay order} that is variable, corresponding to the operator $P(\sigma)$
being elliptic in the standard, differential order, sense, thus the
differential order of the Sobolev space playing no role. (We recall
that such variable order Sobolev spaces played a key role in recent
advances in dynamical systems, see e.g.\ \cite{Faure-Sjostrand:Upper,Dyatlov-Zworski:Dynamical}.) Via the
identification of open subsets of $X$
near $\pa X$ with asymptotically conic subsets of $\RR^n$, such spaces (locally)
reduce to variable order Sobolev spaces on $\RR^n$, where again it is
the decay order that varies. (Via the Fourier transform, however,
microlocally these can be reduced to spaces with variable differential order!)
These Sobolev spaces are defined using pseudodifferential operators of
variable order. Now, in regions where the operator is microlocally
elliptic, the order of the Sobolev space is immaterial (all orders
work equally well for Fredholm theory); only at the characteristic
set, which in this case is $G-\sigma^2=0$, $G$ the dual metric
function on $\Tsc^*X$, is it relevant. Within the characteristic set,
the Hamilton flow tends to, i.e.\ the $H_G$ integral curves tend to,
sink/source manifolds ($L_-$, resp.\ $L_+$, see Figure~\ref{fig:char-set-flow}), called {\em radial
  sets}, in the forward/backward direction. The variable order is a
function $r$ that needs to be monotone along the $H_G$-flow
(corresponding to propagation of singularities, or rather of regularity),
and it needs to be greater than, resp.\ less than, a threshold value,
here $-1/2$, at one, resp.\ the other, radial set. The choice is thus
whether the decay order is $>-1/2$ at the source manifold or the sink
manifold; these two possibilities correspond to the incoming/outgoing limiting resolvents. The Fredholm
statement is that for any $s$ and for variable order $r$ satisfying
these requirements
$$
P(\sigma): \{u\in \Hsc^{s+2,r}: P(\sigma)u\in \Hsc^{s,r+1}\}\to \Hsc^{s,r+1}
$$
is Fredholm, indeed in this case invertible.
Notice that the loss of $\ep>0$ in the usual way the limiting absorption 
principle is stated, cf.\ \eqref{eq:Euclidean-LAP}, {\em disappears} with these variable order spaces 
($r+1$ vs. $r$, exactly the real principal type numerology). (We
remark that, with essentially the same proof, one also has lossless high energy
estimates on large parameter, or semiclassical, scattering Sobolev spaces.)

Concretely, with $\tau,\mu$ the fiber coordinates, in the sc-notation,
on $\Tsc^* X$, i.e.\ writing covectors as
$\tau\,\frac{dx}{x^2}+\mu\cdot\frac{dy}{x}$, one can take for $\beta>0$
$$
r=-\frac{1}{2}\pm\beta\frac{\tau}{\sqrt{\tau^2+|\mu|^2_{g_{\pa X}}}},
$$
which in the Euclidean case just means
$$
r=-\frac{1}{2}\mp \beta\frac{z\cdot\zeta}{|z||\zeta|},
$$
as the variable order.

These Fredholm estimates, i.e.\ the propagation of singularities
estimates, including radial points, are proved by positive commutator
estimates, which become degenerate (quadratic vanishing of symbol) at
0 energy. However, it turns out that phrasing these estimates as b-estimates one
can make them uniform as $\sigma\to 0$.

We already introduced the b-differential operators, $\Diffb(X)$; there
is a pseudodifferential algebra $\Psib(X)$ microlocalizing it,
originating in Melrose's work, see \cite{Melrose:Atiyah} for a
detailed treatment, and
recalled in Section~\ref{sec:background}. The b-Sobolev spaces are then
based on these operators; recalling that the symbolic estimates amount
to iterative regularity with respect to elements of $\Vb(X)$ relative
to $L^\infty$, we see that these Sobolev spaces are thus essentially a finite
regularity version of $L^2$-based (rather than $L^\infty$-based)
symbol classes. As an aside, b-structures
are typically associated with `cylindrical ends' in the Riemannian
geometry literature, but as we discussed $P(0)$ can be considered as $x^2$
times an unweighted second order b-differential operator, i.e.\
`conformally unweighted b'.

Let $\Hb^{\tilde r,l}$ be the b-Sobolev space of differential order
$\tilde r$, considered as a function on $\Sb^*X=(\Tb^*X\setminus
o_\bl)/\RR^+$ (equivalently, a homogeneous degree zero function on $\Tb^*X\setminus
o_\bl$, $o_\bl$ the zero section), and weight $l$, {\em
  relative to the scattering $L^2$-space $L^2=L^2_{\scl}$}; this is a notion
we recall and explain presently. Thus, in this b-setting {\em the
  differential order, not the decay order, is variable}, unlike in the sc-setting.
If $\tilde r$ is a positive integer $m$, 
then
$$
\Hb^{\tilde r,l}=\{u\in \Hb^{0,l}=x^l L^2:\ k\leq
m,\ W_1,\ldots,W_k\in\Vb(X)\Rightarrow  W_1\ldots W_k u\in
x^l L^2\}.
$$
In local coordinates, the regularity statement is simply that
$$
j+|\alpha|\leq
m\Rightarrow  (xD_x)^j D_y^\alpha u\in
x^l L^2.
$$
Notice that, for suitable values of $\tilde r$ (constant) and $l$, the
$j=m$ control, i.e.\ $(xD_x)^m u\in x^lL^2$, can be used to estimate $u$ in
$x^l L^2$ by Hardy-type inequalities, so in some cases these b-Sobolev
spaces are, in case of compactified $\RR^n$, the usual (weighted)
homogeneous Sobolev spaces.
For negative integers $\tilde r$, the spaces can be defined either via 
duality or using elliptic b-pseudodifferential operators. 
In general, for a function $\tilde r$, the Sobolev space $\Hb^{\tilde r,l}$ is defined by taking
$-N<\inf\tilde r$, and $\Lambda\in\Psib^{\tilde r,l}$ elliptic, e.g.\
with principal symbol $x^{-l}(\taub^2+|\mub|^2)^{\tilde r/2}$,
$|\mub|$ e.g.\ the length with respect to $g_{\pa X}$ (but all choices
are equivalent),
$$
\Hb^{\tilde r,l}=\{u\in \Hb^{-N,l}:\ \Lambda u\in L^2\}.
$$

Melrose \cite{Melrose:Atiyah}, see also Guillarmou and Hassell \cite[Theorem~2.1]{Guillarmou-Hassell:Resolvent-I} for
explicitly this statement when $\tilde r$ is constant, showed that
\begin{equation}\label{eq:P-zero-Fredholm}
P(0):\Hb^{\tilde r,l}\to\Hb^{\tilde r-2,l+2}
\end{equation}
is Fredholm of index $0$ when
$|l+1|<\frac{n-2}{2}$; cf.\ Section~5.6 of \cite{Vasy:Minicourse} for
a general $\tilde r$ statement, and see the comments below. Note that $l=-1$ corresponds to the weights of
the domain and target spaces being symmetric relative to $L^2$ (in the
sense of being $L^2$-duals of each other, considering just these
weights, not the differential orders). By the
ellipticity of $P(0)$, $\tilde r$ is actually irrelevant for the
nullspace considerations (i.e.\ the nullspace is independent of this choice),
and, by indicial root considerations, in the range
$|l+1|<\frac{n-2}{2}$ invertibility is also independent of $l$, i.e.\
when $P(0)$ is considered as such an operator, for all $\ep>0$,
\begin{equation}\label{eq:Ker-P0-sp}
\Ker P(0)\subset\Hb^{\infty,-1+\frac{n-2}{2}-\ep}=\Hb^{\infty,\frac{n-4}{2}-\ep}.
\end{equation}

Note
that when $n\leq 4$, this Fredholm, of index 0, range excludes $l=0$, thus the choice of the standard
$L^2$ space, so for our purposes it is {\em not} just the $L^2$ nullspace that
matters, rather the nullspace on a bigger function space,
$\Hb^{\infty,\frac{n-4}{2}-\ep}$, where $\ep>0$ can be taken
arbitrarily small (and should be $<n-2$). Non-$L^2$ elements of this
nullspace, if they exist, are called
{\em half bound states}; $L^2$ elements are called {\em bound states}. On
the other hand, for $n>4$, necessarily any such zero eigenvalue is an
$L^2$-eigenvalue. Also note that if $P(0)=\Delta_g$, then the maximum
principle implies that the nullspace is trivial for all $n\geq 3$
since any element $u$ of this nullspace is in a weighted, infinite order
differentiability, b-Sobolev space that is {\em decaying} relative to
$L^2_{\bl}$ (as $L^2=x^{n/2}L^2_\bl$, $L^2_\bl$ being the $L^2$ space
relative to a {\em b-density}, e.g.\ in local coordinates
$\frac{dx}{x}\,dy$, and we have membership of $u$ in
$x^{(n-4)/2-\ep}L^2$ for all $\ep>0$), and thus decaying at infinity in the $\sup$ sense by
Sobolev embedding. (There are also direct $L^2$-based arguments available using
integration by parts, which needs to be justified, but which can be done.)

In order to proceed, write the b-dual variable of $x$ as
$\taub$, that of $y_j$ as $(\mub)_j$, so covectors are written as
$$
\taub\,\frac{dx}{x}+\sum_j(\mub)_j\,dy_j.
$$
Then a typical symbol of a b-pseudodifferential operator of order
$\tilde r$, where $\tilde r$ may be a function on $\Sb^*X$, is
$(\taub^2+|\mub|^2)^{\tilde r/2}$, with $|\mub|$ being the length with
respect to $g_{\pa X}$.
A typical
example of a variable differential order that we consider is
\begin{equation*}
\tilde r=\frac{1}{2}-(l+1)\pm\beta\frac{\taub}{(\taub^2+|\mub|^2)^{1/2}},
\end{equation*}
where $\beta>0$, and $\pm$ gives rather different function spaces that
will correspond to the $\pm i0$ limit of the resolvent.

We now
explain the connection to the scattering spaces. The symbol of a
representative (for our purposes) elliptic b-pseudodifferential
operator in $\Psib^{\tilde r,l}$, which maps $\Hb^{\tilde r,l}$ to $L^2$,
is $x^{-l}(\taub^2+|\mub|^2)^{\tilde r/2}$. Now,
\begin{equation}\label{eq:b-to-sc-conv}
\taub\,\frac{dx}{x}+\sum_j(\mub)_j\,dy_j=(x\taub)\,
\frac{dx}{x^2}+\sum_j( x\mub)_j\,\frac{dy_j}{x},
\end{equation}
i.e.\ the $\Tsc^*X$ fiber coordinates are $\tau=x\taub$, $\mu=x\mub$,
shows that in terms of the scattering cotangent bundle coordinates,
this symbol is
$$
x^{-l}(\taub^2+|\mub|^2)^{\tilde r/2}=x^{-l-\tilde r} (\tau^2+|\mu|^2)^{\tilde r/2},
$$
which is, away from the 0-section (where it is singular) the symbol of
an elliptic sc-differential operator in $\Psisc^{\tilde r,l+\tilde
  r}$. Recalling that, due to ellipticity, the sc-differential order is
irrelevant for the limiting absorption principle, we see that
microlocally near the sc-characteristic set, for $\sigma\neq 0$ (so
the characteristic set is away from the zero section), these Sobolev
spaces correspond to $\Hsc^{*,r}$,
$$
r=l+\tilde r=-\frac{1}{2}\pm\beta\frac{\tau}{(\tau^2+|\mu|^2)^{1/2}},
$$
in agreement with the orders that appeared in the limiting absorption
principle above!

Now, for $\sigma\neq 0$, $P(\sigma)$ does not map
$\Hb^{\tilde r,l}\to\Hb^{\tilde r-2,l+2}$. We shall instead consider
$P(\sigma)$ as a map
$$
\{u\in\Hb^{\tilde r,l}:\ P(\sigma)u\in \Hb^{\tilde r-1,l+2}\}=\cX=\cX_\sigma\to\Hb^{\tilde r-1,l+2},
$$
with $\cX$ equipped with the squared norm
$$
\|u\|^2_{\cX}=\|u\|^2_{\Hb^{\tilde r,l}}+\|P(\sigma)u\|^2_{\Hb^{\tilde r-1,l+2}},
$$
which makes it into a Hilbert space. While this might be a somewhat
strange space, the space $\dCI(X)$ of $\CI$ functions vanishing to infinite
order at $\pa X$ (i.e.\ Schwartz functions) is dense in it, as
follows by approximating $u\in\cX$ via $u_\ep=\Lambda_\ep u$,
$\Lambda_\ep\in\Psib^{-\infty,-\infty}(X)$, $\ep\to 0$, where
$\Lambda_\ep$ is uniformly bounded in
$\Psib^{0,0}(X)$, converging to $\Id$ in $\Psib^{\delta',\delta'}(X)$
for $\delta'>0$; then
$u_\ep\to u$ in $\Hb^{\tilde r,l}$ and moreover
$$
P(\sigma)u_\ep-P(\sigma)u=P(\sigma)(\Lambda_\ep-\Id)
u=[P(\sigma),\Lambda_\ep] u+(\Lambda_\ep-\Id)P(\sigma)u\to 0
$$
in $\Hb^{\tilde r-1,l+2}$ as is immediate in case of the second term
(for $P(\sigma)u\in \Hb^{\tilde r-1,l+2}$) and in the case of the
first term as $[P(\sigma),\Lambda_\ep]=[P(0)+\sigma Q,\Lambda_\ep]$ is
uniformly bounded in $\Psib^{1,-2}$, converging to $0$ in
$\Psib^{1+\delta',-2+\delta'}$ for $\delta'>0$. This observation uses that
$\sigma^2$ commutes with every operator, thus with $\Lambda_\ep$,
which fact will also play a key role below.

Our main theorem concerns a uniform estimate about $P(\sigma)^{-1}$ in
the above sense as $\sigma\to 0$. While $\cX=\cX_\sigma$ depends on
$\sigma$, the interesting part of the bound on $P(\sigma)^{-1}f$ is
its $\Hb^{\tilde r,l}$ norm (after all, $P(\sigma)$ applied to this is
simply $f$); this is how we state the bound below.

\begin{thm}[See Theorem~\ref{thm:main-improved} for the general
  allowed $\tilde r$]\label{thm:main}
Suppose that $|l+1|<\frac{n-2}{2}$, and suppose that
$P(0):\Hb^{\infty,l}\to\Hb^{\infty,l+2}$ has trivial nullspace, an
assumption independent of $l$ in this range; see \eqref{eq:Ker-P0-sp}.

There exist variable order b-Sobolev spaces $\Hb^{\tilde r,l}$ and
$\sigma_0>0$ such
that
$$
P(\sigma):\{u\in\Hb^{\tilde r,l}:\ P(\sigma)u\in \Hb^{\tilde r-1,l+2}\}\to\Hb^{\tilde r-1,l+2}
$$
is invertible for $0\leq\sigma\leq\sigma_0$, with this inverse being the $\pm i0$ resolvent of
$P(\sigma)$ (with choice corresponding to that of $\tilde r$), and the
norm of $P(\sigma)^{-1}$ as an element of $\cL(\Hb^{\tilde
  r-1,l+2},\Hb^{\tilde r,l})$ is uniformly bounded in $[0,\sigma_0]$.

For instance, one can take any $\beta>0$,
\begin{equation}\label{eq:b-symbol-weight-choice-0}
\tilde r=\hat r_\pm(\beta)=\frac{1}{2}-(l+1)\pm\beta\frac{\taub}{(\taub^2+|\mub|^2)^{1/2}},
\end{equation}
where $\taub,\mub$ are the b-dual variables defined above, $|.|$
the dual metric function of $g_{\pa X}$, and $\pm$
corresponds to the $\pm i0$ limits.

Indeed, the estimates hold in $|\sigma|\leq \sigma_0$, $\im(\sigma^2)\geq
0$ ($+$ case), resp.\ $|\sigma|\leq \sigma_0$, $\im(\sigma^2)\leq
0$ ($-$ case).
\end{thm}

\begin{rem}
For the general orders $\tilde r$, see Theorem~\ref{thm:main-improved}.
\end{rem}

\begin{rem}
Notice that this theorem is lossless in terms of decay orders: for
$\sigma=0$ the condition $P(\sigma)u\in \Hb^{\tilde r-1,l+2}$ together
with $u\in\Hb^{\tilde r,l}$ implies $u\in\Hb^{\tilde r+1,l}$, and $P(0):\Hb^{\tilde r+1,l}\to \Hb^{\tilde r-1,l+2}$ is invertible. Thus,
in this sense this is a very precise estimate. Also, the loss of one
differentiability order relative to elliptic estimates is a standard
real principal type/radial point phenomenon; while $P(0)$ is elliptic
in $\Psib^{2,-2}$, $P(\sigma)$ is not so for $\sigma\neq 0$ in
$\Psib^{2,0}$. For such $\sigma$, $P(\sigma)$ has real principal
type/radial point estimates, mentioned above, due to a
non-trivial characteristic set disjoint from the zero section of the
$\Tsc^*X$. This amounts to an
infinite b-frequency phenomenon in view of \eqref{eq:b-to-sc-conv}, thus a b-differentiable order loss in
terms of the $\Hb$ spaces. See Section~\ref{sec:2-micro} for the
{\em second microlocal scattering} version.
\end{rem}

\begin{rem}
In fact, $\sigma_0>0$ could be arbitrary if $P(\sigma)$ had no
$L^2$-elements of its nullspace for any $\sigma$; this is a potential issue for
$\sigma^2<0$ (thus not an issue for $\sigma$ real). The statement with arbitrary $\sigma_0$, under this
additional assumption, follows easily from this theorem, and the
limiting absorption principle for $\sigma^2>0$, resp.\ the elliptic
theory for $\sigma^2<0$; or indeed it is a direct consequence of the
proof of the Theorem.
\end{rem}

\begin{rem}
We also have a version of the theorem with non-trivial nullspace. The result is then highly 
dimension dependent with the case $n\geq 5$ amounting to more or less 
standard Fredholm perturbation theory. We take this up in the final 
section of this paper, mostly concentrating on the most delicate $n=3$
case; here in the introduction we merely refer to Theorems~\ref{thm:asymp-Eucl-resonances} and \ref{thm:Kerr-resonances}. 
\end{rem}

\begin{rem}\label{rem:bd-state-stable}
By standard Fredholm perturbation theory, if $P_b(\sigma)$ is a family
of operators of the form discussed above for $P(\sigma)$ continuously
depending on a parameter $b\in\RR^k$ (i.e.\ $g$ depends continuously
on $b$ in the class of asymptotically conic metrics, etc.), and $P_0(0)$ is invertible, then
there is $\ep>0$ such that
the same holds for $P_b(0)$ for $|b|<\ep$, i.e.\ the hypothesis of the
theorem is perturbation stable.
\end{rem}

\begin{rem}\label{rem:bundles}
The theorem is stated for operators acting on scalar functions. It
also applies, with essentially the same proofs, if
$P(\sigma)=P(0)+\sigma Q-\sigma^2$ is
acting on sections of a vector bundle when
$x^{-(n+2)/2}P(0)x^{(n-2)/2}$ (cf.\ \eqref{eq:P-sigma-form} and
\eqref{eq:b-Lap}) is elliptic with positive scalar principal symbol and modulo $S^{-2-\delta}\Diffb^2$ (acting on sections
  of the bundle) is of the form
\begin{equation}\label{eq:tensor-op-form}
(xD_x)^2\otimes\Id+\tilde P_{\pa X}+\Big(\frac{n-2}{2}\Big)^2
\end{equation}
with $\tilde P_{\pa X}$ being an elliptic differential operator on $\pa X$, which
is non-negative, self-adjoint with respect to a Hermitian inner
product, and $Q$ is of the class $S^{-2-\delta}\Diffb^1$, acting
on sections of the bundle. Moreover, if $\frac{n-2}{2}$ is replaced by
another constant, the same constant can be used in the statement of
Theorem~\ref{thm:main} in place of $\frac{n-2}{2}$.

Indeed, allowing $Q=Q_0+Q_1$, $Q_0\in x^2\Diffb^1$ and
$Q_1\in S^{-2-\delta}\Diffb^1$ (so the coefficients of $Q$ do
not decay relative to those of $x^{-(n+2)/2}P(0)x^{(n-2)/2}$) leaves
the theorem valid, again with the proofs requiring only minor
modification. While such
a $Q$ affects the (skew-adjoint part of the extended) b-subprincipal symbol at the boundary, hence the
argument of Section~\ref{sec:symbolic}, see the term
$(P(\sigma)^*-P(\sigma))A^*A$ in \eqref{eq:commutator-basic-b-est}, due to the prefactor $\sigma$,
the subprincipal symbol can be made arbitrarily small (depending on
the desired $l$ and $\tilde r$) by choosing $\sigma_0$ small, and then
(a small constant times)
the term in $a\sH_p a$ corresponding to the first term of the right
hand side of either of \eqref{eq:H-p-a-expand} and
\eqref{eq:H-p-a-expand-mod} dominates the new term in the principal
symbol of $(P(\sigma)^*-P(\sigma))A^*A$, so the rest of the argument is
unaffected. Moreover, $Q_0$ now appears in the normal operator
argument, but for the same reason as in the above symbolic argument, the
principal symbol of the
normal family has uniform estimates (in $\taub$) of the same kind as
if $Q_0$ vanished, which in particular gives the positivity/negativity
of \eqref{eq:normal-op-computation} for large $\taub$ on the Mellin
transform side, and then finally the positivity/negativity of \eqref{eq:normal-op-computation} for
finite $\taub$ follows from the vanishing $Q_0$ case by the stability
of the property of being positive/negative.

Indeed, an analogous argument shows that the conclusion of the theorem
also holds if
$x^{-(n+2)/2}P(0)x^{(n-2)/2}$ differs from
\eqref{eq:tensor-op-form} by an operator $R=R_0+R_1$,
$R_0\in x^2\Diffb^1$, $R_1\in S^{-2-\delta}\Diffb^2$, without formal
self-adjointness assumptions on $R_0$ but instead assuming that $R_0$
is small in a suitable seminorm on $x^2\Diffb^1$. (Smallness needs to
be assumed since there is no small parameter in front of $P(0)$, unlike $\sigma$ in the case
of $Q$.)
\end{rem}

\begin{rem}
It is straightforward to add an appropriate conormality statement
(conormal to $\sigma=0$) to
this using the commutation relation of $xD_x+\sigma D_\sigma$ with
$P(\sigma)$. Notice that this is not quite a statement of
$P(\sigma)^{-1}$ being conormal to $0$ in weighted b-Sobolev spaces,
rather a `twisted version' of it.
\end{rem}

\begin{rem}
This result can be extended to complex scaled operators, see \cite{Sjostrand-Zworski:Complex}. Then it gives
a similar uniform bound for the complex scaled resolvent down to $0$,
again under the assumption of $P(0)$ having trivial nullspace as above.
\end{rem}

Since one can take $\beta>0$ arbitrarily small, this in particular
implies a (generalization of a) result of Bony and H\"afner
\cite{Bony-Haefner:Low}, see the end of the paper for the
conversion to the scattering, thus on Euclidean space standard
weighted Sobolev, function spaces:

\begin{cor}\label{cor:2-micro-spaces}
On long-range asymptotically Euclidean spaces, or more generally on
Riemannian scattering spaces, of dimension $\geq 3$, the resolvent of
the Laplacian satisfies that
$$
R(\sigma^2\pm i0):\Hb^{1/2-(l+1)-1+\beta,l+2}\to\Hb^{1/2-(l+1)-\beta,l}
$$
is uniformly bounded (as a function of $\sigma$) on $[0,\sigma_0]$ for
$|l+1|<\frac{n-2}{2}$ and $\beta>0$, and thus for all $s\in\RR$ and
for all $\beta>0$,
$$
R(\sigma^2\pm i0):\Hsc^{s,1+\beta}\to\Hsc^{s+2,-1-\beta}
$$
is uniformly bounded in the same sense, and more precisely, for the
second microlocal spaces of Section~\ref{sec:2-micro}, for $\beta>0$,
$s\in\RR$, $|l+1|<\frac{n-2}{2}$,
$$
R(\sigma^2\pm i0):\Hsc^{s,1/2+\beta,l+2}\to\Hsc^{s+2,-1/2-\beta,l}.
$$
\end{cor}

In fact, in Section~\ref{sec:2-micro} we show an even sharper version,
with variable order weights, in Theorem~\ref{thm:2-micro-spaces}.

\begin{rem}
Other closely related works, which we only briefly mention here,
include those of Bony and H\"afner
\cite{Bony-Haefner:Decay, Bony-Haefner:Improved} and of Vasy and
Wunsch \cite{Vasy-Wunsch:Positive}.
\end{rem}

The basic idea of the proof of our main theorem is a positive (or negative)
(twisted) commutator estimate in the b-pseudodifferential
setting. This is a positivity result modulo compact operators, whose
proof requires two ingredients. First, a positive principal symbol,
capturing the b-differential order part of compactness;
this has very much the flavor of the usual radial point estimate
results on variable order Sobolev spaces, and is discussed in
Section~\ref{sec:symbolic}. Second, a positive normal operator,
capturing the decay part of compactness. In general, this is a
computation in a
non-commutative algebra, namely an operator algebra on $\pa
X$. However, one can arrange that the only operators involved are
$\Delta_{\pa X}$, $x$ and $xD_x$, and $x$ only enters via overall
multiplication by its powers,
which means that in the commutator computation the effect will be a
shift of $xD_x$ by an imaginary constant thanks to a conjugation. Thus, finally,
on the Mellin transform side, using the spectral representation of
$\Delta_{\pa X}$, one has to check the positivity of a
{\em scalar function}, i.e.\ positivity in a commutative
algebra. While this may be somewhat involved, it is in principle
straightforward; we do this in Section~\ref{sec:normal}. This
computation is
simplified by a large parameter symbolic treatment (akin to
semiclassical rescaling), which is how we proceed in
Section~\ref{sec:normal}. We remark that as we prove the general
differential order $\tilde r$ version of the main theorem in
Theorem~\ref{thm:main-improved}, but we prove (and use) the normal operator
positivity only for the special differential order,
\eqref{eq:b-symbol-weight-choice-0}, there is a slightly involved
functional analytic
argument in the first part of Section~\ref{sec:normal} that could be
simplified if only the orders \eqref{eq:b-symbol-weight-choice-0} were considered.

While this is a relatively standard approach, what makes it somewhat
unusual in this case is that $P(\sigma)$ is regarded as effectively an
element of $x^2\Diffb^2(X)\subset\Psib^{2,-2}(X)$, {\em even though the
spectral parameter does not lie in this space!} However, in the
(twisted) positive commutator estimate, it either disappears (since it
is a multiple of the identity, thus commutes with every operator), or
gives the correct sign (in the `twisted term', involving
$P(\sigma)-P(\sigma)^*$) when $\im(\sigma^2)$ is non-zero but has the
correct sign.

In
Section~\ref{sec:2-micro}, we connect the results to scattering Sobolev
spaces, which are the natural spaces for the limiting resolvents for
positive spectral parameters. The `interpolation' between the
scattering and b- Sobolev spaces is given by second microlocalized
scattering Sobolev spaces, with the second microlocalization taking
place at the zero section; this is what brings the `b-picture' into
the scattering problem. We emphasize that this part is {\em not
  necessary} for proving Theorem~\ref{thm:main}; rather it provides a
single framework for understanding Theorem~\ref{thm:main} and the more
common positive energy limiting absorption results, and thus proves
Corollary~\ref{cor:2-micro-spaces} along the way in a strengthened
form in Theorem~\ref{thm:2-micro-spaces}.

Section~\ref{sec:Kerr} notes the changes that are necessary to
adapt these arguments to Kerr-like spaces; the point being the lack of
ellipticity in the interior of $X$.
Finally, in Section~\ref{sec:resonances} we discuss the estimates
obtained when there are zero resonances both in the asymptotically
Euclidean and in the Kerr settings.

The author is very grateful to Kiril Datchev, Dietrich H\"afner, Rafe Mazzeo, Richard Melrose, Maciej
Zworski and especially Peter Hintz and Jared Wunsch for fruitful discussions.

\section{Background}\label{sec:background}
In this background section we discuss b-pseudodifferential
operators. These are treated in great detail in Melrose's book
\cite{Melrose:Atiyah} by describing their Schwartz kernels on a
resolved double space. Here we follow \cite[Section~5.6]{Vasy:Minicourse} by
reducing their study to H\"ormander's uniform class in local coordinates,
with some modifications.

Thus, near a point on $\pa X$, consider local coordinates $(x,y)$,
$x\geq 0$ near $0$, $y$ in an open
subset of $\RR^{n-1}$; we in fact take all symbols to be compactly
supported in the chart. In general, if one
introduces logarithmic coordinates, $t=-\log x$, b-pseudodifferential
operators are just H\"ormander's uniform pseudodifferential class
locally in the corresponding cylindrical regions (positive real line in $t$ times a
compact set in the $y$ variables). Thus, to define $\Psibc^{m,0}(X)$, locally one considers, with $\psi$ compactly
supported, identically $1$ near $0$, 
operators of the form
\begin{equation}\label{eq:b-symbolic-term}
Bu(t,y)=(2\pi)^{-n}\int e^{i[-(t-t')\taub+(y-y')\cdot\mub]}\psi(t-t')b(t,y,\taub,\mub)\,u(t',y')\,dt'\,dy',
\end{equation}
where $b\in S^m_\infty$, i.e.
$$
|\pa_t^k\pa_y^\alpha\pa_{\taub}^j\pa_{\mub}^\beta b|\leq C_{k\alpha j\beta}\langle(\taub,\mub)\rangle^{m-j-|\beta|}.
$$
(The minus sign in front of $(t-t')$ in the phase is added for 
consistency with the compactified notation below.) 
The localizer $\psi$ plays an important role; without it the Schwartz
kernel would only decay polynomially in $|t-t'|$, and to work on
Sobolev spaces with polynomial gain in $x$ one needs Schwartz kernels
decaying faster than $e^{-M|t-t'|}$ for all $M$. Notice that $\pa_t$
can be replaced by $x\pa_x$ in this definition, thus this is {\em exactly} the
symbol space on $\Tb^*X$ as discussed in Section~\ref{sec:2-micro}
ahead of Lemma~\ref{lemma:symbol-space-blowup}.

Notice that in terms of $x$, one could equivalently consider
$$
\tilde Bu(x,y)=(2\pi)^{-n}\int
e^{i[\frac{x-x'}{x}\taub+(y-y')\cdot\mub]}\tilde \psi\Big(\frac{x-x'}{x}\Big)\tilde b(x,y,\taub,\mub)\,u(x',y')\,dt'\,dy',
$$
where $\tilde\psi$ is supported in $(-1/2,1/2)$, say, and identically $1$
near $0$, and where the estimates on $\tilde b$ are those on $b$ pulled back via the
map $(x,y,\taub,\mub)\mapsto(-\log x,y,\taub,\mub)$, so
$\pa_t=-x\pa_x$, and this is simply the conormal estimate for $\tilde
b$. Indeed, $\frac{x-x'}{x}=1-e^{t-t'}$, so the support condition on
$\tilde\psi$ is equivalent to a support condition on $\psi$ stated
above, and similarly in the phase function
$\frac{x-x'}{x}=1-e^{t-t'}$is equivalent to $-(t-t')$ above, in the strong
sense that at the critical set, $t=t'$, the differentials are the same.

Then one adds to
these residual in the symbolic, but not in the decay, sense terms $R$
satisfying estimates on their Schwartz kernel $K_R$
$$
|\pa_t^k\pa_y^\alpha\pa_{t'}^l\pa_{y'}^\gamma
K_R(t,y,t',y')|\leq C_{k\alpha l\gamma
  NM} \langle y-y'\rangle^{-N}e^{-M|t-t'|}
$$
for all $l,k,\alpha,\gamma,M,N$; these are elements of the
H\"ormander class $\Psi^{-\infty}_\infty$, but we impose stronger,
exponential, decay in $|t-t'|$. The full local version of
$\Psibc^{m,0}$ then consists of operators $A$ of the form $A=B+R$. One can then transplant these to a
manifold with boundary via local coordinates, much as in the standard
setting; we refer to \cite[Section~5.6]{Vasy:Minicourse} for a
detailed discussion.

We also define
$\Psibc^{m,l}(X)=x^{-l}\Psibc^{m,0}(X)$. Note that this is the
opposite, in terms of the sign of $l$, of Melrose's order convention,
but it is helpful as the space of operators gets larger with
increasing $m$ as well as with increasing $l$.

Recall that in the standard sense, namely considering H\"ormander's
uniform class with symbolic behavior in the dual variables, a symbol $b$
of order $m$
is called {\em classical} if it has an asymptotic expansion in
homogeneous (with respect to dilations in $(\taub,\mub)$)
symbols of order $m-j$ ($j\in\NN$), in the sense
that the difference between the terms of the expansion with $j\leq k-1$ and
$b$ is in $S^{m-k}_\infty$. This can also be defined instead as a
classical conormality statement when the fibers of $\Tb^*X$ (which are
vector spaces) are
radially compactified (as discussed in the introduction for Euclidean
space; the fibers are such). If $m=0$
classicality is thus a uniform version (in the base variables $(t,y)$)
of smoothness on $\overline{\Tb^*}_{X^\circ}X$. The general case
reduces to this after factoring out the $-m$th power of a defining function of the new
boundary, {\em fiber infinity}; one can take this defining function
(away from the zero section) to
be
$(\taub^2+|\mub|^2)^{-1/2}$, with $|\mub|^2$ the squared length with
respect any dual metric on $\pa X$, much as $r^{-1}$ could be taken as
the defining function of the boundary of the radial compactification
(away from $0$). (If one wants to allow the zero section, resp.\ the
origin, one can take $(\taub^2+|\mub|^2+1)^{-1/2}$, resp.\ $(r^2+1)^{-1/2}$.) From the perspective
of $X$, i.e.\ using $x$ in place of $t$, this amounts to conormality
at $\pa X$ (more precisely at $\overline{\Tb^*}_{\pa X}X$), with
values in classical symbols; thus, $x\pa_x$ and $\pa_{y_j}$ preserve
the smoothness in the fibers.

However, since $X$ itself has a boundary it also makes sense of
classicality at $\pa X$.
The {\em fully classical} operators (i.e.\ classical both in the symbolic
and in the boundary asymptotic expansion sense) are then ones
possessing an expansion in Taylor series in $x$, i.e.\ in terms of
exponentials $e^{-t}$. For operators of order $(0,0)$ thus full
classicality amounts to both $b$ being $\CI$ in the local coordinate
version of $\overline{\Tb^*}X$, and the residual term (which is
trivial in the symbolic sense) possessing an expansion in Taylor series in $x$.

Just like there was a partial classicality in the symbolic sense, with
conormal behavior in the base, there is also a partial classicality in
the base sense, with symbolic behavior in the differential sense; for
elements of $\Psibc^{m,0}$ this amounts to $b$ being smooth (with
values in symbols) in $(x,y)$, i.e.\ having a Taylor series expansion
in $x$, plus again the residual term being smooth in $x$. A limited
version of this partial conormality plays a role in defining the normal
operator below.

The space $\Psibc^{\infty,\infty}(X)=\cup_{m,l\in\RR}\Psibc^{m,l}(X)$ is a
filtered $*$-algebra, with $A\in\Psibc^{m,l}(X)$,
$A'\in\Psibc^{m',l'}(X)$ implying $AA'\in\Psib^{m+m',l+l'}(X)$ and
$A^*\in\Psibc^{m,l}(X)$ where the adjoint is taken with respect to any
non-degenerate (positive)
b-density, or indeed any non-degenerate polynomial multiple of this. In local
coordinates, a b-density has the form
$a\frac{|dx\,dy_1\,\ldots\,dy_{n-1}|}{x}$, with $a>0$ (on $X$,
including at $\pa X$) meaning that
this b-density is positive; notice that this is
$a\,|dt\,dy_1,\ldots\,dy_{n-1}|$; the polynomial multiples take the
form  $ax^p\frac{|dx\,dy_1\,\ldots\,dy_{n-1}|}{x}$ for some $p\in\RR$,
which include scattering densities, such as densities of
asymptotically conic metrics, for which $p=-n$. (So the scattering
$L^2$-space is $L^2_\scl(X)=x^{n/2} L^2_\bl(X)$.)

Furthermore, there is a principal symbol map
$$
\sigma_{m,l}:\Psibc^{m,l}\to S^{m,l}/S^{m-1,l};
$$
for classical operators one often identifies $\sigma_{m,l}(A)$ with a
homogeneous degree $m$ function on $\Tb^*X\setminus o$ without further
comments.  The principal symbol captures $\Psibc^{m,l}(X)$ modulo
$\Psibc^{m-1,l}(X)$, so if the principal symbol of $A\in
\Psibc^{m,l}(X)$ vanishes, then $A\in\Psibc^{m-1,l}(X)$. As usual, this
principal symbol is a $*$-algebra homomorphism, so
$$
\sigma_{m+m',l+l'}(AA')=\sigma_{m,l}(A)\sigma_{m',l'}(A').
$$
This algebra is thus commutative to leading order in the differential
sense, i.e.\ $[A,A']\in\Psibc^{m+m'-1,l+l'}(X)$, but there is no gain
in the decay order. Furthermore, one can compute the principal symbol
of $[A,A']$ as an element of $\Psibc^{m+m'-1,l+l'}(X)$ (rather than
just as an element of $[A,A']\in\Psibc^{m+m',l+l'}(X)$, in which sense
it vanishes by the algebra homomorphism property); it is given by the
usual Hamilton vector field expression:
$$
\sigma_{m+m'-1,l+l'}([A,A'])=\frac{1}{i}\sH_a a',\ a=\sigma_m(A),\ a'=\sigma_{m'}(A').
$$
For $l=0$, $\sH_a$ is a b-vector field on $\Tb^*X$, i.e.\ is tangent
to $\Tb^*_{\pa X}X$ (and in general it
simply has an extra weight factor); indeed in local coordinates it
takes the form
\begin{equation*}\begin{aligned}
&(\pa_{\taub} a) (x\pa_x)-(x\pa_x a) \pa_{\taub}+\sum 
_j\big((\pa_{(\mub)_j}a)\pa_{y_j}-(\pa_{y_j}a)\pa_{(\mub)_j}\big) \\
&\qquad=(-\pa_{\taub} a) \pa_t-\pa_t a (-\pa_{\taub})+\sum 
_j\big((\pa_{(\mub)_j}a)\pa_{y_j}-(\pa_{y_j}a)\pa_{(\mub)_j}\big),
\end{aligned}\end{equation*}
where the $-$ signs in the $\pa_t$-version correspond to
$\taub\,\frac{dx}{x}=-\taub\,dt$; notice that the second line is the
standard form of the Hamilton vector field taking into account that
$\taub$ is the {\em negative} of the canonical dual coordinate of $t$.

In order to capture decay, one can use the {\em normal operator} of
$A\in\Psibc^{m,0}(X)$, which is defined if $A$ is classical in the
base variable $x$ modulo $\Psibc^{m,-\delta}(X)$, i.e.\ $A$ differs
from a dilation invariant operator $N(A)$ on $(0,\infty)\times\pa X$,
with a neighborhood $[0,x_0)\times\pa X$ of $\pa X$ being identified
with this model locally, by an element of
$\Psibc^{m,-\delta}(X)$. The main advantage of this is that $N(A)$
can be analyzed via the Mellin transform in $x$, i.e.\ the Fourier
transform in $-t$, thanks to its dilation invariance; these reduce the
analysis to that of a {\em family} of pseudodifferential operators
$\hat N(A)$ on
$\pa X$. Note that the latter form a non-commutative algebra, thus
this a more complicated object than the principal symbol. For general
$l$, one can instead consider the {\em rescaled normal operator} $N(x^l A)$ for similar effect. For
example, if $g$ is a sc-metric which is asymptotically conic, the
normal operator of $\Delta_g$ is simply that of the Laplacian of the
conic model metric.

When working with variable (differential, in this case) order
operators, i.e.\ $m$ is a smooth function of $(x,y,\taub,\mub)$ which is homogeneous of degree zero
in $(\taub,\mub)$, it is also necessary to 
generalize the definition by allowing losses $\delta\in[0,1/2)$: 
$$
|\pa_t^k\pa_y^\alpha\pa_{\taub}^j\pa_{\mub}^\beta b|\leq C_{k\alpha j\beta}\langle(\taub,\mub)\rangle^{m-j-|\beta|+\delta(k+|\alpha|+j+|\beta|)};
$$
this class of symbols would have a subscript $\delta$, as would do the
corresponding class of pseudodifferential operators $\Psibcdelta^{m,0}$; see
\cite{Vasy:Minicourse} where this is discussed throughout the paper in
various settings.
The need for these (or slightly different versions) arises from the
logarithmic terms coming from differentiating expressions like
$\langle(\taub,\mub)\rangle^{m(x,y,\taub,\mub)}$; the latter is an
example of an elliptic order $m$ symbol.
These symbols and the corresponding operators work completely analogously to the $\delta=0$ case considered 
above, except that the principal symbol is defined in
$S^{m,l}_\delta/S^{m-1+2\delta,l}_\delta$, and the commutator of two operators as
above is in $\Psibcdelta^{m+m'-1+2\delta,l+l'}(X)$, and thus its
principal symbol is computed modulo
$S^{m+m'-1+4\delta,l+l'}_\delta$. It is important to keep in mind that
for our purposes $\delta>0$ can always be taken arbitrarily small.

\section{Symbolic estimate}\label{sec:symbolic}
We start the proof of the main theorem by proving a uniform (in
$\sigma$) symbolic estimate, which does not yet come with
compact error terms. We first show it with a special choice of $\tilde
r$, which makes the computation completely explicit, see Proposition~\ref{prop:symbolic-b-est}, and later on we
generalize to a more geometric condition on the differential order
$\tilde r$, see Proposition~\ref{prop:symbolic-b-est-more}.

\begin{prop}\label{prop:symbolic-b-est}
Let $l\in\RR$, $\beta>0$ and let
\begin{equation*}
\tilde r=\frac{1}{2}-(l+1)\pm\beta\frac{\taub}{(\taub^2+|\mub|^2)^{1/2}},
\end{equation*}
where $|.|$ is
the dual metric function of $g_{\pa X}$.

For $\sigma$ in a compact subset of $[0,\infty)$ and $0<K<\beta$ there exists $C>0$
such that for all $u\in \Hb^{\tilde r-K,l}$ with $P(\sigma)u\in
\Hb^{\tilde r-1,l+2}$ , we have $u\in\Hb^{\tilde r,l}$ and
\begin{equation}\label{eq:basic-b-est-0}
\|u\|_{\Hb^{\tilde r,l}}\leq
C(\|P(\sigma)u\|_{\Hb^{\tilde r-1,l+2}}+\|u\|_{\Hb^{\tilde r-K,l}}).
\end{equation}
\end{prop}

\begin{rem}
Unlike in the main theorem, Theorem~\ref{thm:main}, there are {\em no
  conditions} on the decay order $l$. The latter only enter if one
wants to improve (weaken) the decay order $l$ of the error term,
$\|u\|_{\Hb^{\tilde r-K,l}}$, as we do in the next section.
\end{rem}

\begin{rem}\label{rem:symbolic-b-est}
By the density of $\dCI(X)$ in $\cX=\cX_\sigma$, it in fact suffices
to prove the uniform estimate \eqref{eq:basic-b-est-0} for
$u\in\dCI(X)$ for the purposes of proving the main theorem.
Here we prove the significantly stronger statement made above; this
has the flavor of propagation of singularities (really,
regularity). We note that if we only want to prove the estimate for
$u\in\dCI(X)$, the regularization argument below is in fact
unnecessary, shortening the proof.
\end{rem}

\begin{rem}\label{rem:symbolic-b-est-improved-error}
Note that for any $\tilde r'$, $\|u\|_{\Hb^{\tilde r-K,l}}$ can be
bounded by an arbitrarily small
multiple of $\|u\|_{\Hb^{\tilde r,l}}$ plus a large multiple of
$\|u\|_{\Hb^{\tilde r',l}}$, and the former can be absorbed into the
left hand side of \eqref{eq:basic-b-est-0}, so in fact
\eqref{eq:basic-b-est-0} holds with $\|u\|_{\Hb^{\tilde r-K,l}}$
replaced by $\|u\|_{\Hb^{\tilde r',l}}$. Nonetheless, this {\em does
  not change} the need for the a priori assumption $u\in \Hb^{\tilde
  r-K,l}$ for some $0<K<\beta$; in the form \eqref{eq:basic-b-est-0} is stated, one can
make it implicit by saying that the inequality holds provided the
right hand side is finite. Thus, in the setting of
Remark~\ref{rem:symbolic-b-est}, thus working with $\cX$, it actually
suffices to prove the estimate \eqref{eq:basic-b-est-0} for any single
value of $K>0$, and then the stated version follows immediately.
\end{rem}

{\em The rest of this section, until the statement of Proposition~\ref{prop:symbolic-b-est-more}, consists of the proof of Proposition~\ref{prop:symbolic-b-est}.}

The estimate \eqref{eq:basic-b-est-0}
arises from a positive commutator argument. Although
$P(\sigma)\in\Diffb^2(X)$ only, it is of the form
\begin{equation}\label{eq:P-sigma-form-0}
P(\sigma)=P(0)+\sigma Q-\sigma^2,\qquad P(0)\in x^2\Diffb^2(X),\ Q\in
S^{-2-\delta}\Diffb^1(X)
\end{equation}
and
$\sigma^2$, by virtue of being a multiple of the identity operator,
commutes with every operator. Thus, the argument works as if $P(\sigma)$
were in $x^2\Diffb^2(X)\subset\Psib^{2,-2}(X)$ (in a uniform manner in
$\sigma$ with $\sigma$ in a fixed compact subset of $\Cx$), though
with a bit of care. Moreover, not only is $\sigma Q$ subprincipal (by
being a first order operator), but
it has an extra $x^\delta$ vanishing at $\pa X$, so {\em its} principal
symbol vanishes at $\pa X$.

While dealing with $P(0)$ it
is usually convenient to conjugate it by $x^{(n-1)/2}$ and multiply
by $x^{-2}$ from the left; this gives a formally self-adjoint operator
relative to the b-density $x^n\,dg$. However, in view of $\sigma^2$,
we avoid this conjugation, for $\sigma^2$ is both symmetric relative
to $dg$ and is a multiple of the identity operator. We thus use $dg$
for the density defining the inner product, and use it also for the
$L^2$-space even in the b-algebra below; notice that thus
$$
L^2=x^{n/2}L^2_\bl,
$$
with $L^2_\bl$ defined with respect to any
non-degenerate smooth b-density (i.e.\ locally a positive smooth
multiple of $\frac{dx}{x}\,dy$) since such a density is, up to an overall
positive smooth factor, $x^n$ times $dg$. Thus, $P(\sigma)$
is effectively (though not actually: actually it is merely in
$\Psib^{2,0}(X)$ due to $\sigma^2$) in $\Psib^{2,-2}(X)$; taking
$A\in\Psib^{\tilde r-1/2,l+1}(X)$, with $A^*=A$, we have
$i[P(\sigma),A^*A]\in\Psib^{2\tilde r,2l}(X)$; we also have for $\sigma\in\RR$,
$P(\sigma)=P(\sigma)^*$. In fact, if we allow $\im\sigma\neq 0$, then
depending on the sign of $\im\sigma$, we obtain $i$ times a positive
or negative term; this sign matches the sign below for the matching
choice of weight, and thus function space; see the end of the section. Thus,
\begin{equation}\begin{aligned}\label{eq:commutator-basic-b-est}
i&(P(\sigma)^*A^*A-A^*AP(\sigma))=i(P(\sigma)^*-P(\sigma))A^*A+i[P(\sigma),A^*A],\\
&[P(\sigma),A^*A]\in\Psib^{2\tilde
  r,2l}(X),
\end{aligned}\end{equation}
with principal symbol
$2a\sH_p a$, where $p$ is the principal symbol of $P(0)$. Below we arrange
that this has a definite sign.

This is done by a global version of real principal type and radial
point estimates. These were discussed, in the scattering setting, by
Melrose in \cite{RBMSpec}; the present version essentially presents a
globalized version of
\cite{Vasy:Minicourse}, which in turn is based on
\cite{Vasy-Dyatlov:Microlocal-Kerr} and
\cite{Baskin-Vasy-Wunsch:Radiation}.

Recall that the b-dual variable of $x$ is written as
$\taub$, that of $y_j$ as $(\mub)_j$, so covectors are written as
$$
\taub\,\frac{dx}{x}+\sum_j(\mub)_j\,dy_j,
$$
and a defining function of fiber
infinity of $\overline{\Tb^*}X$ is $\tilde\rho=(\taub^2+|\mub|^2)^{-1/2}$, with
$|\mub|^2$ the squared length with respect to any (dual) metric on
$\pa X$, but here we take this to be $g_{\pa X}$, see
Section~\ref{sec:background} and also Section~\ref{sec:2-micro}. Then we take
\begin{equation}\label{eq:b-symbol-choice}
a=x^{-l-1}(\taub^2+|\mub|^2)^{(\tilde r-1/2)/2}\psi(x),
\end{equation}
where $\tilde r$ is a function of the homogeneous degree zero
expression $\taub/(\taub^2+|\mub|^2)^{1/2}$ on $\Tb^*X$ (minus the zero
section), which is monotone along the Hamilton flow of $p$ at $\pa
X$, and where $\psi\geq 0$ is identically $1$ near $0$, and is supported
sufficiently close to $0$ (so that the collar neighborhood
decomposition is still valid, and later on so that the dynamical
behavior is unchanged). Concretely, depending on the incoming/outgoing choice ($\pm i0$
limits), we take, with $\beta>0$,
\begin{equation}\label{eq:b-symbol-weight-choice}
\tilde r=\frac{1}{2}-(l+1)\pm\beta\frac{\taub}{(\taub^2+|\mub|^2)^{1/2}}.
\end{equation}
Notice that on $\supp d\psi$, $x$ is bounded away from $0$, thus we
have a priori elliptic estimates. Also, $a$ is to be understood as an
amplitude, cut off from the zero section, but for a symbolic argument,
such as those below, any additional term generated by this cutoff is
of order $-\infty$, thus can be absorbed into the error.

With $p_\pa=x^2(\taub^2+|\mub|^2)$ the restriction of $p$ to $\pa X$,
extended using the local product structure as an $x$-independent
function,
$\sH_p-\sH_{p_\pa}$ is $x^{2+\delta}$ times a homogeneous (with
respect to dilations in the fibers of $\Tb^*X$) degree $1$
vector field tangent (in the sense of being a b-vector field with
symbolic order $0$ coefficients) to $\Tb^*_{\pa X}X$, i.e.\ is
$x^{2+\delta}$ times a linear
combination of $x\pa_x,\pa_{y_j}$ with homogeneous degree $1$, and
$\pa_{\taub},\pa_{(\mub)_j}$ with homogeneous degree $2$
coefficients, with these coefficients symbolic of order $0$, i.e.\
remain bounded under applications of $x\pa_x$,
$\pa_{y_j}$, $\pa_{\taub}$, $\pa_{(\mub)_j}$. (Here homogeneity is
used to encode the regularity at fiber infinity.) Furthermore,
\begin{equation*}\begin{aligned}
\sH_{p_\pa}&=(\pa_{\taub} p)x\pa_x-(x\pa_x
p)\pa_{\taub}+x^2\sH_{|\mub|^2}\\
&=2x^2\taub x\pa_x-2x^2(\taub^2+|\mub|^2)\pa_{\taub}+x^2\sH_{|\mub|^2},
\end{aligned}\end{equation*}
so
\begin{equation*}\begin{aligned}
\sH_p=2x^2(\taub+x^\delta q_0) x\pa_x&-2x^2(\taub^2+|\mub|^2+x^\delta
\tilde q_0)\pa_{\taub}\\
&+x^2\sH_{|\mub|^2}+\sum_{j=1}^{n-1} x^{2+\delta}
(q_j\pa_{y_j}+\tilde q_j\pa_{(\mub)_j}),
\end{aligned}\end{equation*}
with $q_j$ homogeneous of degree $1$, $\tilde q_j$ homogeneous of
degree $2$, symbolic of order $0$ in $x$.
As $\sH_{|\mub|^2}$ annihilates $a$ for the above choices,
\begin{equation}\begin{aligned}\label{eq:H-pa-p-form}
\sH_{p_\pa} a=2x^2 x^{-l-1}\Big(&-(l+1)\taub (\taub^2+|\mub|^2)^{(\tilde
  r-1/2)/2}\psi(x)\\
&-(\tilde r-1/2)\taub (\taub^2+|\mub|^2)^{(\tilde
  r-1/2)/2}\psi(x)\\
&-(\pa_{\taub} \tilde r) (\taub^2+|\mub|^2)^{1+(\tilde
  r-1/2)/2}\log (\taub^2+|\mub|^2) \psi(x)\\
&+x\psi'(x) \taub (\taub^2+|\mub|^2)^{(\tilde
  r-1/2)/2}\Big),
\end{aligned}\end{equation}
\begin{equation}\begin{aligned}\label{eq:tilde-r-deriv}
\pa_{\taub}\tilde r&=(\pm\beta)
(\taub^2+|\mub|^2)^{-3/2}(\taub^2+|\mub|^2-\taub^2)\\
&=(\pm \beta)|\mub|^2 (\taub^2+|\mub|^2)^{-3/2}.
\end{aligned}\end{equation}
On the other hand,
\begin{equation*}\begin{aligned}
\sH_p a-\sH_{p_\pa} a=2x^2 x^{-l-1} x^\delta\Big((\taub^2&+|\mub|^2)^{(\tilde
  r-1/2)/2}(f_0\psi+\tilde f_0 x\psi')\\
&+\psi f_1 (\taub^2+|\mub|^2)^{(\tilde
  r-1/2)/2-1}\\
&+\psi  f_2 (\taub^2+|\mub|^2)^{(\tilde
  r-1/2)/2}\log (\taub^2+|\mub|^2)\Big),
\end{aligned}\end{equation*}
with  $f_0,\tilde f_0$ homogeneous of degree $1$, $f_1$ homogeneous of
degree $3$, $f_2$ homogeneous of degree $1$.
Thus
\begin{equation}\begin{aligned}\label{eq:H-p-a-expand}
\sH_p a&=-2(\pm\beta)x^{-l+1}\psi(x)(\taub^2+|\mub|^2)^{(\tilde
  r-3/2)/2}\\
&\qquad\qquad\qquad\cdot\Big(\taub^2+x^\delta f^\sharp_1+(|\mub|^2+x^\delta f^\sharp_2) \log (\taub^2+|\mub|^2)\Big)+e,\\
e&=2x^2 x^{-l-1} x\psi'(x) (\taub+x^\delta f^\sharp_0) (\taub^2+|\mub|^2)^{(\tilde
  r-1/2)/2}
\end{aligned}\end{equation}
with $f^\sharp_0$ homogeneous of degree $1$, $f^\sharp_1,f^\sharp_2$
homogeneous of degree $2$.
The first term in $\sH_p a$ is the `main term', while the second term,
$e$, is the `error term' and it is controlled by a priori elliptic
estimates in $X^\circ$ near $\pa X$, so its sign is irrelevant for
considerations below. Returning to the main term,
for $\taub^2+|\mub|^2>2$ and $x$ sufficiently small, which can be
arranged by making $\supp\psi$ sufficiently small, it is negative, resp.\ positive, depending on
whether the sign in front of $\beta$ is positive, resp.\ negative,
since
\begin{equation}\label{eq:est-xdelta-terms}
\taub^2+x^\delta f^\sharp_1+(|\mub|^2+x^\delta f^\sharp_2) \log
(\taub^2+|\mub|^2)\geq \taub^2+|\mub|^2-Cx^\delta(\taub^2+|\mub|^2)
\end{equation}
there.
Notice
that this is slightly stronger, away from the radial sets (where
$\mub=0$), than the standard positive commutator estimate due to the
presence of the logarithmic factor, but we do not need to explicitly
take advantage of this, though we do need to point out that
the expression in the big parentheses is a symbol of order $2+\delta'$  for all $\delta'>0$, and thus the
usual arguments go through, as discussed in the non-radial point
setting in e.g.\ in the appendix of
\cite{Baskin-Vasy-Wunsch:Radiation}, see also
\cite{Unterberger:Resolution} for earlier work in which the
logarithmic improvement played an important role. We explicitly point
out the expression in the big parentheses in \eqref{eq:H-p-a-expand}
is not only a symbol of order $2+\delta'$ for all $\delta'>0$, but
corresponding to \eqref{eq:est-xdelta-terms} it has
`elliptic' positive lower bounds in an order $2$ sense, thus the standard argument
shows that its positive square root is a symbol of order $1+\delta'$
for all $\delta'$ with an order $1$ `elliptic' positive lower bound,
hence the corresponding term in $2a\sH_p a$ is
\begin{equation}\begin{aligned}\label{eq:H-p-a-sqrt}
b_0^2=4\beta x^{-2l}\psi(x)^2(\taub^2+|\mub|^2)^{\tilde
  r-1}\Big(\taub^2+x^\delta f^\sharp_1+(|\mub|^2+x^\delta f^\sharp_2) \log (\taub^2+|\mub|^2)\Big)
\end{aligned}\end{equation}
with $b_0$ the non-negative square root which is in $x^{-l}S^{\tilde
  r+\delta'}$ for all $\delta'>0$ with a positive elliptic lower bound
in $x^{-l}S^{\tilde r}$.

In fact, as usual, cf.\ \cite{RBMSpec,Vasy-Dyatlov:Microlocal-Kerr},
in order to prove Proposition~\ref{prop:symbolic-b-est} (as opposed
to the weaker version stated in Remark~\ref{rem:symbolic-b-est}) one needs a family of operators
$A_\ep$, where $\ep\in[0,1]$ is a regularization parameter, with, for
$0<K<\beta$ as in the statement of the proposition, $A_\ep\in\Psib^{\tilde r-1/2-K,l+1}(X)$ for $\ep>0$,
$\{A_\ep:\ \ep\in[0,1]\}$ uniformly bounded in $\Psib^{\tilde r-1/2,l+1}(X)$,
converging to $A_0$ as $\ep\to 0$ in slightly weaker topologies, that
of $\Psib^{\tilde r-1/2+\delta',l+1}(X)$, for all $\delta'>0$. (This
convergence implies strong convergence in bounded operators between
b-Sobolev spaces of appropriately shifted orders.)
Concretely, we can take
\begin{equation}\label{eq:a-ep-regularize}
a_\ep=a\phi_\ep(\tilde\rho^{-1}),\qquad \phi_\ep(s)=(1+\ep s)^{-K},\qquad \tilde\rho=(\taub^2+|\mub|^2)^{-1/2}.
\end{equation}
Note that
$$
\phi_\ep'(s)=-K\ep(1+\ep s)^{-1}\phi_\ep,
$$
so, with $f^\sharp_3$ homogeneous of degree $2$ with symbolic order
$0$ coefficients (in the same sense as above),
\begin{equation*}\begin{aligned}
\sH_p \phi_\ep(\tilde\rho^{-1})&=-K\ep(1+\ep
\tilde\rho^{-1})^{-1}\phi_\ep(\tilde\rho^{-1})\sH_p\tilde\rho^{-1}\\
&=x^2K\ep(1+\ep
(\taub^2+|\mub|^2)^{1/2})^{-1}(-2) \big( (\taub^2+|\mub|^2)^{1/2}\taub+x^\delta
f^\sharp_3\big) \phi_\ep(\tilde\rho^{-1}).
\end{aligned}\end{equation*}
Correspondingly
\begin{equation}\begin{aligned}\label{eq:Hp-a-ep}
&\sH_p a_\ep =(\sH_p a)\phi_\ep+a(\sH_p\phi_\ep(\tilde\rho^{-1}))\\
&=-2x^{-l+1}\psi(x)\phi_\ep(\tilde\rho^{-1})\Big((\pm\beta)(\taub^2+|\mub|^2)^{(\tilde
  r-3/2)/2}\big(\taub^2 +x^\delta f^\sharp_1+\big(|\mub|^2 +x^\delta
f^\sharp_2\big) \log (\taub^2+|\mub|^2)\big)\\
&\qquad\qquad+(\taub^2+|\mub|^2)^{(\tilde
  r-1/2)/2}K\ep(1+\ep
(\taub^2+|\mub|^2)^{1/2})^{-1}\big((\taub^2+|\mub|^2)^{1/2}\taub
+x^\delta f^\sharp_3\big)\Big)\\
&\qquad\qquad+e \phi_\ep(\tilde\rho^{-1})\\
&=-2x^{-l+1}\psi(x)\phi_\ep(\tilde\rho^{-1}) (\taub^2+|\mub|^2)^{(\tilde
  r-3/2)/2}\Big((\pm\beta)\big(\taub^2 +x^\delta f^\sharp_1+\big(|\mub|^2 +x^\delta f^\sharp_2\big) \log (\taub^2+|\mub|^2)\big)\\
&\qquad\qquad+K \big(\taub (\taub^2+|\mub|^2)^{1/2}+x^\delta f^\sharp_3\big)\frac{\ep (\taub^2+|\mub|^2)^{1/2}}{1+\ep
(\taub^2+|\mub|^2)^{1/2}}\Big) +e \phi_\ep(\tilde\rho^{-1}).
\end{aligned}\end{equation}
We point out that
$$
\frac{\ep (\taub^2+|\mub|^2)^{1/2}}{1+\ep
(\taub^2+|\mub|^2)^{1/2}}
$$
is uniformly bounded in symbols of order $0$ (for $\ep\in[0,1]$), with
its supremum bounded by $1$, and
tends to $0$ as $\ep\to 0$ in symbols of positive order.

Notice that in the
expression after the last equality in \eqref{eq:Hp-a-ep}, for the $+$ sign in $\pm$, the
two terms in the big parentheses have the same sign in $\taub>0$ and
opposite signs in $\taub<0$ (for $x$ small, when the $x^\delta$ terms
can be absorbed in the others, as in the above discussion around \eqref{eq:est-xdelta-terms}), while for the $-$ sign in $\pm$ the roles
of $\taub>0$ and $\taub<0$ reverse. Since for sufficiently large
$\taub^2+|\mub|^2$, $\log (\taub^2+|\mub|^2)$ is bounded below by any
pre-specified constant $C_0$, we have that $\taub^2+|\mub|^2 \log
(\taub^2+|\mub|^2)\geq\taub^2+C_0|\mub|^2$ there. Correspondingly, for
any $C_1>0$, in the region where in addition $|\taub|\leq C_1|\mub|$,
$\taub (\taub^2+|\mub|^2)^{1/2}\leq\taub^2+\frac{1}{2}|\mub|^2\leq
(C_1^2+1)|\mub|^2$ shows that in this region, the second term can be
absorbed into, say, $\frac{1}{2}$ of the first term by choosing $C_0$ large. Taking the $+$
sign in $\pm$ for definiteness, in the complement of
this region either $\taub>0$, and thus the two terms have the same
sign, or $\taub<0$ and $|\mub|<C_1^{-1}|\taub|$; in this region thus
the second term is bounded by $K(1+C_1^{-2})\taub^2$. Here $C_1>0$ can
be chosen arbitrarily large, thus as $K<\beta$, the second term can
be absorbed into the first. This is exactly the limit of
regularization one can do, i.e. {\em this step is the
  cause of the $K<\beta$ restriction in the statement of the proposition}. Correspondingly, for the $+$ sign in $\pm$
(for the $-$ sign, the overall $-$ sign on the right hand side of the
next equation is removed),
\begin{equation}\begin{aligned}\label{eq:Hp-a-ep-pos}
2a_\ep\sH_p a_\ep
&=-\phi_\ep(\tilde\rho^{-1})^2(b^2+b_1^2+b_{2,\ep}^2)+
\phi_\ep(\tilde\rho^{-1})^2 ae
\end{aligned}\end{equation}
with
\begin{equation*}
b^2=x^{-2l}\psi(x)^2(\beta-K) (\taub^2+|\mub|^2)^{(\tilde
  r-1/2)/2+(\tilde r-3/2)/2+1},
\end{equation*}
thus
$$
b=x^{-l}\sqrt{\beta-K} \psi(x)(\taub^2+|\mub|^2)^{\tilde 
  r},
$$
and appropriate choices of $b_1$ and $b_{2,\ep}$.

Let $B$, $B_1$,
$B_{2,\ep}$, $S_\ep$, $E_\ep$ have principal symbols $b,b_1,b_{2,\ep},\phi_\ep, \phi_\ep(\tilde\rho^{-1})^2 ae$
respectively, and so that $S_\ep\in\Psib^{-K,0}(X)$ uniformly bounded in
$\Psib^{0,0}(X)$, converging to $\Id$ in $\Psib^{\delta',0}(X)$ (with
$\delta'>0$ arbitrary), $B,B_{2,\ep}\in\Psib^{\tilde r,l}(X)$, with $B_{2,\ep}$
uniformly bounded in this space, $B_1\in\Psib^{\tilde r+\delta',l}(X)$
for all $\delta'>0$ (arising from the variable order), while $E_\ep$
is uniformly bounded in $\Psib^{2\tilde r,2l}(X)$ and supported away
from $\pa X$. Therefore, as
can be seen by
computing the principal symbol of both sides, which agree,
one then has
$$
i(P(\sigma)^*A_\ep^*A_\ep-A_\ep^*A_\ep P(\sigma))=-S_\ep^* (B^*B+B_1^*B_1+B_{2,\ep}^*B_{2,\ep}) S_\ep+E_\ep+F_\ep,
$$
with $F_\ep$ uniformly bounded in
$\Psib^{2\tilde r-1+\delta',2l}(X)$ for all $\delta'>0$. Then 
$$
\langle i(P(\sigma)^*A_\ep^*A_\ep-A_\ep^*A_\ep P(\sigma))u,u\rangle=-\|BS_\ep
u\|^2-\|B_{2,\ep}S_{\ep} u\|^2+\langle E_\ep u,u\rangle+\langle F_\ep u,u\rangle.
$$
Now,
$$
\langle i(P(\sigma)^*A_\ep^*A_\ep-A_\ep^*A_\ep
P(\sigma))u,u\rangle=\langle i A_\ep u,A_\ep P(\sigma) u\rangle-\langle
iA_\ep P(\sigma)u,A_\ep u\rangle,
$$
where the moving of the adjoint over to the other side of the pairing
actually requires an additional, straightforward, regularization
(without the limitations on the amount of regularization, as in the
case of $K$ above), see
the proof of the radial point estimates in \cite[Section~5.4.7]{Vasy:Minicourse}.
Using Cauchy-Schwartz (with an elliptic operator used to shift the
orders in the two slots of the pairing), estimating the product by
a sum of squares, with a small constant $\digamma^{-1}>0$ in front of the
$A_\ep u$ term, and absorbing the $A_\ep u$, one deduces the
regularized version of the estimate for $\ep>0$:
$$
\|BS_\ep u\|^2\leq \digamma \|A_\ep
P(\sigma)u\|^2_{\Hb^{-1/2,1}(X)}+\langle E_\ep u,u\rangle+\langle F_\ep u,u\rangle.
$$
Let $\tilde\psi=\tilde\psi(x)$ be a $\CI$ function, $\equiv 1$ on
$\supp\psi$ still supported in a small collar neighborhood of $\pa X$.
Now, {\em for $u\in \Hb^{\tilde r-1/2+\delta',l}$ with $P(\sigma)\in
  \Hb^{\tilde r-1,l+2}$} all terms on the right hand side remain uniformly bounded as
$\ep\to 0$ by the a priori
assumptions, and by elliptic estimates in $x>0$, namely the latter
gives, with a uniform constant,
$$
|\langle E_\ep u,u\rangle|\leq C'\|\tilde\psi P(\sigma)u\|^2_{\Hb^{\tilde r-2,l'}}+\|u\|^2_{\Hb^{\tilde r',l'}}
$$
for any $\tilde r',\tilde l'$. Thus, letting $\ep\to 0$, using the (sequential!) weak-* compactness
of the unit ball in $L^2$ (so $BS_\ep u$ subsequentially converges) as
well as that $BS_\ep u\to Bu$ in tempered distributions, proves that
for $0<K'<\min(K,1/2)$ and $\delta'>0$ there exists $C''>0$ such that for
$u\in \Hb^{\tilde r-K',l}$ with $P(\sigma)u\in \Hb^{\tilde r-1,l+2}$ we have
$$
\|\psi u\|_{\Hb^{\tilde r,l}}\leq
C''(\|\tilde\psi P(\sigma)u\|_{\Hb^{\tilde r-1,l+2}}+\|u\|_{\Hb^{\tilde r-1/2+\delta',l}}).
$$
Then an iterative argument, of step size $<1/2$, improves $\tilde r-K'$ in the norm on
the right hand side to any $\tilde r'$, under the assumption $u\in \Hb^{\tilde r-K,l}$.

Finally, as $P(\sigma)$ is elliptic away from $\pa X$, for any $\chi$
compactly supported in $X^\circ$ and $\tilde\chi$ also compactly
supported there but $\equiv 1$ on $\supp\chi$, one has the
elliptic estimate
\begin{equation}\label{eq:b-symbol-elliptic}
\|\chi u\|_{\Hb^{\tilde r,l}}\leq C'(\|\tilde\chi
P(\sigma)u\|_{\Hb^{\tilde r-1,l+2}}+\|u\|_{\Hb^{\tilde r-K,l}}),
\end{equation}
where the decay order actually does not matter (as all supports are
away from $\pa X$, except for the last error term -- but even that
could be further localized).

In combination, for any $K<\beta$ there exists $C>0$ such that for
$u\in \Hb^{\tilde r-K,l}$ with $P(\sigma)u\in \Hb^{\tilde r-1,l+2}$ we have
$$
\|u\|_{\Hb^{\tilde r,l}}\leq
C(\| P(\sigma)u\|_{\Hb^{\tilde r-1,l+2}}+\|u\|_{\Hb^{\tilde r-K,l}}).
$$
This thus proves \eqref{eq:basic-b-est-0}.

Finally, if we allow $\sigma$ complex, with $\im\sigma^2$ of the correct
sign, matching $\pm$ in \eqref{eq:b-symbol-weight-choice-0}, then in
\eqref{eq:commutator-basic-b-est}, we have the extra term
$i(P(\sigma)^*-P(\sigma))A^*A$. We only consider the case of the
spectral family, when this is
$i(\sigma^2-\bar\sigma^2)a^2=-2\im(\sigma^2) a^2$, thus matches the
sign of $2a\sH_p a$ in the $\im(\sigma^2)\geq 0$, $+\beta$ choice for
$\tilde r$ case, as
well as in the $\im(\sigma^2)\leq 0$, $-\beta$ choice for $\tilde r$
case, thus can simply be dropped from the estimate.

Proposition~\ref{prop:symbolic-b-est} has the following strengthened
more geometric version:

\begin{prop}\label{prop:symbolic-b-est-more}
Suppose that $\tilde r$ is a homogeneous degree $0$
function on $\Tb^*X\setminus o$ (i.e.\ a smooth function on $\Sb^*X$)
with $\tilde r>\frac{1}{2}-(l+1)$ at the source, $\{\taub>0,\mub=0\}$,
$\tilde r<\frac{1}{2}-(l+1)$ at the sink, $\{\taub<0,\mub=0\}$, and 
$
-(\taub^2+|\mub|^2)^{-1/2}x^{-2}\sH_{p_\pa}\tilde r
$
is non-negative.
Let $K>0$ be such that $\tilde r>\frac{1}{2}-(l+1)+K$ in a
neighborhood of the source.

For $\sigma$ in a compact subset of $[0,\infty)$ there exists $C>0$
such that for all $u\in \Hb^{\tilde r-K,l}$ with $P(\sigma)u\in
\Hb^{\tilde r-1,l+2}$ , we have $u\in\Hb^{\tilde r,l}$ and
\begin{equation}\label{eq:basic-b-est-0-more}
\|u\|_{\Hb^{\tilde r,l}}\leq
C(\|P(\sigma)u\|_{\Hb^{\tilde r-1,l+2}}+\|u\|_{\Hb^{\tilde r-K,l}}).
\end{equation}

The analogous result also holds with the source and sink switched, and the
positivity of $-(\taub^2+|\mub|^2)^{-1/2}x^{-2}\sH_{p_\pa}\tilde r$ is
replaced by that of $(\taub^2+|\mub|^2)^{-1/2}x^{-2}\sH_{p_\pa}\tilde r$.
\end{prop}

\begin{rem}
Note that the $+$ sign in $\pm\beta$ in
Proposition~\ref{prop:symbolic-b-est} corresponds to the first case of
this proposition, while the $-$ sign corresponds to the second case
(`analogous result').
\end{rem}

\begin{proof}
We can follow the proof of Proposition~\ref{prop:symbolic-b-est}
very closely, however we modify the commutant slightly and consider,
for $\tilde\beta>0$ to be specified and with $\pm$ meaning $+$ for the
first case of the proposition, $-$ for the second (`analogous') case,
\begin{equation}\begin{aligned}\label{eq:a-choice-mod}
a=x^{-l-1}(\taub^2+|\mub|^2)^{(\tilde r-1/2)/2}\psi(x) \exp\Big(\pm\frac{\tilde\beta}{2}\frac{\taub}{(\taub^2+|\mub|^2)^{1/2}}\Big), 
\end{aligned}\end{equation}
which is
\begin{equation}\label{eq:overall-weight}
\exp\Big(\pm\frac{\tilde\beta}{2}\frac{\taub}{(\taub^2+|\mub|^2)^{1/2}}\Big)
\end{equation}
times the principal symbol used for our symbolic computation in
\eqref{eq:b-symbol-choice}, except that our choice of $\tilde r$ has
been generalized from \eqref{eq:b-symbol-weight-choice}. Notice that this new factor is a smooth
function on $\Sb^*X$, thus does {\em not} change the order of $a$, but
it does affect the principal symbol. Its role is to obtain positivity
of the commutator even away from the radial points even when $\tilde
r$ has an indefinite derivative along $H_{p_\pa}$; indeed note that
this is the exponential of the non-constant part of the choice of
$\tilde r$ employed in \eqref{eq:b-symbol-weight-choice}. We remark
that \eqref{eq:a-choice-mod} will be the principal symbol of our commutant choice for the
normal operator computation in \eqref{eq:a-choice} (albeit with the
special choice of $\tilde r$ from \eqref{eq:b-symbol-weight-choice});
we added the $\frac{1}{2}$ factor in front of $\tilde\beta$ for consistency of notation with the
next section.

Below we prove the first case of the proposition; the analogous case
is completely similar. Thus, we take the sign in front of
$\tilde\beta$ to be positive.
In \eqref{eq:H-pa-p-form}, there is now a new overall factor
\eqref{eq:overall-weight} on the first line,
the third line is replaced by
$$
+\frac{1}{2}x^{-2}(\sH_{p_\pa} \tilde r) (\taub^2+|\mub|^2)^{(\tilde
  r-1/2)/2}\log (\taub^2+|\mub|^2)\psi(x),
$$
and there is a new line, inserted between the third and fourth:
$$
-(\taub^2+|\mub|^2)^{1+(\tilde r-1/2)/2}\frac{\tilde\beta}{2}\pa_{\taub} \big(\frac{\taub}{(\taub^2+|\mub|^2)^{1/2}}\big)\psi(x),
$$
which is
$$
-\frac{\tilde\beta}{2}|\mub|^2 (\taub^2+|\mub|^2)^{-1/2+(\tilde r-1/2)/2}\psi(x),
$$
cf.\ \eqref{eq:tilde-r-deriv}.

Thus \eqref{eq:H-p-a-expand} is replaced by
\begin{equation}\begin{aligned}\label{eq:H-p-a-expand-mod}
\sH_p a&=-2x^{-l+1}\psi(x)(\taub^2+|\mub|^2)^{(\tilde
  r-3/2)/2}\\
&\qquad\qquad\qquad\cdot\Big(\big(\big(\tilde r-\frac{1}{2}+(l+1)\big)\taub+\frac{\tilde\beta}{2}|\mub|^2
+x^\delta f^\sharp_1\big)  (\taub^2+|\mub|^2)^{1/2}\\
&\qquad\qquad\qquad\qquad+(-\frac{1}{2}(\taub^2+|\mub|^2)^{1/2} x^{-2}\sH_{p_\pa} \tilde r+x^\delta f^\sharp_2) \log (\taub^2+|\mub|^2)\Big)+e,\\
e&=2x^2 x^{-l-1} x\psi'(x) (\taub+x^\delta f^\sharp_0) (\taub^2+|\mub|^2)^{(\tilde
  r-1/2)/2}
\end{aligned}\end{equation}
with $f^\sharp_0$ homogeneous of degree $1$, $f^\sharp_1,f^\sharp_2$
homogeneous of degree $2$. Now
\begin{equation}\label{eq:undiff-tilde-r}
\big(\tilde
r-\frac{1}{2}+(l+1)\big)\taub(\taub^2+|\mub|^2)^{1/2}
\end{equation}
is bounded below by a positive multiple
of $\taub^2+|\mub|^2$ near the sources and sinks,
while $\frac{\tilde\beta}{2}|\mub|^2$ is similarly bounded below on
any compact set disjoint
from the sources and sinks, such as on the complement of a set on
which the first term had the desired positive bound, and is
nonnegative everywhere. Correspondingly, choosing
$\tilde\beta>0$ sufficiently large, the sum of these two terms is
bounded below by a positive multiple of $\taub^2+|\mub|^2$.
Also,
$-\frac{1}{2}(\taub^2+|\mub|^2)^{1/2} x^{-2}\sH_{p_\pa} \tilde r$ is
non-negative.
One can then complete the argument as
in the proof of Proposition~\ref{prop:symbolic-b-est} when one has
$u\in \Hb^{\tilde r,l}$ a priori (thus there is no need for
regularization). In the general case, the regularization again
proceeds almost as before since the regularization factor is
unchanged, and now increasing $\tilde\beta$ can dominates this
regularization factor
away from the sources and sinks,
while near the source the two terms have the same
sign, while near the sink, $|\mub|<C_1^{-1}|\taub|$ where $C_1>0$ can
be chosen large; in this region thus
the second, regularization, term of the big parentheses of
\eqref{eq:Hp-a-ep} is bounded by $K(1+C_1^{-2})\taub^2$, and can be
absorbed into the
$$
\big(\tilde r-\frac{1}{2}+(l+1)\big)\taub+\frac{\tilde\beta}{2}|\mub|^2
+x^\delta f^\sharp_1\big)  (\taub^2+|\mub|^2)^{1/2}
$$
term of
\eqref{eq:H-p-a-expand-mod} as $\tilde r-\frac{1}{2}+(l+1)-K>0$. This
completes the proof in general.
\end{proof}

\section{The normal operator and proof of the main theorem}\label{sec:normal}

\subsection{Proof of the main theorem from an estimate from a normal
  operator estimate}

We now improve on \eqref{eq:basic-b-est-0} and \eqref{eq:basic-b-est-0-more},
\begin{equation}\label{eq:basic-b-est}
\|u\|_{\Hb^{\tilde r,l}}\leq
C(\|P(\sigma)u\|_{\Hb^{\tilde r-1,l+2}}+\|u\|_{\Hb^{\tilde r-K,l}}),
\end{equation}
valid for $u\in\Hb^{\tilde r-K,l}$, $K>0$ with $\tilde
r>\frac{1}{2}-(l+1)+K$ at the source in case $\tilde r$ is monotone
decreasing in the sense of Proposition~\ref{prop:symbolic-b-est-more}, with $P(\sigma)u\in
\Hb^{\tilde r-1,l+2}$, and an analogous statement with source replaced
by sink if $\tilde r$ is monotone
increasing. Recall from Remark~\ref{rem:symbolic-b-est-improved-error} that in fact
$\|u\|_{\Hb^{\tilde r-K,l}}$ can be replaced by $\|u\|_{\Hb^{-N,l}}$ for any $-N$, as long as one keeps in mind that one
needs to have $u\in\Hb^{\tilde r-K,l}$, $K>0$ as above, a priori.

Namely, we make the error term on the right
hand side, $\|u\|_{\Hb^{\tilde r-K,l}}$, replaced by a compact error. Directly
we prove:

\begin{prop}\label{prop:impr-b-est}
Let $S\subset[0,\infty)$ compact, and suppose that $\tilde r$ is as in
Proposition~\ref{prop:symbolic-b-est-more}:
$\tilde r$ is a homogeneous degree $0$
function on $\Tb^*X\setminus o$ (i.e.\ a smooth function on $\Sb^*X$)
with $\tilde r>\frac{1}{2}-(l+1)$ at the source, $\{\taub>0,\mub=0\}$,
$\tilde r<\frac{1}{2}-(l+1)$ at the sink, $\{\taub<0,\mub=0\}$, and 
$
-(\taub^2+|\mub|^2)^{-1/2}x^{-2}\sH_{p_\pa}\tilde r
$
is non-negative.
Let $K>0$ be such that $\tilde r>\frac{1}{2}-(l+1)+K$ in a
neighborhood of the source.

There exists $C$ such that
for $u\in\Hb^{\tilde r-K,l}$ with $P(\sigma)u\in \Hb^{\tilde r-1,l+2}$, the estimate
\begin{equation}\label{eq:impr-b-est}
\|u\|_{\Hb^{\tilde r,l}}\leq
C(\|P(\sigma)u\|_{\Hb^{\tilde r-1,l+2}}+\|u\|_{\Hb^{\tilde
    r-K,l-\delta}})
\end{equation}
holds for $\sigma\in S$.

The analogous result also holds with the source and sink switched, and the
positivity of $-(\taub^2+|\mub|^2)^{-1/2}x^{-2}\sH_{p_\pa}\tilde r$ is
replaced by that of $(\taub^2+|\mub|^2)^{-1/2}x^{-2}\sH_{p_\pa}\tilde r$.
\end{prop}

\begin{rem}
As an example, $\tilde r=\hat r_\pm$ satisfies the requirements where
\begin{equation}\label{eq:b-symbol-weight-choice-special}
\hat r_\pm=\hat
r_\pm(\beta)=\frac{1}{2}-(l+1)\pm\beta\frac{\taub}{(\taub^2+|\mub|^2)^{1/2}},\qquad\beta>0;
\end{equation}
the $+$ sign corresponds to the first case, the $-$ sign to the second
(`analogous') case.
\end{rem}

\begin{rem}\label{rem:compact-error-irrelev-norm}
The improvement in the differential order of the Sobolev norm
of $u$ on the right hand side is actually arbitrary, i.e.\ $\tilde
r-K$ could be replaced by $-N$, $N$ arbitrary, with the understanding that one needs
membership of $u$ in $\Hb^{\tilde r-K,l}$ for some $0<K<\beta$, cf.\
Remark~\ref{rem:symbolic-b-est-improved-error}. In fact, by similar
considerations the decay order is also arbitrary, but we still need
the $\Hb^{\tilde r-K,l}$ membership of $u$ as stated.
\end{rem}

A standard argument allows one to conclude from this a uniform estimate
without a compact error under an injectivity hypothesis; note that
this immediately implies Theorem~\ref{thm:main}, which we restate in
the present stronger version below.

\begin{prop}\label{prop:remove-cpt-error}
With the notation of Proposition~\ref{prop:impr-b-est}.
suppose that $P(0):\Hb^{\tilde r,l}\to\Hb^{\tilde r-2,l+2}$ has trivial
nullspace. There exist
$\sigma_0>0$ and $C'>0$ such that for $|\sigma|<\sigma_0$,
\begin{equation}\label{eq:res-b-est}
\|u\|_{\Hb^{\tilde r,l}}\leq
C'\|P(\sigma)u\|_{\Hb^{\tilde r-1,l+2}}
\end{equation}
for $u\in\Hb^{\tilde r,l}$ with $P(\sigma)u\in
\Hb^{\tilde r-1,l+2}$.
\end{prop}

\begin{proof}[Proof of Proposition~\ref{prop:remove-cpt-error} given Proposition~\ref{prop:impr-b-est}]
Indeed, if \eqref{eq:res-b-est} is not true, there is a sequence $u_j$, which one may
assume has unit norm in $\Hb^{\tilde r,l}$, and with
$P(\sigma_j)u_j\in \Hb^{\tilde r-1,l+2}$, and a sequence $\sigma_j\to
0$ such that $P(\sigma_j)u_j\to 0$ in $\Hb^{\tilde r-1,l+2}$. By
taking a subsequence (not shown in notation), using the sequential
compactness of the unit ball in $\Hb^{\tilde r,l}$ in the weak
topology, and the compactness of the inclusion
$\Hb^{\tilde r,l}\to \Hb^{\tilde r-K,l-\delta}$, $K>0$, one may assume that
there is $u\in \Hb^{\tilde r,l}$ such that $u_j\to u$ weakly in
$\Hb^{\tilde r,l}$ and strongly in $\Hb^{\tilde r-K,l-\delta}$. By
\eqref{eq:impr-b-est} then $\liminf \|u_j\|_{\Hb^{\tilde
    r-K,l-\delta}}\geq C^{-1}>0$, so $u\neq 0$ by the strong convergence. On the other hand,
$P(\sigma_j)u_j\to P(0) u$ in $\Hb^{\tilde r-K-2,l-\delta}$ as
$$
P(\sigma_j)u_j-P(0)u=(P(\sigma_j)-P(0))u_j+P(0)(u_j-u)
$$
since $P(\sigma_j)\to P(0)$ as bounded operators in $\cL(\Hb^{\tilde
  r-K,l-\delta},\Hb^{\tilde r-K-2,l-\delta})$ and $u_j$ converges to $u$
(thus is bounded) in $\Hb^{\tilde
  r-K,l-\delta}$. Thus, $P(0)u=0$, so $u$ is a non-trivial element of the
nullspace of $P(0)$ on $\Hb^{\tilde r,l}$, which contradicts our
assumptions. This proves \eqref{eq:res-b-est}, and thus the
proposition.
\end{proof}

Since the nullspace of (the elliptic in $x^2\Diffb^2(X)$!) $P(0)$ is independent of the differential order
$\tilde r$, this immediately implies our main theorem, which we now
state in a slightly strengthened version:

\begin{thm}\label{thm:main-improved}
Suppose that $\tilde r$ is as in
Proposition~\ref{prop:symbolic-b-est-more}:
$\tilde r$ is a homogeneous degree $0$
function on $\Tb^*X\setminus o$ (i.e.\ a smooth function on $\Sb^*X$)
with $\tilde r>\frac{1}{2}-(l+1)$ at the source, $\{\taub>0,\mub=0\}$,
$\tilde r<\frac{1}{2}-(l+1)$ at the sink, $\{\taub<0,\mub=0\}$, and 
$
-(\taub^2+|\mub|^2)^{-1/2}x^{-2}\sH_{p_\pa}\tilde r
$
is non-negative.

Suppose also that $|l+1|<\frac{n-2}{2}$ and $P(0):\Hb^{\infty,l}\to\Hb^{\infty,l+2}$ has trivial
nullspace.

Then there exist
$\sigma_0>0$ and $C>0$ such that for $|\sigma|<\sigma_0$,
$\im(\sigma^2)\geq 0$,
\begin{equation}\label{eq:res-b-est-2}
\|P(\sigma)^{-1}f\|_{\Hb^{\tilde r,l}}\leq
C\|f\|_{\Hb^{\tilde r-1,l+2}}
\end{equation}
for $f\in
\Hb^{\tilde r-1,l+2}$.

The analogous result also holds with the source and sink switched, and the
positivity of $-(\taub^2+|\mub|^2)^{-1/2}x^{-2}\sH_{p_\pa}\tilde r$ is
replaced by that of $(\taub^2+|\mub|^2)^{-1/2}x^{-2}\sH_{p_\pa}\tilde
r$ provided one also changes the condition on $\sigma^2$ to $\im(\sigma^2)\leq 0$.
\end{thm}

Thus, it suffices to prove Proposition~\ref{prop:impr-b-est}, i.e.\
\eqref{eq:impr-b-est}, which we means we need to gain decay for the
error term of \eqref{eq:basic-b-est} on the right hand side.

In general, decay is
controlled by the normal operator of a b-differential operator,
which arises by setting $x=0$ in its coefficients after factoring out an
overall weight, and where one thinks of it as acting on functions on
$[0,\infty)_x\times\pa X$, of which $[0,\delta_0)_x\times\pa X$ is
identified with a neighborhood of $\pa X$ in $X$. Now, $P(\sigma)\in\Psib^{-2,0}$ only, and in the usual
sense the normal operator is simply $-\sigma^2$ as
$P(\sigma)+\sigma^2\in\Psib^{-2,2}$. Thus, we instead consider the
`effective normal operator', which from \eqref{eq:P-sigma-form}, namely
\begin{equation*}
P(\sigma)=P(0)+\sigma Q-\sigma^2,\qquad P(0)\in x^2\Diffb^2(X),\ Q\in
x^3\Diffb^1(X),
\end{equation*}
is
$$
\tilde N(P(\sigma))=N(P(0))-\sigma^2,
$$
so
$$
P(\sigma)-\tilde N(P(\sigma))\in x^{2+\delta}S^0\Diffb^2(X).
$$
For a normally long-range asymptotically Euclidean metric, we have
\begin{equation}\label{eq:Delta-sc-def}
\tilde N(P(\sigma))=\Delta_{\scl}-\sigma^2,\ \text{where}\ \Delta_{\scl}=x^{n+1} D_x x^{-n-1} x^4D_x+x^2\Delta_{\pa X}
\end{equation}
is the model scattering Laplacian at infinity.

Typically in b-problems one proceeds by obtaining separate principal
symbol and normal operator estimates. In the present case, in Section~\ref{subsec:positivity-reduction}, we prove a
normal operator estimate with the special case of the differential
order $\hat r=\hat r_\pm(\beta)$:

\begin{prop}\label{prop:eff-norm-op-b-est-base}
For $S\subset [0,\infty)$ compact,
the effective normal operator $\tilde N(P(\sigma))$ satisfies
\begin{equation}\label{eq:eff-norm-op-b-est-base}
\|v\|_{\Hb^{\hat r,l}}\leq C\|\tilde N(P(\sigma))v\|_{\Hb^{\hat r-1,l+2}},
\end{equation}
with $\hat r=\hat r_\pm(\beta)$, $\beta>0$, as in \eqref{eq:b-symbol-weight-choice-special}.
\end{prop}

Due to the lack of ellipticity, namely the loss of one derivative in
\eqref{eq:basic-b-est} on $P(\sigma)u$ relative to the elliptic shift
of the order of the norm on the left hand side,
this will give a somewhat weaker estimate than needed for
Proposition~\ref{prop:impr-b-est} even in the special case $\tilde
r=\hat r$, but we improve on it by using the symbolic estimate yet
again. In the special case $\tilde
r=\hat r$ this could be avoided by doing {\em both} the symbolic and
normal operator positivity argument {\em in a single step} (rather
than broken up into two steps, as done in the previous and the current
sections), but as we
intend to
prove the general
differential order $\tilde r$ version of the main theorem in
Theorem~\ref{thm:main-improved}, yet prove (and use) the normal operator
positivity only for the special differential order $\hat r$, we give a
unified, but slightly involved,
functional analytic
argument below.

We now proceed to prove Proposition~\ref{prop:impr-b-est}, given Proposition~\ref{prop:eff-norm-op-b-est-base}.

\begin{proof}[Proof of Proposition~\ref{prop:impr-b-est} given Proposition~\ref{prop:eff-norm-op-b-est-base}.]

Suppose \eqref{eq:eff-norm-op-b-est-base} holds.

{\em We first show the
estimate \eqref{eq:impr-b-est} of
Proposition~\ref{prop:impr-b-est} under the additional assumption that
$\tilde r>\hat r+1$, where $\hat r=\hat r_+(\beta)$ with $\beta$
fixed.} Notice that this can only be satisfied, in view of the low
regularity radial point estimate, if $\beta>1$, but in that case it
indeed can be satisfied e.g.\ by $\tilde r=\hat r+1+\ep$, $0<\ep<\beta-1$. We point out that under the `analogous'
second part of the statement of the proposition, the same arguments
prove \eqref{eq:impr-b-est} under the additional assumption that
$\tilde r>\hat r_-(\beta)+1$; {\em this will be used later in duality arguments}.

We start with \eqref{eq:basic-b-est}:
\begin{equation*}
\|u\|_{\Hb^{\tilde r,l}}\leq
C(\|P(\sigma)u\|_{\Hb^{\tilde r-1,l+2}}+\|u\|_{\Hb^{-N,l}}),
\end{equation*}
where we may assume that $u\in\dCI(X)$, so membership in various
spaces below is automatic, and where we take $N>\beta$. It thus remains to estimate the error term $\|u\|_{\Hb^{-N,l}}$.

Let $\chi$ be a
cutoff supported in a collar neighborhood of the boundary, identically
$1$ on a smaller neighborhood. Then
$$
\|u\|_{\Hb^{-N,l}}\leq\|\chi u\|_{\Hb^{-N,l}}+\|(1-\chi)u\|_{\Hb^{-N,l}}
$$
shows that it suffices to estimate the $\chi u$ term (for the other is
compactly supported in the interior, so can be absorbed into the
second term of the right hand side of \eqref{eq:impr-b-est}). This in
turn is estimated, thanks to \eqref{eq:eff-norm-op-b-est-base}, by
$$
\|\chi u\|_{\Hb^{-N,l}}\leq\|\chi u\|_{\Hb^{\hat r,l}}\leq C\|\hat N(P(\sigma))(\chi
u)\|_{\Hb^{\hat r-1,l+2}}.
$$
Further (with $l'$ arbitrary in the compactly supported term below)
\begin{equation*}\begin{aligned}
&\|\tilde N(P(\sigma))(\chi u)\|_{\Hb^{\hat r-1,l+2}}\\
&\leq
\|P(\sigma) u\|_{\Hb^{\hat r-1,l+2}}
+\|P(\sigma)(1-\chi)u\|_{\Hb^{\hat r-1,l+2}}+\|(P(\sigma)-\tilde N(P(\sigma)))(\chi u)\|_{\Hb^{\hat r-1,l+2}}\\
&\leq \|P(\sigma) u\|_{\Hb^{\hat r-1,l+2}}+C'\|u\|_{\Hb^{\hat r+1,l'}}+C'\|\chi u\|_{\Hb^{\hat r+1,l-\delta}}\\
&\leq \|P(\sigma) u\|_{\Hb^{\hat r-1,l+2}}+C\|u\|_{\Hb^{\hat r+1,l-\delta}}.
\end{aligned}\end{equation*}
Thus,
\begin{equation*}
\|u\|_{\Hb^{\tilde r,l}}\leq
C(\|P(\sigma)u\|_{\Hb^{\tilde r-1,l+2}}+\|P(\sigma) u\|_{\Hb^{\hat r-1,l+2}}+\|u\|_{\Hb^{\hat r+1,l-\delta}}),
\end{equation*}
As $\hat r-1\leq\tilde r-1$ and $\hat r+1<\tilde r$, this gives
\begin{equation}\begin{aligned}
\|u\|_{\Hb^{\tilde r,l}}&\leq
C(\|P(\sigma)u\|_{\Hb^{\tilde r-1,l+2}}+\|u\|_{\Hb^{\hat
    r+1,l-\delta}})\\
&\leq C\|P(\sigma)u\|_{\Hb^{\tilde r-1,l+2}}+\ep \|u\|_{\Hb^{\tilde
    r,l-\delta}}+ C_\ep\|u\|_{\Hb^{\tilde r-K,l-\delta}},
\end{aligned}\end{equation}
and now the second term on the right hand side can be absorbed into
the left hand side to yield
\begin{equation}\label{eq:impr-b-est-8}
\|u\|_{\Hb^{\tilde r,l}}\leq C(\|P(\sigma)u\|_{\Hb^{\tilde
    r-1,l+2}}+\|u\|_{\Hb^{\tilde r-K,l-\delta}}),
\end{equation}
{\em This proves the estimate \eqref{eq:impr-b-est} of
Proposition~\ref{prop:impr-b-est} in the case of $\tilde
r>\hat r+1$ (which, again, requires $\beta>1$ to be non-vacuous),
completing our first goal.}

As already mentioned at the beginning of the proof, this also proves
the  estimate \eqref{eq:impr-b-est}  in the second part of  Proposition~\ref{prop:impr-b-est} in the case of $\tilde
r>\hat r_-(\beta)+1$.

Before proceeding, we remark that using the regularity estimate
\eqref{eq:basic-b-est}, valid under just the assumption that its right
hand side is finite (i.e.\ that $u$ and $P(\sigma)u$ are in the
appropriate spaces), we in fact obtain that {\em \eqref{eq:impr-b-est-8}
holds if $P(\sigma)u\in\Hb^{\tilde
    r-1,l+2}$ and $u\in\Hb^{\tilde r-K,l}$. Notice that this is a
stronger condition than the right hand side of \eqref{eq:impr-b-est-8}
being finite;} the reason is that while in Section~\ref{sec:symbolic}
we did a regularization argument for the differential order, we do not
perform a similar argument for the decay order in the present section.

\begin{figure}
\begin{center}
\includegraphics[width=100mm]{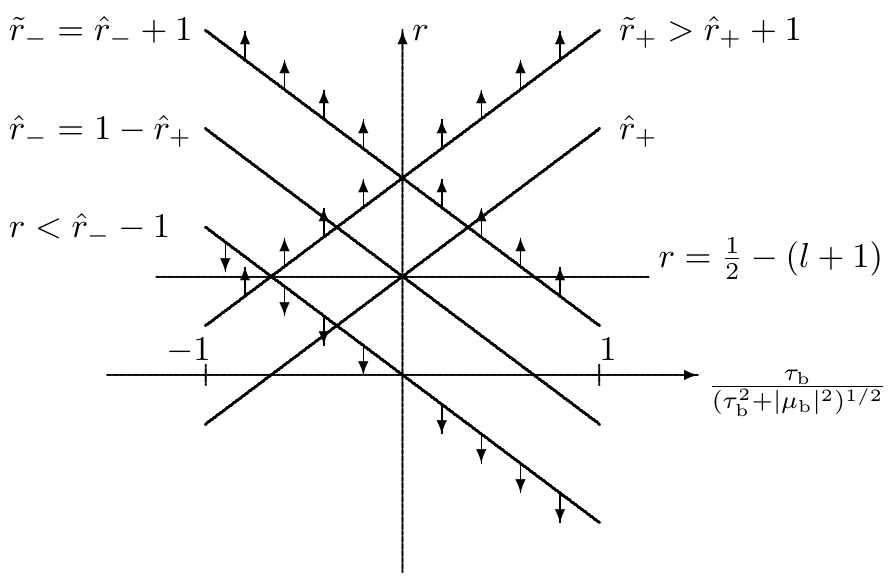}
\end{center}
\caption{The variable order functions used in the proof: in order for
  the symbolic estimate \eqref{eq:basic-b-est} to apply one needs a
  weight $r$ which is monotone and is strictly below, resp.\ above, the
  critical line $r=\frac{1}{2}-(l+1)$ at {\em exactly} one of $-1$ and
  $1$. (Of course, the symbolic order need not be a function of
  $\frac{\taub}{(\taub^2+|\mub|^2)^{1/2}}$ only.) On the
  $\tilde r_+= \hat r_++1$ line, the arrows indicate the region $\tilde r_+>\hat r_++1$ in
  which the graph of $\tilde r_+$ must lie for the first step of our
  proof to apply. Then $r<\hat r_--1$ indicates the region where
  the graph of the order on the {\em dual} space must correspondingly
  lie to create a Fredholm estimate; this is however not admissible
  for our proof of \eqref{eq:impr-b-est-8} in this first step to hold. Finally, the line $\tilde
  r_-=\hat r_-+1$ shows the dual order that is actually used in the
  second step; the argument of the first step applies to this when
  replacing $P(\sigma)$ by $P(\sigma)^*$ and $\hat r_+$ by $\hat r_-$
  --- the arrows indicate the region in which the graph of $\tilde r_-$
  must be for the argument of the first step to apply.}
\label{fig:b-weights}
\end{figure}

Equation \eqref{eq:impr-b-est-8} is only one half of a Fredholm
estimate: it implies closed range and a finite dimensional nullspace.
To see the other half, which gives that the range is finite
codimensional, we need to dualize and work with $P(\sigma)^*$. In order to keep the
notation clear, write $\hat r_\pm$ corresponding to the sign $\pm$ in the definition of
$\hat r$, and similarly write $\tilde r_\pm$ with $\tilde
r_\pm>\hat r_\pm+1$. Now, the dual of $\Hb^{\tilde r_+,l}$ is
$\Hb^{-\tilde r_+,-l}$ thus that of $\Hb^{\tilde
    r_+-1,l+2}$ is $\Hb^{1-\tilde
    r_+,-l-2}$, so correspondingly we would like to have an
estimate
\begin{equation}\label{eq:impr-b-est-dual}
\|v\|_{\Hb^{1-\tilde r_+,-l-2}}\leq C(\|P(\sigma)^*v\|_{\Hb^{-\tilde
    r_+,-l}}+\|v\|_{\Hb^{1-\tilde r_+-K,-l-2-\delta}})
\end{equation}
to complete the argument.

{\em Our next goal is to prove \eqref{eq:impr-b-est-dual} in the case $\tilde
r_+>\hat r_++1$.}
Notice that $-\tilde r_+=(1-\tilde r_+)-1$,
$-l=(-l-2)+2$, so the arithmetic relationship between the spaces on
the left and right hand side of the estimate is the {\em same} as in \eqref{eq:impr-b-est-8}. Now, $-l-2$ satisfies the same assumptions
as $l$, namely $|(-l-2)+1|<\frac{n-2}{2}$. Moreover,
\begin{equation*}
\tilde r_+>\frac{1}{2}-(l+1),\ \text{resp.}\ \tilde
r_+<\frac{1}{2}-(l+1)
\end{equation*}
microlocally at the source or sink are
equivalent to
\begin{equation*} 
1-\tilde
r_+<\frac{1}{2}-(-(l+2)+1),\ \text{resp.}\ 1-\tilde
r_+>\frac{1}{2}-(-(l+2)+1),
\end{equation*}
i.e.\ for $\tilde r_+,l$ satisfying the
conditions for the symbolic estimates for the $+$ choice of sign, the
dual spaces of order $1-\tilde r_+,-l-2$ satisfy the analogous
conditions with the location of high and low regularity reversed. Thus, as
$P(\sigma)^*$ satisfies the same assumptions as $P(\sigma)$, \eqref{eq:impr-b-est-dual} is
indeed the same kind of estimate as \eqref{eq:impr-b-est-8}, but with monotonicity direction (increase/decrease) along the $H_p$
flow is reversed. The only issue, due to which \eqref{eq:impr-b-est-8} (with $P(\sigma)^*$ in place of $P(\sigma)$)
does not immediately yield \eqref{eq:impr-b-est-dual}, is that in \eqref{eq:impr-b-est-8}
(now with $P(\sigma)^*$ in place of $P(\sigma)$) we had the additional
restriction $\tilde r>\hat r+1$ which for
\eqref{eq:impr-b-est-dual} would require $1-\tilde r_+>\hat
r_-+1$, i.e.\ as $1-\hat r_+=\hat r_-$, the requirement translates to $\tilde r_+<\hat r_+-1$,
which is incompatible with the restriction $\tilde r_+>\hat r_++1$
for \eqref{eq:impr-b-est-8}  applying (to $P(\sigma)$) with $\tilde
r=\tilde r_+$, i.e.\ \eqref{eq:impr-b-est-8} does not directly give
matching semi-Fredholm estimates.

To fix this, for $\beta>1$, we apply \eqref{eq:impr-b-est-8} to $P(\sigma)^*$ with $\tilde r=\tilde
r_-=\hat r_-+1+\ep=(1-\hat r_+)+1+\ep=2+\ep-\hat r_+$, $0<\ep<\beta-1$, and with $l$ replaced by
$-l-2$. This satisfies all the requirements for
\eqref{eq:impr-b-est-8}, and gives
for $v\in \Hb^{\tilde r_--K,-l-2}$ with $P(\sigma)^*v\in \Hb^{\tilde
    r_--1,-l}$ (see the remarks in the second paragraph after \eqref{eq:impr-b-est-8})
that
\begin{equation}\label{eq:impr-b-est-dual-actual}
\|v\|_{\Hb^{\tilde r_-,-l-2}}\leq C(\|P(\sigma)^*v\|_{\Hb^{\tilde
    r_--1,-l}}+\|v\|_{\Hb^{\tilde r_--K,-l-2-\delta}}).
\end{equation}
Notice that $\tilde r_-\geq
(1-\tilde r_+)+2$, so this is in {\em stronger} spaces (in terms of
differential order) than the
desired \eqref{eq:impr-b-est-dual}. This estimate gives, via
approximating the compact inclusion map by a finite rank map, that
there are finitely many linear functionals $\ell_1,\ldots,\ell_M\in
(\Hb^{\tilde r_-,-l-2})^*=\Hb^{-\tilde r_-,l+2}$ (identified via the
sesquilinear $L^2$-pairing) such that
\eqref{eq:impr-b-est-dual-actual} can be replaced (under unchanged
conditions for $v$) by
\begin{equation}\label{eq:impr-b-est-dual-finite-rank}
\|v\|_{\Hb^{\tilde r_-,-l-2}}\leq C(\|P(\sigma)^*v\|_{\Hb^{\tilde
    r_--1,-l}}+\sum_{j=1}^M|\ell_j(v)|);
\end{equation}
indeed one may assume (by approximating the finite rank operator) that $\ell_j\in\dCI(X)$.

In order to proceed, it is useful to rewrite
\eqref{eq:impr-b-est-dual-finite-rank} as an estimate without an error
term for a slightly different operator. So let, with $R$ in the
differential order at this point arbitrary,
$$
\tilde P(\sigma): \Hb^{R,l}\oplus\Cx^M\to\Hb^{R-2,l+2}
$$
be given by
$$
\tilde P(\sigma)(u,c)=P(\sigma)u+\sum c_j\ell_j,
$$
so the formal adjoint is
$$
\tilde P(\sigma)^*v=(P(\sigma)^*v,\ell(v)),
$$
where $\ell=(\ell_1,\ldots,\ell_M)$. Our estimate is, for any $v\in \Hb^{\tilde r_--K,-l-2}$
with $\tilde P(\sigma)^*v\in \Hb^{\tilde
    r_--1,-l}\oplus\Cx^M$,
\begin{equation}\label{eq:impr-b-est-dual-finite-rank-compl}
\|v\|_{\Hb^{\tilde r_-,-l-2}}\leq C\|\tilde P(\sigma)^*v\|_{\Hb^{\tilde
    r_--1,-l}\oplus\Cx^M}.
\end{equation}
In particular, this is valid for all $v\in \Hb^{\tilde r_-+1,-l-2}$,
the set of which is dense in $\Hb^{\tilde r_-,-l-2}$.
Thus, by duality, namely Hahn-Banach in a Hilbert space-setting, defining a continuous linear
functional on the range of $\tilde P(\sigma)^*$ on $\Hb^{\tilde
  r_-+1,-l-2}$, extended uniquely to the closure of the range in $\Hb^{\tilde
    r_--1,-l}\oplus\Cx^M$ (without
changing the constant of the estimate), and then the whole space using
an orthogonal projection (thus again without changing the constant), one can solve
$\tilde P(\sigma)(u,c)=f\in \Hb^{-\tilde r_-,l+2}$ with a uniform bound:
$$
\|(u,c)\|_{\Hb^{1-\tilde r_-,l}\oplus\Cx^M}\leq C\|f\|_{\Hb^{-\tilde r_-,l+2}},
$$
with $C$ independent of $\sigma$ in a compact set. But this means that
$$
P(\sigma)u=f-\sum_j c_j\ell_j
$$
and
\begin{equation}\label{eq:complemented-est-8}
\|u\|_{\Hb^{1-\tilde r_-,l}}\leq C\|f\|_{\Hb^{-\tilde r_-,l+2}},\qquad
|c_j|\leq C\|f\|_{\Hb^{-\tilde r_-,l+2}}
\end{equation}
(with $C$ independent of $f$ and of $\sigma$).

Now, if $f\in
\Hb^{\tilde r_+-1,l+2}$ (note that
$\Hb^{\tilde r_+-1,l+2}\subset\Hb^{-\tilde r_-,l+2}$ as
$\tilde r_-\geq
(1-\tilde r_+)+2\geq 1-\tilde r_+$)
then the regularity estimates of Section~\ref{sec:symbolic}, see
Proposition~\ref{prop:symbolic-b-est-more}, which
required the regularization argument, apply, and as $\sum c_j\ell_j\in\dCI(X)$, show that $u\in \Hb^{\tilde
  r_+,l}$, and indeed
\begin{equation*}\begin{aligned}
\|u\|_{\Hb^{\tilde
  r_+,l}}&\leq C'(\|f-\sum c_j\ell_j\|_{\Hb^{\tilde r_+-1,l+2}}+\|u\|_{\Hb^{1-\tilde
    r_-,l}})\\
&\leq C'(\|f\|_{\Hb^{\tilde r_+-1,l+2}}+\sum_j|c_j|+\|u\|_{\Hb^{1-\tilde
    r_-,l}}).
\end{aligned}\end{equation*}
Hence, by \eqref{eq:complemented-est-8}
$$
\|u\|_{\Hb^{\tilde
  r_+,l}}\leq C''(\|f\|_{\Hb^{\tilde r_+-1,l+2}}+\|f\|_{\Hb^{-\tilde
  r_-,l+2}})\leq C''' \|f\|_{\Hb^{\tilde r_+-1,l+2}}
$$
with $C'''$ independent of $f$ and $\sigma$.

Now, for $v\in\Hb^{1-\tilde r_+,-l-2}$ with $P(\sigma)^*v\in\Hb^{-\tilde r_+,-l}$,
\begin{equation*}\begin{aligned}
|\langle f,v\rangle|&=|\langle (u,c),\tilde P(\sigma)^* v\rangle|\leq \|(u,c)\|_{\Hb^{\tilde
  r_+,l}}\|\tilde P(\sigma)^* v\|_{\Hb^{-\tilde
  r_+,-l}}\\
&\leq C''' \|f\|_{\Hb^{\tilde r_+-1,l+2}}\|\tilde P(\sigma)^* v\|_{\Hb^{-\tilde
  r_+,-l}},
\end{aligned}\end{equation*}
where the equality of the first two expressions is justified by a
simple regularization argument in the decay order for $v$.
Since $f$ is arbitrary in $\Hb^{\tilde r_+-1,l+2}$, this gives (for $v\in\Hb^{1-\tilde r_+,-l-2}$ with $P(\sigma)^*v\in\Hb^{-\tilde r_+,-l}$)
\begin{equation}\begin{aligned}\label{eq:impr-b-est-dual-proved}
\|v\|_{\Hb^{1-\tilde r_+,-l-2}}&\leq C''' \|\tilde P(\sigma)^* v\|_{\Hb^{-\tilde
  r_+,-l}}\\
&\leq C(\|P(\sigma)^* v\|_{\Hb^{-\tilde
  r_+,-l}}+\sum_j|\ell(v_j)|),
\end{aligned}\end{equation}
which immediately implies the desired estimate \eqref{eq:impr-b-est-dual}.
{\em Thus, we achieved our second goal and proved \eqref{eq:impr-b-est-dual} in the case $\tilde
r_+>\hat r+1$.} In particular, {\em given \eqref{eq:eff-norm-op-b-est-base}, this
  proves Proposition~\ref{prop:impr-b-est} in case $\tilde r>\hat r+1$.}

It remains to extend to range of $\tilde r$ to all functions allowed
in the statement of Proposition~\ref{prop:impr-b-est}. {\em The final
  step of the proof of Proposition~\ref{prop:impr-b-est} is to remove
  the restriction $\tilde r>\hat r+1$.}

We break this final part into two steps. {\em We next prove
  Proposition~\ref{prop:impr-b-est} in case $\tilde r<\hat r-1$.}

As we already discussed, $P(\sigma)$ and $P(\sigma)^*$ have analogous
properties as pseudodifferential operators, so we can interpret
\eqref{eq:impr-b-est-dual} as an estimate for $P(\sigma)$,
namely
\begin{equation*}
\|u\|_{\Hb^{1-\tilde r_+,-l-2}}\leq
C(\|P(\sigma)u\|_{\Hb^{-\tilde r_+,-l}}+\|u\|_{\Hb^{1-\tilde
    r_+-K,-l-2-\delta}}),
\end{equation*}
which, with $\tilde r_-=1-\tilde r_+$ is the estimate
\begin{equation*}
\|u\|_{\Hb^{\tilde r_-,-l-2}}\leq
C(\|P(\sigma)u\|_{\Hb^{\tilde r_--1,-l}}+\|u\|_{\Hb^{\tilde
    r_--K,-l-2-\delta}});
\end{equation*}
now $\tilde r_-<-\hat r_+=\hat r_--1$, i.e.\ (recall that the role
of the $\pm$ sign choices is completely symmetric) the estimate
\eqref{eq:impr-b-est}
\begin{equation}\label{eq:impr-b-est-low}
\|u\|_{\Hb^{\tilde r,l}}\leq
C(\|P(\sigma)u\|_{\Hb^{\tilde r-1,l+2}}+\|u\|_{\Hb^{\tilde
    r-K,l-\delta}});
\end{equation}
in the low-regularity case $\tilde r<\hat
r-1$. Since in this case we already have the dual semi-Fredholm
estimate, {\em this proves Proposition~\ref{prop:impr-b-est} if $\tilde r<\hat
r-1$.}

{\em We now remove all restrictions on $\tilde r$ beyond those of the statement of
  Proposition~\ref{prop:impr-b-est}.}
With $\tilde r=\tilde r_+$ arbitrary in Proposition~\ref{prop:impr-b-est}, we are going to go through the first part of the argument again,
starting with
\eqref{eq:basic-b-est}:
\begin{equation*}
\|u\|_{\Hb^{\tilde r_+,l}}\leq
C(\|P(\sigma)u\|_{\Hb^{\tilde r_+-1,l+2}}+\|u\|_{\Hb^{-N,l}}),
\end{equation*}
with $N>\beta+1$, $\beta>1$ and $-N<\tilde r_+$. But now we estimate the error term
$\|u\|_{\Hb^{-N,l}}$ using \eqref{eq:impr-b-est-low} with $-N\leq\tilde
r<\min(\hat r_+(\beta)-1,\tilde r_+)$ to obtain
\begin{equation*}
\|u\|_{\Hb^{\tilde r_+,l}}\leq
C(\|P(\sigma)u\|_{\Hb^{\tilde r_+-1,l+2}}+\|P(\sigma)u\|_{\Hb^{\tilde
    r-1,l+2}}+\|u\|_{\Hb^{\tilde r-K,l-\delta}}),
\end{equation*}
and thus the general case of the estimate \eqref{eq:impr-b-est} of
Proposition~\ref{prop:impr-b-est}. {\em Since this also applies to
$P(\sigma)^*$ on the dual spaces, this completes the proof of Proposition~\ref{prop:impr-b-est}.}
\end{proof}

\subsection{Reduction of the normal operator estimate, Proposition~\ref{prop:eff-norm-op-b-est-base}, to a positivity computation}\label{subsec:positivity-reduction}
Thus, it suffices to prove
Proposition~\ref{prop:eff-norm-op-b-est-base}, which we recall is the estimate
\begin{equation}\label{eq:eff-norm-op-b-est-base-recall}
\|v\|_{\Hb^{\hat r,l}}\leq C\|\tilde N(P(\sigma))v\|_{\Hb^{\hat r-1,l+2}},
\end{equation}
where
\begin{equation}\label{eq:b-symbol-weight-choice-restate}
\hat r=\frac{1}{2}-(l+1)\pm\beta\frac{\taub}{(\taub^2+|\mub|^2)^{1/2}},\qquad\beta>0.
\end{equation}

For this purpose
it is more convenient to work with $L^2_{\bl}$, so we set $\Hbb$
to be the b-Sobolev space relative to $L^2_{\bl}$, here this really is
of interest in $[0,\infty)\times\pa X$, with density
$\frac{dx}{x}\,dg_{\pa X}$. Our argument will
proceed by showing that a modification of the commutator \eqref{eq:commutator-basic-b-est} in $\Psib^{2\hat r,2l}$ that we considered in the proof of
\eqref{eq:basic-b-est} actually can be arranged to have a positive normal operator as
well, relative to $L^2$, which statement is equivalent to saying that
$x^{-n}$ times this normal operator is positive relative to
$L^2_{\bl}$ since the quadratic form on $L^2_\bl$ is $\langle x^{n}\cdot,\cdot\rangle_g$. This will immediately imply
\eqref{eq:eff-norm-op-b-est-base-recall}. Now, if we write
$$
N(A)=x^{n/2}\tilde A
x^{-n/2} x^{-l-1},
$$
$\tilde A\in\Psib^{\hat r-1/2,0}$ on the model space
$[0,\infty)\times\pa X$, dilation invariant, then we need to compute
\begin{equation}\begin{aligned}\label{eq:normal-op-b-commutator}
&i(\tilde N(P(\sigma)^*-P(\sigma)))N(A^*A)+i[\tilde
N(P(\sigma)),N(A^*A)]\\
&\qquad=-2\im(\sigma^2)x^{-n/2-l-1}\tilde A^*x^n\tilde A
x^{-n/2-l-1} +i[\Delta_{\scl},x^{-n/2-l-1}\tilde A^* x^n\tilde A
x^{-n/2-l-1}],
\end{aligned}\end{equation}
to the extent that we can show its positivity (or negativity) on
$L^2$ (more precisely a lower bound for $\pm$ this operator by $C x^{-2l}$, $C>0$). As mentioned above, this is equivalent to
$$
-2\im(\sigma^2)x^{-n}x^{-n/2-l-1}\tilde A^*x^n\tilde A
x^{-n/2-l-1} +ix^{-n}[\Delta_{\scl},x^{-n/2-l-1}\tilde A^*x^n\tilde A
x^{-n/2-l-1}]
$$
having the corresponding sign
on $L^2_{\bl}$, with the explicit lower bound now being $C x^{-2(l+n/2)}$. The reason for the extra factor $x^{-n/2}$ conjugating
$\tilde A$ in our definition of $\tilde A$ is that {\em if $\tilde A$ is symmetric relative to the
$L^2_\bl$ inner product, as we arrange to simplify our arguments, then
$x^{n/2}\tilde A x^{-n/2}$ is symmetric with respect to the
$L^2$-inner product}, so in fact we have to compute
\begin{equation}\begin{aligned}\label{eq:twisted-normal-comm-comp}
&-2\im(\sigma^2)x^{-n}x^{-l-1}(x^{n/2}\tilde A x^{-n/2})^2
x^{-l-1}+ix^{-n}[\Delta_{\scl},x^{-l-1}(x^{n/2}\tilde A x^{-n/2})^2
x^{-l-1}]\\
&=-2\im(\sigma^2)x^{-n/2-l-1}\tilde A^2
x^{-n/2-l-1}+ix^{-n}[\Delta_{\scl},x^{n/2-l-1}\tilde A^2 
x^{-n/2-l-1}].
\end{aligned}\end{equation}
Notice that the first term here is a negative operator (in an
indefinite sense) if
$\im(\sigma^2)\geq 0$, and a positive operator if $\im(\sigma^2)\leq
0$, so if in the first case we arrange that the second term is
negative definite (in an appropriate sense), while in the second case
we arrange that the second term is positive definite, in an estimate
one may simply drop the first term, i.e.\ allowing complex $\sigma$
with $\sigma^2$ of the correct imaginary part does not affect the
argument below.

Let
$$
\Delta_{\bl}=x^{-(n+2)/2}\Delta_{\scl}x^{(n-2)/2}\in\Diffb^2,
$$
as $x^{-1}\Delta_{\scl} x^{-1}$ is symmetric with respect to the
$L^2$-inner product, $\Delta_{\bl}$ is symmetric with respect to the
$L^2_\bl$ inner product. Explicitly,
\begin{equation}\begin{aligned}\label{eq:b-Lap}
\Delta_{\bl}&=x^{n/2}D_x x^{-n+3}D_xx^{n/2-1}+\Delta_{\pa X}\\
&=(D_x x+i\frac{n}{2}) x^{-n/2+2}D_xx^{n/2-1}+\Delta_{\pa X}\\
&=(D_x x+i\frac{n}{2}) (xD_x-i\frac{n-2}{2})+\Delta_{\pa X}\\
&=(xD_x)^2 +\Delta_{\pa X}+\Big(\frac{n-2}{2}\Big)^2;
\end{aligned}\end{equation}
notice that this is a positive definite operator on $L^2_\bl$ for
$n\geq 3$, since on the
Mellin transform side it is multiplication by a positive function.

Now
\begin{equation*}\begin{aligned}
&ix^{-n}[\Delta_{\scl},x^{n/2-l-1}\tilde A^2 
x^{-n/2-l-1}]\\
&\qquad=ix^{-n}x^{\frac{n+2}{2}}\Delta_\bl x^{-\frac{n-2}{2}}x^{n/2-l-1}\tilde A^2 
x^{-n/2-l-1}\\
&\qquad\qquad-ix^{-n}x^{n/2-l-1}\tilde A^2 
x^{-n/2-l-1}x^{\frac{n+2}{2}}\Delta_\bl x^{-\frac{n-2}{2}},
\end{aligned}\end{equation*}
whose positivity is equivalent to that of the operator obtained by
multiplying from both sides by $x^{l+n/2}$ (chosen to make the total weight
$x^0$, this also changes the desired lower bound to a positive constant):
$$
ix^{l+1}\Delta_\bl x^{-l}\tilde A^2 
x^{-1}-ix^{-1}\tilde A^2 
x^{-l}\Delta_\bl x^{l+1}\in\Psib^{2\hat r,0}.
$$
But $\Delta_\bl$ and $\tilde A$ are both dilation invariant
pseudodifferential operators, so the effect of conjugating them by
$x^k$ (i.e.\ multiplying by this from the right, and by its inverse
from the left) is replacing $xD_x$ by $xD_x-ik$, or on the Mellin
transform side replacing $\taub$ by $\taub-ik$. Writing such a change by
affixing $(\cdot-ik)$ to the operator, we need to compute
\begin{equation*}\begin{aligned}
&ix^{l+1}\Delta_\bl x^{-l}\tilde A^2 
x^{-1}-ix^{-1}\tilde A^2 
x^{-l}\Delta_\bl x^{l+1}\\
&\qquad=i\Delta_\bl(\cdot+i(l+1))\tilde A(\cdot+i)^2-i \tilde A(\cdot-i)^2\Delta_\bl(\cdot-i(l+1)).
\end{aligned}\end{equation*}
Now all operators on the right hand side are multiplication operators
on the Mellin transform side (no $x$ dependence), so {\em if we choose
  $\tilde A$ to depend on $y$ and its b-dual variables only through
  $\Delta_{\pa X}$, and still symmetric with respect to the
  $L^2_\bl$-inner product}, then the positivity or negativity of this
expression is a commutative calculation, and the expression can be
written as (with imaginary part on the second line meaning
skew-adjoint part, which becomes the imaginary part on the Mellin
transform/spectral side)
\begin{equation}\begin{aligned}\label{eq:normal-op-computation}
&i \tilde A(\cdot+i)^2\Delta_\bl(\cdot+i(l+1))-i \tilde
A(\cdot-i)^2\Delta_\bl(\cdot-i(l+1))\\
&=-2\im \tilde A(\cdot+i)^2\Delta_\bl(\cdot+i(l+1)).
\end{aligned}\end{equation}

Notice that if \eqref{eq:normal-op-computation} is positive (with the
negative case essentially the same), i.e.\ on
the Mellin transform side, replacing $\Delta_{\pa X}$ by its spectral
parameter $\lambda\geq 0$, is given by multiplication by a positive
function, which is the square of a positive
elliptic symbol $\tilde b$ of
order $\hat r$ in terms of $(\taub,\sqrt{\lambda})$ (with the square
root present due to $\Delta_{\pa X}$ being second order), then one has (for $\im(\sigma^2)\leq 0$; or $\im(\sigma^2)\geq
0$ in the negative case)
$$
x^{-n}x^{l+n/2}\Big((\tilde N(P(\sigma)^*-P(\sigma)))N(A^*A)+[\tilde
N(P(\sigma)),N(A^*A)]\Big) x^{l+n/2}\geq \tilde B^{*_\bl}\tilde B,
$$
with $\tilde B\in\Psib^{\hat r,0}$, with $\tilde B^{*_\bl}$ the
adjoint in $L^2_\bl$ (and is $=\tilde B$), namely $\tilde B$ has
Mellin transformed normal operator given by the functional calculus of
$\Delta_{\pa X}$ for $\tilde b$. Here the inequality is due to both merely
using a lower bound for \eqref{eq:normal-op-computation} and to dropping the $\im(\sigma^2)$ term,
which has a sign matching that of \eqref{eq:normal-op-computation}. Thus, with the inner product now the $L^2$-based, so
$\tilde B^*=x^n\tilde B^{*_\bl} x^{-n}=x^n\tilde B x^{-n}$, and with
$B=x^{n/2}\tilde B x^{-l-n/2}$, so $B^*=x^{-l-n/2}\tilde B^*
x^{n/2}=x^{n}x^{-l-n/2}\tilde B^{*_\bl} x^{n/2} x^{-n}$, we have
\begin{equation}\begin{aligned}\label{eq:pos-comm-to-est-normal}
\|B u\|^2&\leq\langle iN(A) u,N(A)\tilde N(P(\sigma))u\rangle-\langle
iN(A)\tilde N(P(\sigma))u,N(A) u\rangle\\
&\leq C\|N(A) u\|_{\Hb^{1/2,-1}}\|N(A)\tilde
N(P(\sigma))u\|_{\Hb^{-1/2,1}}\\
&\leq C\|u\|_{\Hb^{\hat r,l}}\|\tilde
N(P(\sigma))u\|_{\Hb^{\hat r-1,l+2}}.
\end{aligned}\end{equation}
But by the positivity of $\tilde B$, $\|Bu\|$ is equivalent to
$\|u\|_{\Hb^{\hat r,l}}$, so dividing
\eqref{eq:pos-comm-to-est-normal} by the latter proves the weaker estimate
\eqref{eq:eff-norm-op-b-est-base} instead of the desired
\eqref{eq:impr-b-est}.

{\em In summary, we have proved
  \eqref{eq:eff-norm-op-b-est-base-recall}, thus
  Proposition~\ref{prop:eff-norm-op-b-est-base} , and hence Proposition~\ref{prop:impr-b-est}, if we show that
  we can choose $A$ so
  that the positivity of
\eqref{eq:normal-op-computation} holds, with a lower bound by the
square of a positive elliptic symbol of order $\hat r$ on
the Mellin transform/spectral side.}

\subsection{Choice of the operator $\tilde A$ and completion of the
  proof of Proposition~\ref{prop:eff-norm-op-b-est-base}}
In order to obtain a positive \eqref{eq:normal-op-computation},
we arrange below that the principal symbol of $A$ is, for
suitable $\tilde\beta>0$ to be chosen, and with $\psi$ identically $1$
near $0$, of sufficiently small support as in \eqref{eq:b-symbol-choice},
\begin{equation}\begin{aligned}\label{eq:a-choice}
a=x^{-l-1}\psi(x)\exp\Big(&\pm\frac{\tilde\beta}{2}\frac{\taub}{(\taub^2+|\mub|^2)^{1/2}}\\
&\qquad+\Big(\pm\frac{\beta}{2}\frac{\taub}{(\taub^2+|\mub|^2)^{1/2}}\Big)\log\big(\taub^2+|\mub|^2\big)\\
&\qquad-\frac{l+1}{2}\log\big(\taub^2+|\mub|^2\big)\Big),
\end{aligned}\end{equation}
which is
$$
\exp\Big(\pm\frac{\tilde\beta}{2}\frac{\taub}{(\taub^2+|\mub|^2)^{1/2}}\Big)
$$
times the principal symbol used for our symbolic computation in
\eqref{eq:b-symbol-choice} (which uses the same $\tilde r=\hat r$ as here), and is the same as that employed in the more general
computation \eqref{eq:a-choice-mod}. Much as in the more general
$\tilde r$ setting of Proposition~\ref{prop:symbolic-b-est-more}, this extra factor will give us sufficient
flexibility in order to ensure the normal operator positivity. {\em Notice
that it automatically ensures the symbolic positivity of the normal
operator, i.e.\ that it is positive modulo compact terms, and moreover
it can be bounded below by the square of a pseudodifferential operator
modulo compact terms, in accordance with the proof of
Proposition~\ref{prop:symbolic-b-est-more}, cf.\ the discussion around
\eqref{eq:H-p-a-sqrt} for dropping logarithmic terms.}

With this motivation,
in accordance with
\eqref{eq:a-choice}, we may
{\em almost} take $\tilde A$ to be the Mellin conjugate of
\begin{equation}\begin{aligned}\label{eq:tilde-A-choice}
\exp\Big(&\pm\frac{\tilde\beta}{2}\frac{\taub}{(\taub^2+\Delta_{\pa
    X}+\tilde\digamma^2)^{1/2}}\\
&\qquad+\Big(\pm\frac{\beta}{2}\frac{\taub}{(\taub^2+\Delta_{\pa
    X}+\digamma^2)^{1/2}}\Big)\log\big(\taub^2+\Delta_{\pa X}+\digamma^2\big)\\
&\qquad-\frac{l+1}{2}\log\big(\taub^2+\Delta_{\pa 
  X}+\check\digamma^2\big)\Big),
\end{aligned}\end{equation}
where
$$
\tilde\digamma\geq \digamma\geq\check\digamma>1,\ \tilde\beta\geq 0
$$
are parameters, and where the square root and
the logarithm are defined with branch cuts along the negative real
axis, and are real for positive arguments. Here `almost' refers to that
with this choice we do not have an entire (operator-valued) function on the Mellin transform
side, rather simply holomorphic in a strip, of width $2\check\digamma$ around the real
axis; this is easily remedied, see Lemma~\ref{lemma:convolve}. Indeed, in the spectral representation of $\Delta_{\pa X}\geq 0$
it may be replaced by a non-negative spectral parameter $\lambda$, and
then, for complex $\taub$,
$$
\re(\taub^2+\lambda+\check\digamma^2)\geq\check\digamma^2-(\im\taub)^2,
$$
so is positive in $|\im\taub|<\check\digamma$, and similarly for
$\check\digamma$ replaced by $\digamma,\tilde\digamma$, and thus
\eqref{eq:tilde-A-choice} indeed gives a holomorphic function in this
strip. Note that \eqref{eq:normal-op-computation} makes sense for all $\check\digamma>1$ in view of 
the domain of holomorphy. Notice also that \eqref{eq:tilde-A-choice}
{\em does} have the correct principal symbol, i.e.\ behavior as
$(\taub,\mub)\to\infty$, namely that given by \eqref{eq:a-choice}.

While this formula looks complicated, we
mention already now that in the important case of $l+1=0$, which we
discuss separately below, it
simplifies significantly: one can take $\tilde\beta=0$, so only the
second term in the exponent is non-trivial. Otherwise, in general, we
take $\tilde\beta$ to be actually positive, indeed
$$
\tilde\beta=\frac{\pi}{2}\tilde\digamma 
$$
when $l+1\neq 0$. We also remark that while the $\tilde\beta$ term
(being a symbol of order $0$, thus bounded)
does not affect the {\em order} of $\tilde A$, it does affect the {\em
  principal
symbol} when $\tilde\beta>0$; cf.\ Corollary~\ref{cor:tilde-A-princ-symbol}.

We start by remarking that \eqref{eq:tilde-A-choice} is defined via
the functional calculus, which gives a family of pseudodifferential
operators due to the Cauchy-Stokes formula of Helffer-Sj\"ostrand
\cite{Helffer-Sjostrand:Schrodinger} that expresses the function as an
integral involving the resolvent and an almost analytic extension of
the function, see \cite{Hassell-Vasy:Symbolic} for a treatment. Indeed, it is the quantization of a symbol of class
$S^{\tilde r-1/2}_{1-\delta',\delta'}$ for all $\delta'\in(0,1)$, with
principal symbol
\begin{equation}\begin{aligned}\label{eq:tilde-A-choice-symbol}
\exp\Big(&\pm\frac{\tilde\beta}{2}\frac{\taub}{(\taub^2+|\mub|^2+\tilde\digamma^2)^{1/2}}\\
&\qquad+\Big(\pm\frac{\beta}{2}\frac{\taub}{(\taub^2+|\mub|^2+\digamma^2)^{1/2}}\Big)\log\big(\taub^2+|\mub|^2+\digamma^2\big)\\
&\qquad-\frac{l+1}{2}\log\big(\taub^2+|\mub|^2+\check\digamma^2\big)\Big),
\end{aligned}\end{equation}
which is symbolic {\em jointly} in $(\taub,\mub)$ in
$|\im\taub|<\check\digamma$. In fact, if we regard $\check\digamma,\digamma,\tilde\digamma$
as large parameters, notice that the three terms arising by taking the
logarithm have joint symbolic properties. Namely
\begin{equation}\label{eq:log-symbol-tilde}
\frac{\taub}{(\taub^2+|\mub|^2+\tilde\digamma^2)^{1/2}}
\end{equation}
is symbolic jointly in $(\taub,\mub,\tilde\digamma)$ (with $y$ as the
`base variable' which is in a compact set, entering via the dual
metric) of order zero, and similarly
\begin{equation}\label{eq:log-symbol-plain}
\frac{\taub}{(\taub^2+|\mub|^2+\digamma^2)^{1/2}}\log\big(\taub^2+|\mub|^2+\digamma^2\big)
\end{equation}
is symbolic jointly in $(\taub,\mub,\digamma)$ of any positive (indeed logarithmic) 
order, and
\begin{equation}\label{eq:log-symbol-0}
\log\big(\taub^2+|\mub|^2+\check\digamma^2\big),\ 
\end{equation}
is symbolic jointly in $(\taub,\mub,\check\digamma)$ of any positive
order. In addition:

\begin{lemma}\label{lemma:imag-order-minus-one}
In $|\im\taub|<\check\digamma$,
the imaginary part of \eqref{eq:log-symbol-tilde}, resp.\
\eqref{eq:log-symbol-plain}, resp.\ \eqref{eq:log-symbol-0}, is, for all $\delta'>0$, a symbol of order $-1+\delta'$ in
$(\taub,\mub,\tilde\digamma)$, resp.\ $(\taub,\mub,\digamma)$, resp.\
$(\taub,\mub,\check\digamma)$ (in fact, in the first and last cases order $-1$).

Furthermore, on the line $\im\taub=1$, the principal symbol of the
imaginary part of \eqref{eq:log-symbol-tilde}, resp.\ \eqref{eq:log-symbol-plain}, resp.\ \eqref{eq:log-symbol-0}, is
\begin{equation}\label{eq:log-symbol-tilde-pr}
(\taub^2+|\mub|^2+\tilde\digamma^2)^{-3/2}(|\mub|^2+\tilde\digamma^2), \ \text{modulo} \ S^{-3}.
\end{equation}
resp.
\begin{equation}\label{eq:log-symbol-plain-pr}
\big(\taub^2+|\mub|^2+\digamma^2\big)^{-3/2}\big((|\mub|^2+\digamma^2)\log\big(\taub^2+|\mub|^2+\digamma^2\big)+2\taub^2\big), \ \text{modulo} \ S^{-3+\delta'},
\end{equation}
resp.
\begin{equation}\label{eq:log-symbol-0-pr}
2\taub\big(\taub^2+|\mub|^2+\check\digamma^2\big)^{-1},\ \text{modulo} \ S^{-3}.
\end{equation}
Thus, the principal symbols at $\im\taub=1$, when regarded as a
symbol in $(\taub,\mub)$ (i.e.\ $\check\digamma,\digamma,\tilde\digamma$ are fixed or
in a compact set), are
\begin{equation}\label{eq:log-symbol-tilde-pr-cl}
(\taub^2+|\mub|^2)^{-3/2}|\mub|^2
\end{equation}
resp.
\begin{equation}\label{eq:log-symbol-plain-pr-cl}
\big(\taub^2+|\mub|^2\big)^{-3/2}\big(|\mub|^2\log\big(\taub^2+|\mub|^2\big)+2\taub^2\big),
\end{equation}
resp.
\begin{equation}\label{eq:log-symbol-0-pr-cl}
2\taub\big(\taub^2+|\mub|^2\big)^{-1},
\end{equation}
modulo $S^{-3}$, resp.\ $S^{-3+\delta'}$, resp.\ $S^{-3}$.
\end{lemma}

\begin{rem}
The relevance of $\im\taub=1$ is due to the $+i$ in
\eqref{eq:normal-op-computation}.

In addition, as the principal symbols already indicate, in fact the
symbol is only logarithmically bigger than one of order $-1$.
\end{rem}

\begin{proof}
Since the imaginary part vanishes for real $\taub$ for all these
holomorphic functions, we express it as
the integral, from the real axis along the imaginary direction, of
their derivative; this is an integral over an interval of bounded
length ($=1$). But the derivative is a symbol of order
$-1+\delta'$ for all $\delta'>0$, so
the conclusion of being a symbol of this order follows immediately. The actual
principal symbol arises by simply integrating the principal symbol of
the derivative. To see the more
precise conclusion regarding the error, namely that it is (almost) two
orders lower than the principal term, note that for $g$ real on the
reals, writing $\pa_{\im\taub}=i\pa_{\taub}$ for the derivative
along the imaginary direction, by Taylor's formula,
\begin{equation}\begin{aligned}\label{eq:symm-Taylor}
&2i\im
g(\taub+i)=g(\taub+i)-g(\taub-i)=(g(\taub+i)-g(\taub))-(g(\taub-i)-g(\taub))\\
&\qquad=\big(\pa_{\im\taub}g(\taub)+\frac{1}{2}\pa^2_{\im\taub}g(\taub)+\frac{1}{2}\int_0^1
(1-s)^2\, \pa^3_{\im\taub}g(\taub+is)\,ds\big)\\
&\qquad\qquad-\big(-\pa_{\im\taub}g(\taub)+\frac{1}{2}\pa^2_{\im\taub}g(\taub)-\frac{1}{2}\int_0^1
(1-s)^2\, \pa^3_{\im\taub}g(\taub-is)\,ds\big)\\
&\qquad=2\big(\pa_{\im\taub}g(\taub)+\frac{1}{2}\int_0^1
(1-s)^2\, (\pa^3_{\im\taub}g(\taub+is)+\pa^3_{\im\taub}g(\taub-is))\,ds,
\end{aligned}\end{equation}
and the third derivatives listed are in $S^{-3+\delta'}$ in all cases
and in $S^{-3}$ in the first and last cases.
\end{proof}

The aforementioned limited domain of holomorphy actually suffices for the argument by choosing sufficiently
large $\check\digamma>1$ and working with a large b-calculus, with still
sufficient decay at the left and right boundaries (given by $\pa X$)
so that the departure from the small calculus is irrelevant. However,
it is in any case straightforward to fix this absence of being entire:
simply convolve \eqref{eq:tilde-A-choice} with a Gaussian,
$\frac{1}{\sqrt{\pi s}}e^{-\taub^2/(2s)}$, which {\em does not change the
principal symbol}:

\begin{lemma}\label{lemma:convolve}
Suppose that with $\tilde A$ replaced by the Mellin conjugate
  of multiplication by
  \eqref{eq:tilde-A-choice}, \eqref{eq:normal-op-computation} is
  positive, resp.\ negative, with a lower bound given by the square of a symbol on the Mellin transform/spectral side. Then, for sufficiently small $s>0$,
  letting $\tilde A$ be the Mellin conjugate
  of multiplication by the convolution of
  \eqref{eq:tilde-A-choice} with
  $\frac{1}{\sqrt{\pi s}}e^{-\taub^2/(2s)}$, $\tilde
  A\in\Psib^{\tilde r-1/2,0}$ and the same
  positivity, resp.\ negativity, along with the symbolic property, holds.
\end{lemma}

Thus,
in order to prove Proposition~\ref{prop:eff-norm-op-b-est-base}, it suffices to show the positivity or negativity of 
  \eqref{eq:normal-op-computation} for the choice of multiplication by 
  \eqref{eq:tilde-A-choice} as the candidate for the Mellin conjugate 
  of $\tilde A$; the actual choice will arise by a convolution with a Gaussian.

\begin{proof}
On the Schwartz
kernel side, with boundary defining functions $x,x'$ in the left and
right factors, the convolution by $\frac{1}{\sqrt{\pi s}}e^{-\taub^2/(2s)}$ corresponds to multiplication by a function which,
being a Gaussian in terms of $\log(x/x')$, is
superexponentially decaying, and thus in terms of the defining
functions of the two `side faces', locally $x/x'$ and $x'/x$ in the
regions where these are bounded, is superpolynomially decaying. (See
Section~\ref{sec:background} for a discussion of the decay of the
residual terms; note that the symbolic term, \eqref{eq:b-symbolic-term}, has compact support in
$t-t'$, thus $x/x'$ and $x'/x$ are bounded above on it. Note also that
such multiplication leaves the principal symbol unaffected as the
multiplication is by a function that is $1$ at the diagonal.) While
this convolution changes \eqref{eq:tilde-A-choice}, the asymptotic
behavior (thus positivity) of the Mellin conjugate of \eqref{eq:normal-op-computation},  as $|(\taub,\lambda)|\to\infty$, is unaffected (for any value
of $s>0$, and indeed in a uniform sense as $s\to 0$), moreover letting
$s\to 0$ the convolution converges to \eqref{eq:tilde-A-choice}
uniformly on compact sets, so in view of the asymptotic positivity,
for sufficiently small $s>0$ the convolution results in a
function that is both entire (with appropriate estimates) and
positive, so that $\tilde A\in\Psib^{\tilde r-1/2,0}$ indeed and the desired
positivity holds.
\end{proof}

With the choice \eqref{eq:tilde-A-choice} on the Mellin transform
side, we work on the spectral side of $\Delta_{\pa X}$, and we
may replace the latter by its spectral parameter $\lambda\geq
0$. Then $\tilde A$ becomes a multiplication operator by
\begin{equation}\begin{aligned}\label{eq:tilde-A-choice-mult}
f(\taub)=f_\lambda(\taub)=\exp\Big(&\pm\frac{\tilde\beta}{2}\frac{\taub}{(\taub^2+\lambda+\tilde\digamma^2)^{1/2}}\\
&\qquad+\Big(\pm\frac{\beta}{2}\frac{\taub}{(\taub^2+\lambda+\digamma^2)^{1/2}}\Big)\log\big(\taub^2+\lambda+\digamma^2\big)\\
&\qquad -\frac{l+1}{2}\log\big(\taub^2+\lambda+\check\digamma^2\big)\Big),
\end{aligned}\end{equation}
where we consider $\lambda$ a non-negative parameter and suppress it in
the notation, together with the other parameters
$\beta,\tilde\beta,\digamma,\tilde\digamma,\check\digamma$.  {\em In
  order to show the positivity, resp.\ negativity, of
  \eqref{eq:normal-op-computation},
by \eqref{eq:normal-op-computation} and \eqref{eq:b-Lap}, it suffices
to show that
\begin{equation}\label{eq:commutator-imag-version}
-2\im \Big(f(\taub+i)^2\Big((\taub+i(l+1))^2+\lambda+\Big(\frac{n-2}{2}\Big)^2\Big)\Big)
\end{equation}
is positive, resp.\ negative, (everywhere) for appropriate choices of the
constants $\tilde\beta$, $\check\digamma,\digamma$ and $\tilde\digamma$.} This can be
achieved, via taking logarithms, by showing that
\begin{equation}\begin{aligned}\label{eq:tilde-A-choice-mult-log}
\im\Big(&\pm\tilde\beta\frac{\taub+i}{((\taub+i)^2+\lambda+\tilde\digamma^2)^{1/2}}\Big)\\
&\qquad+\im\Big(\Big(\pm\beta\frac{\taub+i}{((\taub+i)^2+\lambda+\digamma^2)^{1/2}}\Big)\log\big((\taub+i)^2+\lambda+\digamma^2\big)\Big)\\
&\qquad\qquad+\im\Big(-(l+1) \log\big((\taub+i)^2+\lambda+\check\digamma^2\big)\Big)\\
&\qquad\qquad\qquad+\im\log\Big((\taub+i(l+1))^2+\lambda+\Big(\frac{n-2}{2}\Big)^2\Big)
\end{aligned}\end{equation}
is in $(0,\pi)$ (for the negative sign conclusion in
\eqref{eq:commutator-imag-version}, corresponding to $\beta>0$), resp.\ in $(-\pi,0)$ (for the positive
sign conclusion in \eqref{eq:commutator-imag-version}, corresponding
to $\beta<0$).

For the last term, observe here that:

\begin{lemma}\label{lemma:complex-Lap-arg}
Suppose $|l+1|<\frac{n-2}{2}$. Let $\lambda=\nu^2$.
Then the last term of \eqref{eq:tilde-A-choice-mult-log},
\begin{equation}\label{eq:tilde-A-choice-mult-log-3rd}
\im\log\Big((\taub+i(l+1))^2+\nu^2+\Big(\frac{n-2}{2}\Big)^2\Big),
\end{equation}
is a symbol in $(\taub,\nu)$ of order $-1$, with principal symbol
$$
\frac{2(l+1)\taub}{\taub^2+\nu^2},\ \text{modulo}\ S^{-3}.
$$

In addition, there is
  $\alpha_0\in(0,\pi/2)$ such that \eqref{eq:tilde-A-choice-mult-log-3rd}
lies in $[-\alpha_0,\alpha_0]$.

Furthermore, if $l+1=0$, \eqref{eq:tilde-A-choice-mult-log-3rd} vanishes.
\end{lemma}

\begin{proof}
Since
$$
\re\Big((\taub+i(l+1))^2+\lambda+\Big(\frac{n-2}{2}\Big)^2\Big)\Big)=\taub^2-(l+1)^2+\Big(\frac{n-2}{2}\Big)^2>0,
$$
\eqref{eq:tilde-A-choice-mult-log-3rd}, which is the
argument of the complex number whose real part is displayed, is in
$\big(-\frac{\pi}{2},\frac{\pi}{2}\big)$.

In addition, much like in Lemma~\ref{lemma:imag-order-minus-one}, as
$\log\Big(\taub^2+\nu^2+\Big(\frac{n-2}{2}\Big)^2\Big)$ is
real on the real axis, is holomorphic in $|\im\taub|<\frac{n-2}{2}$,
and its derivative is a symbol of order $-1$ in $(\taub,\nu)$, its imaginary part is
also a
symbol of order $-1$ in this strip, thus tends to $0$ as
$|(\taub,\nu)|\to\infty$. Since the derivative is
$$
2\taub\Big(\taub^2+\nu^2+\Big(\frac{n-2}{2}\Big)^2\Big)^{-1},
$$
integrating from the real axis to the line with imaginary part $l+1$
gives the principal symbol claim (in this sense the constant term is
irrelevant). The improved error term (the modulo $S^{-3}$ statement) follows from \eqref{eq:symm-Taylor} in the proof
of
Lemma~\ref{lemma:imag-order-minus-one}.

Thus, \eqref{eq:tilde-A-choice-mult-log-3rd}
is bounded away from 
the endpoints of the interval $\big(-\frac{\pi}{2},\frac{\pi}{2}\big)$ since as $|(\taub,\nu)|\to\infty$
it tends to $0$, and in a compact region this boundedness claim
automatically holds.

The last part follows since in this case
\eqref{eq:tilde-A-choice-mult-log-3rd} is the argument of a positive number.
\end{proof}

As a consequence of Lemma~\ref{lemma:imag-order-minus-one} and
Lemma~\ref{lemma:complex-Lap-arg}, we have:

\begin{cor}\label{cor:tilde-A-princ-symbol}
For any $\tilde\digamma\geq\digamma\geq\check\digamma>1$,
the expression \eqref{eq:tilde-A-choice-mult-log} is, for all $\delta'>0$, a symbol of order
$-1+\delta'$ in $(\taub,\nu)$, with principal symbol
$$
\pm\beta(\taub^2+\nu^2)^{-3/2}\big(\nu^2\log\big(\taub^2+\nu^2\big)+2\taub^2\big)\pm\tilde\beta(\taub^2+\nu^2)^{-3/2}\nu^2,
$$
and is thus positive/negative depending on the $\pm$ sign. Thus, for
any such $\digamma,\tilde\digamma,\check\digamma$,
\eqref{eq:tilde-A-choice-mult-log} lies in $(0,\pi)$, resp.\
$(-\pi,0)$, indeed in an arbitrarily small specified one-sided
neighborhood of $0$, for sufficiently large $(\taub,\nu)$.
\end{cor}

\begin{rem}
We emphasize that we are computing the argument (imaginary part of the
logarithm), not the imaginary part of the function multiplication by
which gives the modified (by multiplication by powers of $x$) normal operator \eqref{eq:normal-op-computation} on the Mellin transform/spectral side. The
actual imaginary part of that function is a symbol of order $2\tilde
r+\delta'$ for all $\delta'>0$, corresponding
to the real part of the function being a symbol of order $2(\tilde
r-1/2)+2=2\tilde r+1$.
\end{rem}

\begin{proof}
We just need to observe that, by using the above expressions for the
principal symbols of the four summands, the principal symbol of \eqref{eq:tilde-A-choice-mult-log} is
\begin{equation*}\begin{aligned}
&\pm\tilde\beta(\taub^2+\nu^2)^{-3/2}\nu^2\\
&\qquad\pm\beta(\taub^2+\nu^2)^{-3/2}\big(\nu^2\log\big(\taub^2+\nu^2\big)+2\taub^2\big)\\
&\qquad -(l+1)2\taub\big(\taub^2+\nu^2\big)^{-1}\\
&\qquad
+\frac{2(l+1)\taub}{\taub^2+\nu^2}\\
&=\pm\beta(\taub^2+\nu^2)^{-3/2}\big(\nu^2\log\big(\taub^2+\nu^2\big)+2\taub^2\big) \pm\tilde\beta(\taub^2+\nu^2)^{-3/2}\nu^2,
\end{aligned}\end{equation*}
which is positive, up to the $\pm$ sign. The final part follows as
symbols of negative order tend to $0$ at infinity.
\end{proof}

Note that when $|l+1|$ is close to $\frac{n-2}{2}$,
\eqref{eq:tilde-A-choice-mult-log-3rd} can be arbitrarily close to
both of $\pi/2,-\pi/2$, and does so for $\taub$ very close to $0$, which
makes the treatment of this problem harder since the other terms in
\eqref{eq:tilde-A-choice-mult-log} need to be in a very precisely
controlled interval around $\pi/2$ so that the sum is in $(0,\pi)$.

Notice that $l+1=0$ corresponds to the
`symmetric' (in terms of weight) mapping $\Hb^{\tilde
  r,-1}\to\Hb^{\tilde r-1,1}$. The last statement in Lemma~\ref{lemma:complex-Lap-arg} makes the $l+1=0$ case a bit
simpler, so we discuss it first. (If one takes $\beta>0$ small, it can
{\em easily}
be made {\em completely explicit, without the use of large
  parameters}, though this deducts from the simplicity, and moreover
we need to allow $\beta>1$ in the statement of
Proposition~\ref{prop:eff-norm-op-b-est-base} in order to prove Proposition~\ref{prop:impr-b-est}.)

\subsubsection{The case $l+1=0$ and preliminary computations for the
  general case}\label{subsubsec:special-l}
If $l+1=0$, we  simply take $\tilde\beta=0$, so (in addition to the
third and fourth terms) the first term of \eqref{eq:tilde-A-choice-mult-log} also
vanishes (and $\tilde\digamma$ becomes irrelevant), and we need
to estimate the second term only. This will be useful for the general
case as well, so we state it as a lemma.

\begin{lemma}\label{lemma:tilde-A-choice-mult-log-2nd-est}
The second term of \eqref{eq:tilde-A-choice-mult-log}, with the
notation $\lambda=\nu^2$, and without the
$\beta-\tilde\beta$ prefactor,
\begin{equation}\label{eq:tilde-A-choice-mult-log-2nd-est}
II=\im\Big(\Big(\frac{\taub+i}{((\taub+i)^2+\nu^2+\digamma^2)^{1/2}}\Big)\log\big((\taub+i)^2+\nu^2+\digamma^2\big)\Big),
\end{equation}
satisfies the following: let $\varepsilon>0$. There exists $\digamma_*>0$
such that for $\digamma\geq\digamma_*$, $II\in (0,\varepsilon)$ for all $(\taub,\nu)$.
\end{lemma}

\begin{cor}
For any $\beta\neq 0$ there exists $\digamma_*>0$ such that for
$\digamma\geq\digamma_*$, \eqref{eq:tilde-A-choice-mult-log} is in
$(0,\pi)$ if $\beta>0$, and \eqref{eq:tilde-A-choice-mult-log} is in
$(-\pi,0)$ if $\beta<0$.
\end{cor}

\begin{proof}
Take $\varepsilon=\pi/(2|\beta|)$, and apply
Lemma~\ref{lemma:tilde-A-choice-mult-log-2nd-est}, noting that, with $\tilde\beta=0$, the
  only non-zero term of \eqref{eq:tilde-A-choice-mult-log} is the
  second one, which is in $(0,\pi)$.
\end{proof}

\begin{proof}[Proof of
  Lemma~\ref{lemma:tilde-A-choice-mult-log-2nd-est}]
In the large parameter sense, the principal symbol of
\eqref{eq:tilde-A-choice-mult-log-2nd-est} is
\begin{equation}\label{eq:tilde-A-choice-mult-log-2nd-est-princ}
(\taub^2+\nu^2+\digamma^2)^{-3/2}\big((\nu^2+\digamma^2)\log(\taub^2+\nu^2+\digamma^2)+2\tau^2\big)
\end{equation}
modulo $S^{-3+\delta'}$, $\delta'>0$,
thus for any $\delta'>0$, \eqref{eq:tilde-A-choice-mult-log-2nd-est}  differs from this principal symbol, in absolute value, by
\begin{equation}\label{eq:tilde-A-choice-mult-log-2nd-est-error}
\leq C_1(\taub^2+\nu^2+\digamma^2)^{-3/2+\delta'}\leq C_1\digamma^{-1+2\delta'}(\taub^2+\nu^2+\digamma^2)^{-1}
\end{equation}
with $C_1$ independent of $\digamma$.
On the other hand, for $\digamma\geq 2$, \eqref{eq:tilde-A-choice-mult-log-2nd-est-princ}
has a lower bound
$$
(\taub^2+\nu^2+\digamma^2)^{-3/2}\big((\nu^2+\digamma^2)\log(\taub^2+\nu^2+\digamma^2)+2\tau^2\big)\geq (\taub^2+\nu^2+\digamma^2)^{-1/2}
$$
since the logarithm is $\geq 2\log\digamma\geq 1$,
and an upper bound
$$
2(\taub^2+\nu^2+\digamma^2)^{-1/2}\log(\taub^2+\nu^2+\digamma^2)\leq C_2(\taub^2+\nu^2+\digamma^2)^{-1/4}.
$$
Thus, in view of \eqref{eq:tilde-A-choice-mult-log-2nd-est-error}, there is $\digamma_*>2$ such that for $\digamma>\digamma_*$,
\eqref{eq:tilde-A-choice-mult-log-2nd-est} has a lower bound
\begin{equation}\label{eq:tilde-A-choice-mult-log-2nd-est-LB}
\frac{1}{2}(\taub^2+\nu^2+\digamma^2)^{-1/2}>0
\end{equation}
and an upper bound
$$
3(\taub^2+\nu^2+\digamma^2)^{-1/2}\log(\taub^2+\nu^2+\digamma^2)\leq C_2\digamma^{-1/2}<\varepsilon.
$$
This proves the lemma.
\end{proof}

This corollary implies that for in case $l+1=0$, with $\beta>0$
  arbitrary (and $\tilde\beta=0$), with $\digamma$ chosen sufficiently large (depending on
  $\beta$)
  \eqref{eq:commutator-imag-version}, thus
  \eqref{eq:normal-op-computation}, are indeed negative.
{\em In particular, Proposition~\ref{prop:eff-norm-op-b-est-base} holds in
this case}, completing the proof of Theorem~\ref{thm:main} when $l+1=0$. 

Also note that the positivity result for $l+1=0$ (choosing $\digamma$
depending on $\beta$, and then allowing $l+1$ of size depending on
these two choices) proves an analogous
result for $|l+1|$ small since \eqref{eq:commutator-imag-version} will
have an unchanged sign.

\subsubsection{The case of general $l+1$}\label{subsubsec:general-l}

In the general case of
$|l+1|<\frac{n-2}{2}$
we first
choose $\check\digamma>1$ so that the sum of the last two terms is still
in a compact subinterval of $(-\pi/2,\pi/2)$, see Lemma~\ref{lemma:last-terms-est}, then choose $\digamma$
sufficiently large so that the same property remains when the second
term is added, and in addition for sufficiently large
$\taub^2+\nu^2$ this sum is positive, and indeed close to $0$, see
Corollary~\ref{cor:last-3-terms-est}, and
then let $\tilde\beta=\frac{\pi}{2}\tilde\digamma$ and choose $\tilde\digamma$ sufficiently large so that the first
term is close to the constant $\pi/2$ on the compact region of
$(\taub,\nu)$ where we had not established positivity and is bounded
below by a suitably small negative quantity
everywhere, bounded above by a constant slightly greater than $\pi/2$, so that the
total sum is in $(0,\pi)$, see Lemma~\ref{lemma:first-term-est} and Corollary~\ref{cor:total-est}.

We thus prove:

\begin{lemma}\label{lemma:last-terms-est}
There exists $\check\digamma_*>1$ such that for
$\check\digamma\geq\check\digamma_*$, the last two terms of
\eqref{eq:tilde-A-choice-mult-log} (with the notation $\lambda=\nu^2$),
\begin{equation}\begin{aligned}\label{eq:tilde-A-choice-mult-log-last-two}
&\im\Big(-(l+1) \log\big((\taub+i)^2+\nu^2+\check\digamma^2\big)\Big)\\
&\qquad+\im\log\Big((\taub+i(l+1))^2+\nu^2+\Big(\frac{n-2}{2}\Big)^2\Big),
\end{aligned}\end{equation}
have opposite signs (in the sense that if one is $\geq 0$, the other
is $\leq 0$), the sum
lies in a compact subinterval $[-\alpha_0,\alpha_0]$ of
$(-\pi/2,\pi/2)$, and is a symbol of order $-3$ in $(\taub,\nu)$.

Thus, there exist $C_0\geq 1$ such that for
$|(\taub,\nu)|\geq 1$, the absolute value of the sum is $\leq
C_0|(\taub,\nu)|^{-3}$.
\end{lemma}

\begin{rem}\label{rem:constants-normalized-last-terms}
The upper bound in the last statement is only useful on
$|(\taub,\nu)|\geq R_0$ where $C_0R_0^{-3}=\alpha_0$, i.e.\
$R_0=C_0^{1/3}/\alpha_0^{1/3}$.
\end{rem}

\begin{proof}
Since $\im
\Big((\taub+i(l+1))^2+\nu^2+\Big(\frac{n-2}{2}\Big)^2\Big)=2(l+1)\taub$,
$$
\im\log\Big((\taub+i(l+1))^2+\nu^2+\Big(\frac{n-2}{2}\Big)^2\Big)\in(-\pi/2,\pi/2)
$$
is $\geq 0$, resp.\ $\leq 0$ when $2(l+1)\taub\geq 0$, resp.\
$2(l+1)\taub\leq 0$.

On the other hand,
$\im\big((\taub+i)^2+\nu^2+\check\digamma^2\big)=2\taub$, hence has the
sign of $\taub$, and thus, as
$\re\big((\taub+i)^2+\nu^2+\check\digamma^2\big)=\taub^2-1+\nu^2+\check\digamma^2>0$, 
$\im\log\big((\taub+i)^2+\nu^2+\check\digamma^2\big)$ is in $[0,\pi/2)$,
resp.\ $(-\pi/2,0]$ corresponding to whether $\taub\geq 0$ or
$\taub\leq 0$. Moreover, by choosing $\check\digamma>1$ sufficiently
large, one can arrange that
$|\im\log\big((\taub+i)^2+\nu^2+\check\digamma^2\big)|<\frac{\pi}{2(1+|l+1|)}$,
as follows from the large parameter symbolic considerations from
Lemma~\ref{lemma:imag-order-minus-one} namely that the left hand side is bounded by
$C\big(\taub^2+\nu^2+\check\digamma^2\big)^{-1/2}\leq C/\check\digamma $, with $C$ independent
of $\check\digamma$. (It is also straightforward to arrange this
explicitly.) Thus,
$$
\im\Big(-(l+1)
\log\big((\taub+i)^2+\nu^2+\check\digamma^2\big)\Big)\in \Big
(-\frac{\pi|l+1|}{2(1+|l+1|)},
\frac{\pi|l+1|}{2(1+|l+1|)}\Big)\subset\Big(-\frac{\pi}{2},\frac{\pi}{2}\Big),
$$
with sign matching that of $-(l+1)\taub$, thus indeed the opposite of
that of the second term of
\eqref{eq:tilde-A-choice-mult-log-last-two}.

Since the sum of two quantities of opposite sign, each of which is in
$(-\pi/2,\pi/2)$, is itself in $(-\pi/2,\pi/2)$, and since the asymptotic
behavior of each term is that of a symbol of order $-1$, thus the
terms decay to $0$, the conclusion regarding $\alpha_0$ follows. The
principal symbol as a symbol of order $-1$ is the sum of two principal
symbols, which are $\pm 2(l+1)\frac{\taub}{\taub^2+\nu^2}$, the sum
has vanishing principal symbol modulo $S^{-3}$, and thus it is a symbol of order
$-3$. The final statement is an immediate consequence of this symbolic property.
\end{proof}

Then:

\begin{cor}\label{cor:last-3-terms-est}
Let $\check\digamma_*>1$ be as in
Lemma~\ref{lemma:last-terms-est}. There exist $\digamma_*>0$ and
$\alpha_*\in(0,\pi/2)$ such
that for $\digamma\geq\digamma_*$ the sum of the last three terms of 
\eqref{eq:tilde-A-choice-mult-log} is in $[-\alpha_*,\alpha_*]$ and
there is $R>0$ such that for $\taub^2+\nu^2\geq R^2$, the sum of these
last three terms is $>\frac{1}{8}(\taub^2+\nu^2+\digamma^2)^{-1/2}>0$.
\end{cor}

\begin{proof}
This is immediate using Lemma~\ref{lemma:last-terms-est} first, then
choosing $\varepsilon=(\pi/2-\alpha_0)/(4|\beta|)$ in
Lemma~\ref{lemma:tilde-A-choice-mult-log-2nd-est} to find
$\digamma_*$. Since the second term of
\eqref{eq:tilde-A-choice-mult-log} is positive and less than
$|\beta|\varepsilon=(\pi/2-\alpha_0)/4$, the sum of the last three terms
lies in $[-\alpha_*,\alpha_*]$ with
$0<\alpha_*=\alpha_0+(\pi/2-\alpha_0)/4<\pi/2$. Finally, due to
\eqref{eq:tilde-A-choice-mult-log-2nd-est-LB}, the second
term has a lower bound $\frac{1}{2}(\taub^2+\nu^2+\digamma^2)^{-1/2}$,
which is bounded from below by $\frac{1}{4}(\taub^2+\nu^2)^{-1/2}$ for $|(\taub,\nu)|\geq\digamma$,
while the sum of the last two terms has an upper bound in absolute
value $\leq
C_0|(\taub,\nu)|^{-3}$ (where $|(\taub,\nu)|>1$), so for sufficiently
large $|(\taub,\nu)|$ (namely,
$|(\taub,\nu)|\geq R=\max(\digamma,\sqrt{8C_0})$), the second term dominates,
and thus the last two terms can be absorbed into it as desired, giving
the statement of the corollary.
\end{proof}

Hence, at this point, we have that the sum of the last three terms is
contained in $(0,\pi)$ for $\taub^2+\nu^2\geq R^2$, and is contained
in a compact subinterval of $(-\pi/2,\pi/2)$ for all $(\taub,\nu)$.

\begin{lemma}\label{lemma:first-term-est}
There exists $\tilde\digamma_*\geq\digamma_*$ such that for
$\tilde\digamma\geq\tilde\digamma_*$, the first term of
\eqref{eq:tilde-A-choice-mult-log}, with the $\tilde\beta=\frac{\pi}{2}\tilde\digamma$ prefactor,
\begin{equation}\begin{aligned}\label{eq:tilde-A-choice-mult-log-1st}
\im\Big(&\tilde\beta\frac{\taub+i}{((\taub+i)^2+\nu^2+\tilde\digamma^2)^{1/2}}\Big),
\end{aligned}\end{equation}
is $\geq -\frac{1}{16}(\taub^2+\nu^2+\digamma^2)^{-1/2}$, is $<
\frac{\pi}{2}+\frac{1}{4}\big(\frac{\pi}{2}-\alpha_*\big)$ everywhere,
and is $>\alpha_*$ for $\taub^2+\nu^2\leq R^2$.
\end{lemma}

\begin{proof}
In the large parameter sense, the principal symbol of
\eqref{eq:tilde-A-choice-mult-log-1st} is, without the $\tilde\beta$
prefactor,
\begin{equation}\begin{aligned}\label{eq:tilde-A-choice-mult-log-1st-est-princ-0}
&(\taub^2+\nu^2+\tilde\digamma^2)^{-3/2}(\nu^2+\tilde\digamma^2)
\end{aligned}\end{equation}
modulo $S^{-3}$, which is non-negative; including the prefactor it
becomes, as a symbol of order $0$ now jointly in $(\taub,\nu,\tilde\digamma)$,
\begin{equation}\begin{aligned}\label{eq:tilde-A-choice-mult-log-1st-est-princ}
&\frac{\pi}{2}\tilde\digamma(\taub^2+\nu^2+\tilde\digamma^2)^{-3/2}(\nu^2+\tilde\digamma^2)
\end{aligned}\end{equation}
modulo $S^{-2}$. But, using $\tilde\digamma\geq\digamma$,
\begin{equation}\begin{aligned}\label{eq:tilde-A-choice-mult-log-1st-est-princ-error}
&\frac{\pi}{2}\tilde\digamma\Big|\im\Big(\frac{\taub+i}{((\taub+i)^2+\nu^2+\tilde\digamma^2)^{1/2}}\Big)-(\taub^2+\nu^2+\tilde\digamma^2)^{-3/2}(\nu^2+\tilde\digamma^2)\Big|\\
&\qquad\leq C\tilde\digamma (\taub^2+\nu^2+\tilde\digamma^2)^{-3/2}\leq C\tilde\digamma^{-1}(\taub^2+\nu^2+\digamma^2)^{-1/2},
\end{aligned}\end{equation}
which is $\leq \frac{1}{16}(\taub^2+\nu^2+\digamma^2)^{-1/2}$ for
sufficiently large $\tilde\digamma$ (say,
$\tilde\digamma\geq\tilde\digamma_*'$). In view of
\eqref{eq:tilde-A-choice-mult-log-1st-est-princ} being positive, this proves that
\eqref{eq:tilde-A-choice-mult-log-1st} is $\geq
-\frac{1}{16}(\taub^2+\nu^2+\digamma^2)^{-1/2}$.

The upper bound
$\frac{\pi}{2}+\frac{1}{4}\big(\frac{\pi}{2}-\alpha_*\big)$ is proved
similarly, for \eqref{eq:tilde-A-choice-mult-log-1st-est-princ} is
bounded from above by $\frac{\pi}{2}$ and
\eqref{eq:tilde-A-choice-mult-log-1st-est-princ-error} is bounded from
above by $C\tilde\digamma^{-1}$, so for sufficiently large
$\tilde\digamma$, say $\tilde\digamma\geq\tilde\digamma_*''$, is $<\frac{1}{4}\big(\frac{\pi}{2}-\alpha_*\big)$.

Finally, \eqref{eq:tilde-A-choice-mult-log-1st} is
\begin{equation}\label{eq:tilde-A-choice-mult-log-1st-int}
\frac{1}{2}\int_{-1}^1 \frac{\pi}{2}\tilde\digamma((\taub+is)^2+\nu^2+\tilde\digamma^2)^{-3/2}(\nu^2+\tilde\digamma^2)\,ds
\end{equation}
so with $\taubh=\taub/\tilde\digamma$, $\nuh=\nu/\tilde\digamma$,
$$
=\frac{\pi}{4}\int_{-1}^1 ((\taubh+is\tilde\digamma^{-1})^2+\nuh^2+1)^{-3/2}(\nuh^2+1)\,ds
$$
which is an equicontinuous family in $\tilde\digamma^{-1}<2^{-1/2}$ (say), taking
the value $\frac{\pi}{2}$ at $\taubh=\nuh=0$, $\tilde\digamma^{-1}=0$. Hence, for any
$\varepsilon'>0$, such as $\varepsilon'=(\pi/2-\alpha_*)/4$, there is $\hat R>0$, which one may assume to be $<R/2$, such that for
$\taubh^2+\nuh^2\leq\hat R^2$ (and any $\tilde\digamma$) it differs from
$\frac{\pi}{2}$ by $<\varepsilon'$, thus for $|(\taub,\nu)|\leq\hat
R\tilde\digamma$ the analogous conclusion holds.

Now simply pick
$\tilde\digamma'''_*=R/\hat R>2$, and let
$\tilde\digamma_*=\max(\tilde\digamma_*',\tilde\digamma_*'',\tilde\digamma_*''')$. Then
the $\geq -\frac{1}{16}(\taub^2+\nu^2+\digamma^2)^{-1/2}$,
$<\frac{\pi}{2}+\frac{1}{4}\big(\frac{\pi}{2}-\alpha_*)$, and
$>\alpha_*$ for $\taub^2+\nu^2\leq R^2$, claims follow immediately.
\end{proof}

\begin{cor}\label{cor:total-est}
With $\digamma\geq\digamma_*$, $\tilde\digamma\geq\tilde\digamma_*$,
$\tilde\beta=\frac{\pi}{2}\tilde\digamma$,
\eqref{eq:tilde-A-choice-mult-log} lies in $(0,\pi)$.
\end{cor}

\begin{proof}
Since by Corollary~\ref{cor:last-3-terms-est} the sum of the last three terms lies in $(-\pi/2,\alpha_*]$,
$\alpha_*<\pi/2$, while
the first is $\leq\pi/2+(\pi/2-\alpha_*)/4$, certainly the sum is $<\pi$
everywhere. Moreover, for $\taub^2+\nu^2\leq R^2$, the sum of the
first three terms is $\geq-\alpha_*$, while the first term is
$>\alpha_*$, so the sum is positive. Finally, for $\taub^2+\nu^2\geq
R^2$, the sum of the last three terms is $\geq
\frac{1}{8}(\taub^2+\nu^2+\digamma^2)$, while the first is $\geq
-\frac{1}{16}(\taub^2+\nu^2+\digamma^2)$, so the sum is $\geq
\frac{1}{16}(\taub^2+\nu^2+\digamma^2)$, thus positive, as desired.
\end{proof}

This corollary implies that for $\beta>0$
  arbitrary, with $\tilde\digamma,\digamma,\check\digamma$ chosen sufficiently large (depending on
  $\beta$) {\em \eqref{eq:commutator-imag-version}, thus
  \eqref{eq:normal-op-computation}, are indeed negative}.
In particular, {\em this completes the proof of
  Proposition~\ref{prop:eff-norm-op-b-est-base}, and thus that of
  Theorem~\ref{thm:main}. as well.}

\section{Second microlocal analysis}\label{sec:2-micro}

We now consider a second microlocal scattering version of Theorem~\ref{thm:main-improved}. The reason
for this is that for $\sigma\neq 0$, it is simplest to consider the
limiting resolvent as a map between variable order scattering
spaces. One can extend such a picture to $\sigma=0$ if one resolves
the zero section by blowing it up; this gives rise to the second
microlocal statement below.

First, recall the basic sc-analysis ingredients from \cite{RBMSpec}, keeping in mind that
these arise by generalizing the structure of $\RR^n$ near infinity by
identifying asymptotically conic open sets first with open sets in the radial
compactification $\overline{\RR^n}$ of $\RR^n$, and then the latter
with open sets on the compact manifold with boundary $X$. This is
completely analogous to how pseudodifferential operators are
transplanted from $\RR^n$ to manifolds via coordinate charts; the
off-diagonal smooth contributions to the Schwartz kernel in that case
are replaced by Schwartz contributions in the present case (i.e.\
right-densities on $X^2$ which vanish to infinite order at $\pa
(X^2)$); see \cite[Section~5.3]{Vasy:Minicourse} for a complete treatment along these lines.

Thus, recall from the introduction that the set (and Lie algebra) of sc-vector fields, $\Vsc(X)$, is
$x\Vb(X)$, and is spanned by $x^2\pa_x,x\pa_{y_j}$ over $\CI(X)$, i.e.\ a
sc-vector field is a linear combination of $x^2\pa_x,x\pa_{y_j}$ with
$\CI(X)$ coefficients. Such vector fields are all smooth sections of a
vector bundle, $\Tsc X$, over $X$; a local basis for the latter is
$x^2\pa_x,x\pa_{y_j}$. If $X$ is the radial compactification of
$\RR^n$, i.e.\ $\RR^n$ is compactified by gluing $r=+\infty$ to it via
identifying $\RR^n\setminus \overline{B_1(0)}$ with
$(1,\infty)_r\times\sphere^{n-1}$ via spherical coordinates,
introducing $x=r^{-1}$ and adding $x=0$ to the space, so
$(1,\infty)_r\times\sphere^{n-1}$ is identified with
$(0,1)_x\times\sphere^{n-1}$ in $[0,1)_x\times\sphere^{n-1}$, then
$\Vsc(X)$ consists exactly of linear combinations of the coordinate
vector fields $\pa_{z_j}$ with classical symbolic of order $0$
coefficients. Correspondingly,
$\Tsc\overline{\RR^n}=\overline{\RR^n_z}\times\RR^n_v$, where over the
interior of $\RR^n_z$ one simply writes tangent vectors as $\sum_j
v_j\pa_{z_j}$, and the identification actually only uses the vector
space structure of $\RR^n$ (but not the basis, or having a metric).

Dually, the sc-cotangent bundle, $\Tsc^*X$, is spanned by
$\frac{dx}{x^2},\frac{dy_j}{x}$ in local coordinates. In case of
$X^\circ=\RR^n$, one obtains
$$
\Tsc^*\overline{\RR^n}=\overline{\RR^n_z}\times(\RR^n)^*_\zeta,
$$
i.e.\
sc-one-forms are of the form $\sum_j a_j\,dz_j$,
$a_j\in\CI(X)=S^0_\cl(\RR^n)$. One writes $(\tau,\mu_j)$ as local
coordinates.
Below it is sometimes convenient to compactify the fibers of $\Tsc^*X$
radially, which is possible by exactly the same construction as the
compactification of $\RR^n$; see Figure~\ref{fig:sc-comp}. Thus,
$$
\overline{\Tsc^*}\overline{\RR^n}=\overline{\RR^n_z}\times\overline{(\RR^n)^*_\zeta}.
$$

One defines $\Diffsc(X)$ as finite sums of products of elements of
$\Vsc(X)$ with $\CI(X)$ coefficients; in case of $\RR^n$ these are just of the form
$\sum_\alpha a_\alpha D^\alpha$,
$a_\alpha\in\CI(X)=S^0_\cl(\RR^n)$. We also consider general symbolic
coefficients, and write $S^l\Diffsc(X)$ for sc-differential operators
with symbolic coefficients of order $l$; in the case of $\RR^n$ again
this simply means that $a_\alpha\in S^l(\RR^n)$, the standard space of
symbols of order $l$.

\begin{figure}
\begin{center}
\includegraphics[width=125mm]{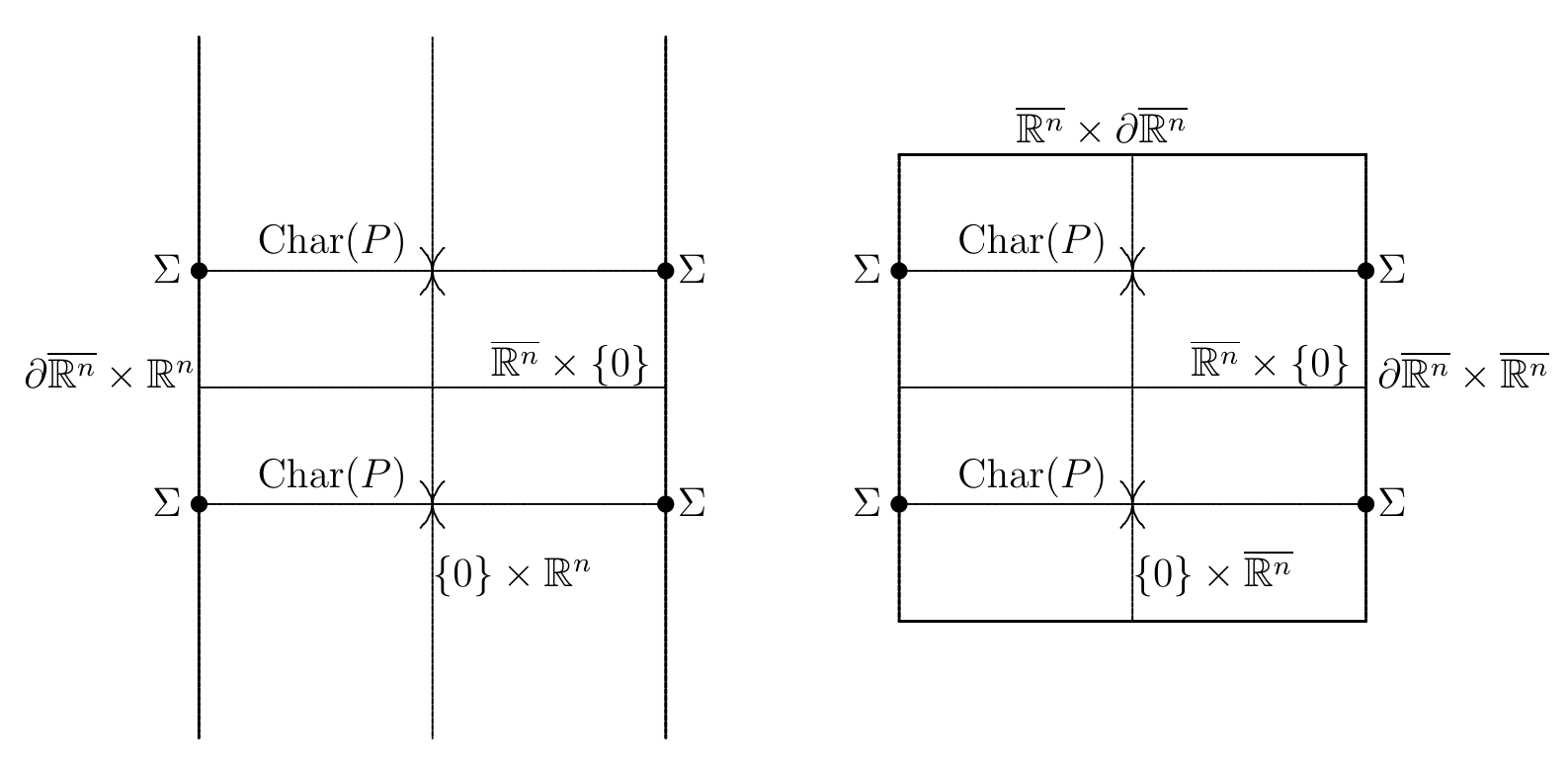}
\end{center}
\caption{The scattering cotangent bundle $\Tsc^*X$ (left) and its
  fiber-compactification $\overline{\Tsc^*}X$, with $X=\overline{\RR^n}$. Also shown is the
  characteristic set of $P(\sigma)$ for $\sigma\neq
  0$, both from the compactified perspective, as $\Sigma$, which is a
  subset of the boundary, and from the conic perspective, here conic
  in the base (i.e.\ the dilations are in the $\RR^n_z$ factor), as $\Char(P)$. The fiber of cotangent bundle over the
  origin, i.e.\ $\{0\}\times\overline{(\RR^n)^*}$, is also indicated;
  this is only special from the conic (dilation) perspective, in which
  it is the analogue of the zero section in standard microlocal analysis.}
\label{fig:sc-comp}
\end{figure}

Scattering pseudodifferential
operators are microlocalizations of these sc-differential operators. They arise as quantizations
of symbols on $\Tsc^*X$; classical symbols of order $(0,0)$ are smooth
functions on the fiber-compactification $\overline{\Tsc^*}X$. In case
of $X^\circ=\RR^n$, the latter are order $(0,0)$ classical symbols in
the standard sense, i.e.\ they have asymptotic expansions at both
boundary hypersurfaces. Conversely, their principal symbols are in
$S^{m,l}(\Tsc^*X)/S^{m-1,l-1}(\Tsc^*X)$, which for classical symbols
can be identified with functions on the boundary.

One has Sobolev spaces, including variable order Sobolev spaces,
corresponding to this structure; e.g.\
for $s$ a non-negative integer
$\Hsc^{s,r}$ consists of distributions such that for all
$Q\in\Diffsc^s(X)$, $Qu\in x^r L^2=x^r L^2_\scl$. Again, via the
identification via asymptotically conic coordinate charts in $\RR^n$,
this corresponds to the standard (albeit possibly variable order)
Sobolev spaces. We refer to \cite[Section~5.3.9]{Vasy:Minicourse} for
a detailed discussion.

We also recall, see \cite{RBMSpec} and \cite[Section~5.3]{Vasy:Minicourse}, that elliptic theory applies in this sc-setting, and
gives for instance the Fredholmness, and via formal self-adjointness
invertibility, of $P(\sigma)$, $\sigma\in\Cx\setminus\RR$, as a map
$$
P(\sigma):\Hsc^{s+2,r}\to\Hsc^{s,r},
$$
where $s,r$ can be taken to be {\em arbitrary} variable orders, and
the inverse is independent of the choices, e.g.\ in the sense that on
the dense Schwartz subspace (functions vanishing to infinite order at
$\pa X$) all inverse agree.

Then the limiting absorption principle, in the Fredholm setting,
for $\sigma\neq 0$, and
for suitable variable order $r$, namely one monotone along the
Hamilton flow and satisfying the appropriate inequalities at the
radial sets, with choice {\em irrelevant in the
  elliptic region} (by elliptic theory), and which can be taken
to be
\begin{equation}\label{eq:sc-symbol-weight-choice-0}
r=-\frac{1}{2}\pm\beta\frac{\tau}{(\tau^2+|\mu|^2)^{1/2}}
\end{equation}
near the characteristic set (which is disjoint from the zero section,
where \eqref{eq:sc-symbol-weight-choice-0} is singular),
and which in case of $\RR^n$ becomes
$$
r=-\frac{1}{2}\mp\beta \frac{z\cdot\zeta}{|z||\zeta|},
$$
is that
$$
P(\sigma):\{u\in\Hsc^{s+2,r}:\ P(\sigma)u\in\Hsc^{s,r+1}\}\to \Hsc^{s,r+1}
$$
is actually invertible; see \cite[Section~5.4.8]{Vasy:Minicourse}, and
concretely Proposition~5.28 there. Note that the order is irrelevant away from
the characteristic set, where the operator is elliptic.

In order to do second microlocal analysis at the zero section, which
is required for the $\sigma\to 0$ analysis as $P(\sigma)$ becomes degenerate from the sc-perspective
at $\sigma=0$, we need to discuss first resolving the zero section at
$\pa X$. Note that the weight \eqref{eq:sc-symbol-weight-choice-0}
indeed becomes singular at the zero section; resolving the zero
section removes this singularity.

From a slightly different perspective, the weight $r$ has to be
monotone along the Hamilton flow of the principal symbol, and the
Hamilton flow degenerates at the zero section since the principal
symbol has a quadratic zero there, thus a resolution is needed. We
recall that such a resolution, called a blowup, of a submanifold, here
the boundary of the zero section, consists of replacing it by its
inward pointing spherical normal bundle and naturally obtaining a
smooth manifold with corners. Thus, one keeps track from which {\em
  normal} (i.e.\ modulo tangential) direction one is approaching the
submanifold; correspondingly {\em projective coordinates}, like the
ones discussed below, are usually particular easy to use.

Thus, we blow up the boundary $o_{\pa X}$ of the zero section $o$ in
$\Tsc^*X$: $[\Tsc^*X;o_{\pa X}]$; see Figure~\ref{fig:2-micro}. Notice that near the interior of the
front face, $x,y,\tau/x,\mu/x$ are coordinates. As one forms are
written as
$$
\tau\,\frac{dx}{x^2}+\sum_j
\mu_j\,\frac{dy_j}{x}=\frac{\tau}{x}\,\frac{dx}{x}+\sum_j \frac{\mu_j}{x}\,dy_j,
$$
we have
\begin{equation}\label{eq:b-sc-covector-conv}
\taub=\frac{\tau}{x},\ \mub=\frac{\mu}{x}
\end{equation}
in terms of the earlier discussed b-coordinates, i.e.\ the front face
is exactly the b-cotangent bundle, $\Tb^*X$, over $\pa X$.

On the
other hand,
near the corner given by the boundary of the front face, $x/(\tau^2+|\mu|^2)^{1/2},\ (\tau^2+|\mu|^2)^{1/2}, y$
plus homogeneous degree zero coordinates on the $(\tau,\mu)$-sphere
(like $\tau/(\tau^2+|\mu|^2)^{1/2}$, etc.), give coordinates; note that
$x/(\tau^2+|\mu|^2)^{1/2}=(\taub^2+|\mub|^2)^{-1/2}$ is exactly the
defining function of the lift of $\Tsc^*_{\pa X}X$, while
$(\tau^2+|\mu|^2)^{1/2}$ is the (local!) defining function of the front
face, while the product of these two, $x$, is the total defining
function of these two faces.
The corner (i.e.\ the boundary of the front
face) is identifiable with $\Ssc^*X=(\Tsc^*X\setminus o)/\RR^+$ (the
quotient being dilations in the fibers of a vector bundle), the
sc-cosphere bundle, which is naturally identified in turn with
$\Sb^*X$, since the b- and sc- structures differ merely by a
conformal factor. Indeed, the order function $r$ of \eqref{eq:sc-symbol-weight-choice-0} is a
function on $\Ssc^*X$ (thus on $\Sb^*X$), thus on $\Tsc^*X\setminus o$, and the blow-up allows the use of $r$ down to the zero section. 

Near fiber infinity on
$\overline{\Tsc^*}X$, the blow-up does not change anything, so the
defining function of fiber infinity is $(\tau^2+\mu^2)^{-1/2}=x^{-1}(\taub^2+\mub^2)^{-1/2}$, while
that of the spatial boundary is $x$.

\begin{figure}[ht]
\begin{center}
\includegraphics[width=120mm]{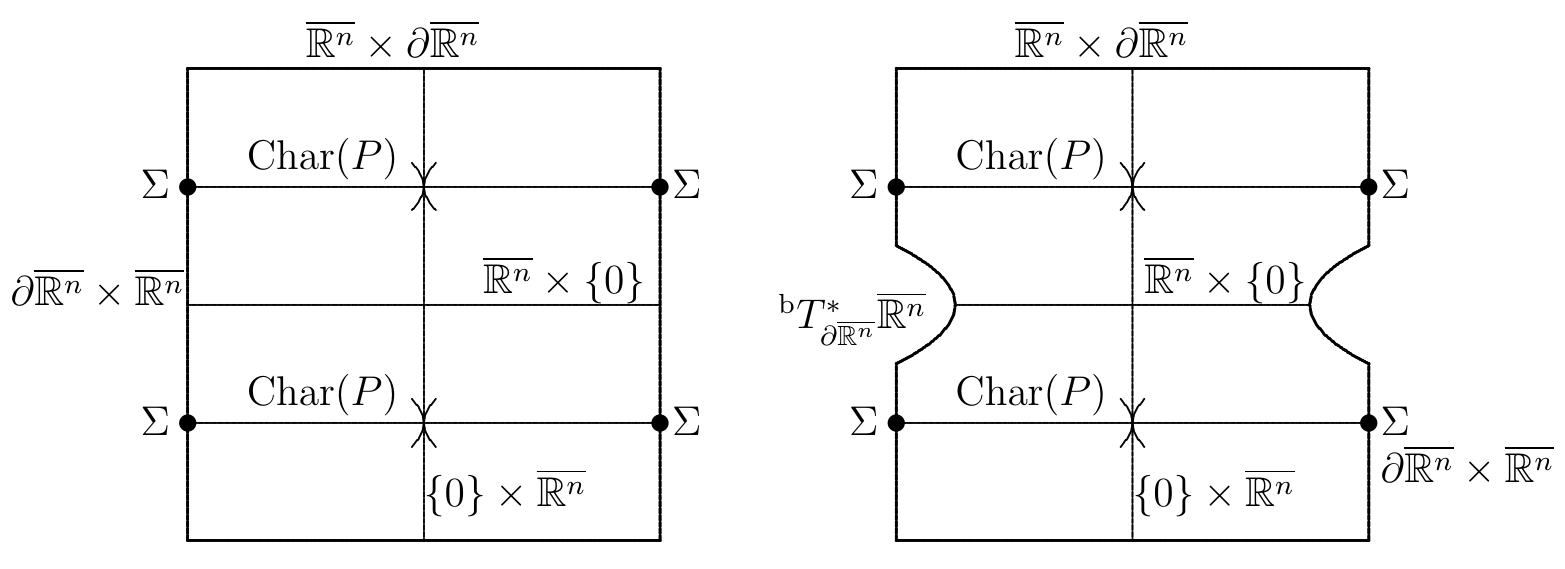}
\end{center}
\caption{Second microlocalized Euclidean space $\RR^{n}$. The left
  hand side is the fiber-compactified sc-cotangent bundle,
  $\overline{\Tsc^*}\overline{\RR^n}=\overline{\RR^n}\times\overline{(\RR^n)^*}$,
  the right hand side is its blow-up at the boundary of the zero
  section. The
  (interior of the) front
  face of the blow-up, shown by the curved arcs, can be identified
  with $\Tb^*_{\pa\overline{\RR^n}}\overline{\RR^n}$. The
  characteristic set of $P(\sigma)$, $\sigma\neq 0$, is also shown,
  both from the compactified perspective, as $\Sigma$, which is a
  subset of the boundary, and from the conic perspective, here conic
  in the base (i.e.\ the dilations are in the $\RR^n_z$ factor), as $\Char(P)$. The fiber of cotangent bundle over the
  origin, i.e.\ $\{0\}\times\overline{(\RR^n)^*}$, is also indicated;
  this is only special from the conic (dilation) perspective, in which
  it is the analogue of the zero section in standard microlocal analysis.}
\label{fig:2-micro}
\end{figure}

In order to actually do analysis, one needs pseudodifferential
operators corresponding to this 2-microlocal resolution. Formally, one
could say that one takes singular sc-symbols, namely ones that are
conormal (or potentially classical) on the blown-up space, and
quantizes these using the scattering quantization via local coordinate
charts. The problem with
this approach is that the contribution from the interior of the front
face is necessarily global in nature (corresponds to the global normal
operator in the b-setting), even modulo `trivial' (rapidly decaying
Schwartz kernel) terms, and thus it cannot literally make sense. This
singular symbol approach is most common in the standard
pseudodifferential setting, and requires significant effort to
justify, see \cite{Bony:Calcul}, and also
\cite{Sjostrand-Zworski:Fractal} and
\cite{Vasy-Wunsch:Semiclassical} in the technically simpler
semiclassical setting. In fact, even in the standard pseudodifferential setting it is
conceptually easier to work with paired Lagrangian distributions, see
\cite{Melrose-Uhlmann:Intersection,Guillemin-Uhlmann:Oscillatory} and
\cite{DUV:Diffraction} for a recent treatment of these, though second
microlocalization, which would correspond to the two Lagrangians being
the conormal bundle of the diagonal on the one hand, and the product
of the Lagrangian with itself (`primed', i.e.\ the covector sign reversed) on the other, has not been explicitly discussed using these (in
\cite{DUV:Diffraction} the singularities are lower order).

\begin{figure}[ht]
\begin{center}
\includegraphics[width=120mm]{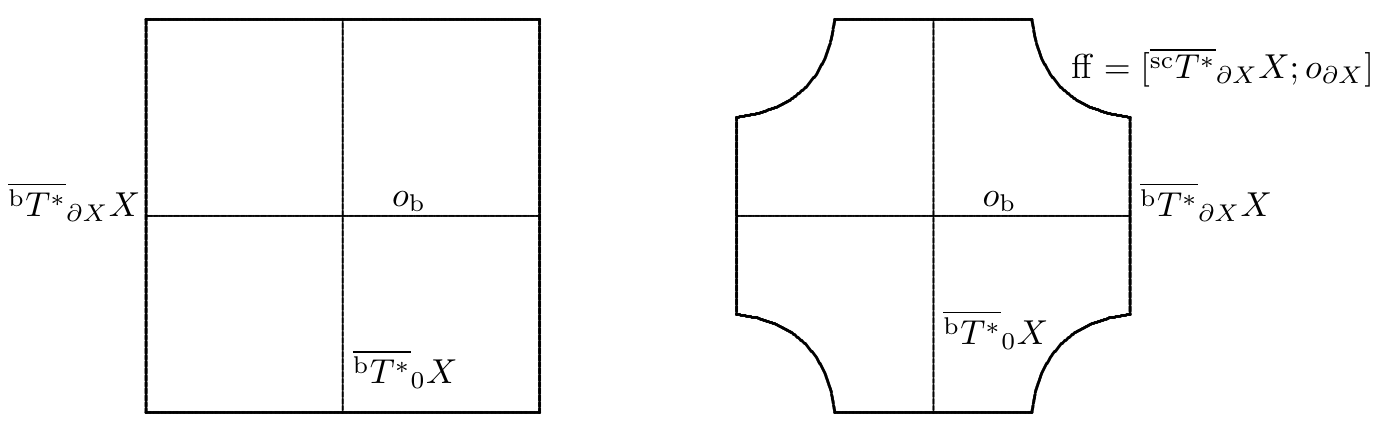}
\end{center}
\caption{The second microlocal space, on the right, obtained by blowing up the corner
  of $\overline{\Tb^*}X$, shown on the left.}
\label{fig:b-2-micro}
\end{figure}

Thus, we consider instead the `opposite' perspective, in which one
starts with the b-cotangent bundle, and blows up its corner, i.e.\
fiber-infinity at the boundary, see Figure~\ref{fig:b-2-micro}. A simple computation shows that the
resulting space is naturally diffeomorphic to the previously discussed
one:

\begin{lemma}
The identity map
from the interior, $T^*X$, extends smoothly to an
invertible map with smooth inverse,
$[\overline{\Tb^*}X;\pa\overline{\Tb^*}_{\pa X}
X]\to[\overline{\Tsc^*}X;o_{\pa X}]$.
\end{lemma}

\begin{proof}
We cover $[\overline{\Tb^*}X;\pa\overline{\Tb^*}_{\pa X}
X]$ and $[\overline{\Tsc^*}X;o_{\pa X}]$ by coordinate charts and show
that the identity map restricted to the {\em interior} of these charts
extends smoothly to the boundary in both directions. Concretely, we
will have four different regions to deal with: one over $X^\circ$ (which is straightforward), and three near
$\pa X$; for actual local coordinates we will then restrict to working
over local
coordinate charts $O$ in $\pa X$.

First of all, over $X^\circ$ both $\overline{\Tb^*}X$ and
$\overline{\Tsc^*}X$, thus their blow ups (which are over $\pa X$) are
naturally identified
with (in the same sense as above) $\overline{T^*}X$, i.e.\ the b- and
sc-structures are irrelevant. Thus, it suffices to work near $\pa
X$.

Next, in view of \eqref{eq:b-sc-covector-conv}, for $c>0$,
\begin{equation}\label{eq:fiber-infty-coords}
\{(\tau^2+|\mu|^2)^{1/2}>c\}=\{x(\taub^2+|\mub|^2)^{1/2}>c\},
\end{equation}
and the former is a neighborhood of fiber infinity in
$\overline{\Tsc^*X}$, thus in $[\overline{\Tsc^*}X;o_{\pa X}]$ (as the
blown up submanifold, $o_{\pa X}$, is disjoint from fiber infinity), while the
latter is a neighborhood of the lift of fiber infinity to
$[\overline{\Tb^*}X;\pa\overline{\Tb^*}_{\pa X}]$. If $y_j$ are
local coordinates on a set $O$ in $\pa X$ then $x,y_j,
(\tau^2+|\mu|^2)^{-1/2}$ and spherical coordinates in $(\tau,\mu_j)$ are
coordinates in the lift (inverse image under the blow-down map)
of $\overline{\Tsc^*}_O X$ in this region in $[\overline{\Tsc^*}X;o_{\pa
  X}]$, while $x,y_j, x^{-1}(\taub^2+|\mub|^2)^{-1/2}$ and spherical
coordinates in $(\taub,\mub)$ are coordinates in the lift of
$\overline{\Tb^*}_O X$ in this region in $[\overline{\Tb^*}X;\pa\overline{\Tb^*}_{\pa X}
X]$. Thus, $x^{-1}(\taub^2+|\mub|^2)^{-1/2}=(\tau^2+|\mu|^2)^{-1/2}$ and the
spherical equivalence of $(\tau,\mu)$ and $(\taub,\mub)$ shows the
equivalence of the two spaces (in the sense of smooth extension of the
identity map in either direction) in this region.

Next, consider the region
\begin{equation}\label{eq:b-face-coords}
\{x^{-1}(\tau^2+|\mu|^2)^{1/2}<c\}=\{(\taub^2+|\mub|^2)^{1/2}<c\},
\end{equation}
$c>0$. This is on the one hand a neighborhood of a compact subset of
the interior of the front face of $[\overline{\Tsc^*}X;o_{\pa X}]$, on the other hand
a neighborhood of a compact subset of (the lift of) $\Tb^*_{\pa X}X$ in $[\overline{\Tb^*}X;\pa\overline{\Tb^*}_{\pa X}
X]$. Again if $y_j$ are local coordinates on a set $O$ in $\pa X$,
then $x,y_j,\tau/x,\mu_j/x$ are (projective) coordinates in the lift of
$\overline{\Tsc^*}_O X$ in this region in $[\overline{\Tsc^*}X;o_{\pa
  X}]$, while $x,y_j,\taub,\mub$ are coordinates in the lift of
$\overline{\Tb^*}_O X$ in this region in $[\overline{\Tb^*}X;\pa\overline{\Tb^*}_{\pa X}
X]$ (note that this blow-up does not affect this region), so by \eqref{eq:b-sc-covector-conv}, the identity map extends to
be smooth in this region.

Finally, consider the region
\begin{equation}\label{eq:bsc-corner-coords}
\{c_1
x<(\tau^2+|\mu|^2)^{1/2}<c_2\}=\{c_1x<x(\taub^2+|\mub|^2)^{1/2}<c_2\}.
\end{equation}
This is on the one hand a neighborhood of the boundary of the front
face of $[\overline{\Tsc^*}X;o_{\pa X}]$ (as well as the new boundary of the lift of
$\overline{\Tsc^*}_{\pa X}X$), on the other hand
a neighborhood of the boundary of the lift of $\overline{\Tb^*}_{\pa
  X}X$ to $[\overline{\Tb^*}X;\pa\overline{\Tb^*}_{\pa X}
X]$. If $y_j$ are local coordinates on a set $O$ in $\pa X$,
then $x (\tau^2+|\mu|^2)^{-1/2},y_j,(\tau^2+|\mu|^2)^{1/2}$ and spherical coordinates in $(\tau,\mu_j)$ are coordinates in the lift of
$\overline{\Tsc^*}_O X$ in this region in $[\overline{\Tsc^*}X;o_{\pa
  X}]$, while $x (\taub^2+|\mub|^2)^{1/2},y_j, (\taub^2+|\mub|^2)^{-1/2}$ and spherical
coordinates in $(\taub,\mub)$ are coordinates in the lift of
$\overline{\Tb^*}_O X$ in this region in $[\overline{\Tb^*}X;\pa\overline{\Tb^*}_{\pa X}
X]$, so by \eqref{eq:b-sc-covector-conv} (with $(\tau^2+|\mu|^2)^{1/2}
=x (\taub^2+|\mub|^2)^{1/2}$ and $x (\tau^2+|\mu|^2)^{-1/2}=(\taub^2+|\mub|^2)^{-1/2}$) and the equivalence of the
spherical variables, the identity map extends to
be smooth in this region.

Since the regions listed (for various values of $c,c_1,c_2$) provide an open cover of $[\overline{\Tb^*}X;\pa\overline{\Tb^*}_{\pa X}
X]$ as well as $[\overline{\Tsc^*}X;o_{\pa X}]$, the lemma follows.
\end{proof}

The advantage of this
perspective is that the blow-up is at fiber infinity, where symbolic
improvements are available, and thus are easy to handle, i.e.\ the
real work was done in {\em defining} $\Psibc(X)$ and analyzing its
properties, cf.\ the perspective of paired Lagrangian distributions in
the case of standard microlocal analysis mentioned above.

Indeed,
recall that on a manifold with corners $\cM$ with boundary defining
functions $\rho_1,\ldots,\rho_k$, the symbol space
$S^{m_1,m_2,\ldots,m_k}(\cM)$ consists of $\CI$ functions on the
interior $\cM^\circ$ which are in
$\prod_{j=1}^k\rho_j^{-m_j}L^\infty$, and they remain so under
iterated applications of $\CI$ b-vector fields on $\cM$, i.e.\ vector fields
tangent to all boundary hypersurfaces. Suppose now that $\cM$ is the radial
compactification in the fibers of a rank $p$ vector bundle $\cY$ on a manifold with
corners $\cX$, where $\rho_1,\ldots,\rho_{k-1}$ are boundary defining
functions of $\cX$, and $\rho_k$ boundary defining function of fiber
infinity, i.e.\ the new boundary hypersurface obtained via the radial
compactification of the fibers. Thus, using local product decompositions given by local trivializations, with $\xi_i$
coordinates on the fibers of the local trivialization, $\rho_k$ can be
taken to be locally equivalent to $|\xi|^{-1}$, with the norm being,
say, the standard Euclidean norm.
Then on the one hand this description of the symbol place applies
directly on $\cM$,
and on the other hand the usual symbol space on $\cY$ is also defined via
product decompositions given by local trivializations and requiring
that iterated applications of combinations of b-vector fields on $\cX$
and both linear and constant coefficient vector fields in the fibers
(such as $\xi_i\pa_{\xi_j}$,  resp.\ $\pa_{\xi_j}$), preserve
estimates in
$\prod_{j=1}^{k-1}\rho_j^{-m_j}|\xi|^{m_k}L^\infty=\prod_{j=1}^k\rho_j^{-m_j}L^\infty$. These
two definitions are {\em equivalent} as is immediate in $|\xi|<C$,
$C>0$ fixed, and in the region $|\xi|>C'$, $C'>0$, as the
$\CI(\overline{\RR^p})$- (i.e.\ classical order $0$ symbol) span of
$\xi_i\pa_{\xi_j}$ is the same as $\Vb(\overline{\RR^p})$ away from
$0$ (thus in this region),
cf.\ the discussion in the introduction on
symbols and b-vector fields on the radial compactification of
Euclidean space.

Such a discussion applies in particular to $\overline{\Tb^*}X$; thus the
fully (in both differentiability and decay sense) conormal (as opposed
to classical or polyhomogeneous) symbol space consists of $\CI$
functions on the interior, $\Tb^*_{X^\circ} X$, which satisfy a bound
$|a|\leq Cx^{-l}|(\taub,\mub)|^{m}$ stable under the iterative
application of b-vector fields on $\overline{\Tb^*}X$. In local coordinates, say where
$\taub>c|\mub|$ ($c>0$), these are spanned
by $x\pa_x$, $\pa_{y_j}$, $\taub\,\pa_{\taub}$,
$\taub\,\pa_{(\mub)_j}$, here it is the fibers of the vector bundle in which
the analogous compactification is taking place.
This space is {\em unchanged} by this blow-up of the
corner in the following sense:

\begin{lemma}\label{lemma:symbol-space-blowup}
The pullback of $S^{m,l}(\overline{\Tb^*}X)$ to $[\overline{\Tb^*}X;\pa\overline{\Tb^*}_{\pa X}
X]$ by the blow-down map is $S^{m,l,m+l}([\overline{\Tb^*}X;\pa\overline{\Tb^*}_{\pa X}
X] )$, where the symbolic order at the lift of fiber
infinity is $m$, at the lift of $x=0$ is $l$, while at the new front face the {\em sum of these two
  orders}, $m+l$.
\end{lemma}

Thus, for instance, symbols of order $0$ in both senses pull back to
symbols of order $0$ in all three senses.

\begin{proof}
Local boundary defining functions of $[\overline{\Tb^*}X;\pa\overline{\Tb^*}_{\pa X}
X]$ can be
given as follows. In the complement of the lift of $x=0$ (and the
b-zero section), which is the region \eqref{eq:fiber-infty-coords} (if
one takes the union over various $c>0$),
$\rho_{\infty,\loc}=x^{-1}(\taub^2+|\mub|^2)^{-1/2}$ defines fiber infinity, $x$ defines
the front face, so $|(\taub,\mub)|=\rho_{\infty,\loc}^{-1} x^{-1}$,
and
$$
x^{-l}|(\taub,\mub)|^{m}L^\infty=x^{-l-m}\rho_{\infty,\loc}^{-m}L^\infty.
$$
On
the other hand, in the complement of the lift of fiber infinity (and,
technically, away from the b-zero section, but the blow-up did not
change the structure near the b-zero section), i.e.\ in the region
\eqref{eq:bsc-corner-coords} (but again as a union as $c_1,c_2$ vary)
$\rho_{\bl,\loc}=x(\taub^2+|\mub|^2)^{1/2}$ defines the lift of $x=0$, and
$(\taub^2+|\mub|^2)^{-1/2}$ defines the front face, so
$x=\rho_{\bl,\loc}|(\taub,\mub)|^{-1}$, so
$$
x^{-l}|(\taub,\mub)|^{m}L^\infty=\rho_{\bl,\loc}^{-l}|(\taub,\mub)|^{l+m}L^\infty.
$$
In the intersection
of these two regions, $x(\taub^2+|\mub|^2)^{1/2}$ is bounded both from
above and from below by positive constants so these two spaces are
indeed the same.

Finally, (lifts of) b-vector fields on $\overline{\Tb^*}X$ with coefficients in
$\CI([\overline{\Tb^*}X;\pa\overline{\Tb^*}_{\pa X}
X])$ span b-vector fields on $[\overline{\Tb^*}X;\pa\overline{\Tb^*}_{\pa X}
X]$; this is a standard fact for blowing up a boundary face of a
manifold with corners, which is very easily checked in local
coordinates, using projective coordinates on the blown-up space.

Thus
the iterative regularity statement for the symbols is unaffected by
the blow-up process, proving the lemma.
\end{proof}

An inspection in the various regions discussed above shows that
$$
\rho_\infty=(x^2(\taub^2+|\mub|^2)+1)^{-1/2}=(\tau^2+|\mu|^2+1)^{-1/2}
$$
is a global defining function of fiber infinity,
$$
\tilde\rho_{\bl}=x(1+\taub^2+|\mub|^2)^{1/2}=(x^2+\tau^2+|\mu|^2)^{1/2}
$$
is a defining function of the lift of $x=0$ away from fiber infinity,
and in the form
$$
\rho_{\bl}=x(1+\taub^2+|\mub|^2)^{1/2}\rho_\infty=(x^2+\tau^2+|\mu|^2)^{1/2}(\tau^2+|\mu|^2+1)^{-1/2}
$$
even at fiber infinity.

Correspondingly:

\begin{Def}
Symbolic 2-microlocal pseudodifferential operators
are defined as elements of $\Psibc^{m,l}(X)$, where the
subscript `c' stands for the symbolic (conormal) behavior at {\em
  both} boundary hypersurfaces of $\overline{\Tb^*}X$, and also
for Schwartz kernels at the b-diagonal (corresponding, under the Fourier
transform, to fiber infinity) and the b-front face.
\end{Def}

Thus, at this
level one cannot actually see any difference in the 2-microlocal vs.\
the b-framework, since the symbol spaces are the same. This definition
assures that these conormal 2-microlocal operators form a *-algebra,
with a {\em full symbol calculus modulo $S^{-\infty,l}$, and a
  principal symbol calculus modulo $S^{m-1,l}$}.

It is a simple
matter of checking the composition rules, which in turn rely on left-
and right- reductions of full symbols (amplitudes) to see that the
subset given by quantizing {\em classical} symbols on $[\overline{\Tb^*}X;\pa\overline{\Tb^*}_{\pa X}
X]$ is also a *-algebra. Note that these are certainly {\em much
  larger than}
the pullback of classical symbols on $\overline{\Tb^*}X$; for
instance, a classical symbol supported near the interior of the front face $[\overline{\Tb^*}X;\pa\overline{\Tb^*}_{\pa X}
X]$ cannot be classical on $\overline{\Tb^*}X$ unless it is
of order $-\infty$ at the front face. As an illustration, if $m=l=0$,
so classical symbols are smooth functions on $[\overline{\Tb^*}X;\pa\overline{\Tb^*}_{\pa X}
X]$, even the
principal symbol cannot be the pullback of a continuous function on
$\overline{\Tb^*}X$ unless it vanishes. To see that this space is
indeed closed under composition, we just need to consider,
say, left reduction, i.e.\ eliminating dependence on the input variables of
a joint amplitude (depending on both the input, primed, variables and
the output, unprimed, variables), as composition ultimately follows
from this. But this takes the
form of an asymptotic sum, of terms of the form (constant multiples
of) $(-\pa_{t'})^j\pa_{y'}^\beta\pa_{\taub}^j\pa_{\mub}^\beta
a(t,y,t',y',\taub,\mub)|_{t'=t,y'=y}$, i.e. with an abuse of notation $(x'\pa_{x'})^j\pa_{y'}^\beta\pa_{\taub}^j\pa_{\mub}^\beta
a(x,y,x',y',\taub,\mub)|_{x'=x,y'=y}$, so as these differentiations
preserve classicality even on the blown up space (since the
differential operators lift to be smooth b-operators), this symbolic
approximation of the left
reduction of a classical amplitude is classical. Finally, the
remainder term gives rise to a classical operator in
$\Psibc^{-\infty,l}$, i.e.\ one in $\Psib^{-\infty,l}$, for which the composition
property is easily checked explicitly.

The place where 2-microlocal considerations appear is that now one can
microlocalize both {\em at the lift of fiber infinity}, of
$\overline{\Tb^*}X$, to $[\overline{\Tb^*}X;\pa\overline{\Tb^*}_{\pa X}
X]$, which is naturally identified with fiber infinity of
$\overline{\Tsc^*}X$ and we simply refer to as {\em fiber infinity} below, {\em and at new front face}, which can naturally
be identified with $[\overline{\Tsc^*}_{\pa X}X;o_{\pa X}]$ on the one
hand, and $\overline{\RR^+}\times\Sb^*_{\pa X} X$ on the other, and we
simply refer to as the {\em scattering face} below. For
instance:

\begin{Def} One says that a pseudodifferential operator
$A\in\Psibc^{m,l}(X)$, written as the left quantization, say, of $a\in
S^{m,l}$ modulo $\Psibc^{-\infty,l}(X)$, is elliptic
at a point $\alpha$ either at fiber infinity or at the scattering face if
$\alpha$ has a neighborhood $O$ in $[\overline{\Tb^*}X;\pa\overline{\Tb^*}_{\pa X}
X]$ such that  $|a|_O|$
is bounded below by a positive constant multiple of the product of
defining functions $\rho_\infty$ of fiber infinity raised to the power
$-m$, $\rho_{\scl}$ of the scattering face raised to the power
$-(m+l)$ and $\rho_{\bl}$ of the b-face raised to the power $-l$.

Similarly, for such an $\alpha$ one says that $\alpha\notin\WFscb'(A)$
if $\alpha$ has a neighborhood $O$ in $[\overline{\Tb^*}X;\pa\overline{\Tb^*}_{\pa X}
X]$ such that $a|_O$
 {\em vanishes to infinite order at both of these
  hypersurfaces} (inside $O$).
\end{Def}

This makes these two boundary hypersurfaces the
locus of the property of ellipticity and operator wave front set, and thus in the standard way
also the locus
of the wave front set of distributions, which is called the second microlocal sc-wave
front set: in view of the above identifications this is a subset of the
lift of the boundary hypersurfaces of $\overline{\Tsc^*}X$ to
$[\overline{\Tsc^*}X;o_{\pa X}]$. This {\em refines} the standard $\WFb$, which is at
fiber infinity in $\overline{\Tb^*}X$, usually identified with
$\Sb^*X$, or, via the $\RR^+$-action in the fibers of $\Tb^*X$, with
conic subsets of $\Tb^*X\setminus o_\bl$.

In fact, this discussion indicates that it is useful to introduce a pseudodifferential operator space $\Psiscb^{s,r,l}(X)$
with three orders corresponding to the three boundary hypersurfaces $[\overline{\Tb^*}X;\pa\overline{\Tb^*}_{\pa X}
X]$: the original two, giving the differential and b-decay orders, and
the new one, at the front face, giving the sc-decay order; we order
these as the fiber infinity order, the sc-decay order and finally the
b-decay order. This thus
arises from symbols in the class
$$
\rho_{\bl}^{-l}\rho_{\scl}^{-r}\rho_\infty^{-s} S^{0,0}(\overline{\Tb^*}X)=\rho_{\bl}^{-l}\rho_{\scl}^{-r}\rho_\infty^{-s} S^{0,0,0}([\overline{\Tb^*}X;\pa\overline{\Tb^*}_{\pa X}
X])
$$
with $\rho_\infty$, resp.\ $\rho_\scl$, resp.\ $\rho_\bl$ the defining
functions of the lift of fiber infinity, the scattering face, and the lift
of $x=0$, respectively. This is a subset of
$$
x^{-l}(\taub^2+|\mub|^2)^{(s+\max(r-(s+l),0))/2}S^{0,0}(\overline{\Tb^*}X),
$$
for, as discussed before, the symbol space is insensitive to the
blowup, and the weights give symbolic orders $l$ at the the lift of
$x=0$, $s+\max(r-(s+l),0)\geq s$ at fiber infinity, and
$l+s+\max(r-(s+l),0)\geq r$ at the scattering face. Thus,
$$
\Psiscb^{s,r,l}(X)\subset\Psibc^{s+\max(r-(s+l),0),l}(X),
$$
and membership in this smaller subspace is characterized purely by
symbolic properties, namely finite order vanishing conditions within the
class, concretely of order $\max(r-(s+l),0)$ at fiber infinity and
$l+s-r+\max(r-(s+l),0)=\max((s+l)-r,0)$ at the scattering
face. Therefore, $\Psiscb^{s,r,l}(X)$ can be easily seen to be invariant under composition and
adjoints by the left-reduction formula discussed above. Notice that
one has
$$
\Psibc^{m,l}(X)=\Psiscb^{m,m+l,l}(X).
$$

Alternatively, though we do not give details here, this class can be characterized by considering operators of the form
$A=A_1+A_2$ where
$A_1\in\Psibc^{s,r-s}(X)$ and $A_2\in\Psibc^{r-l,l}$ and the amplitude
of $A_1$ (i.e.\ when $A_1$ is written as a, say, left quantization) is microlocalized away from the lift of
$\overline{\Tb^*}_{\pa X}X$, while, modulo $\Psibc^{-\infty,l}(X)$, the amplitude of $A_2$ is
microlocalized away from fiber infinity, i.e.\ the lift of
$\pa\overline{\Tb^*}X$. A priori it is not clear that this is a
well-behaved space of operators under composition. This is made clear
by using a decomposition $A_1+A_2+A_3$ where
$A_1,A_2$ as above, but with amplitudes supported in a pre-specified
(small) neighborhood of fiber infinity, resp.\ the lift of
$\overline{\Tb^*}_{\pa X}X$, while
$$
A_3\in\Psibc^{s,r-s}(X)\cap\Psibc^{r-l,l}(X),\ {\text{indeed}}\ A_3\in\Psiscc^{-\infty,r}(X)
$$
has amplitude microlocalized
away from both of these boundary hypersurfaces, i.e.\ the essential
support of the amplitude only intersecting the interior of the front
face of the blow up (the sc-face). Thus, $A_3$ can be combined with
either of the other two terms for the use of standard $\Psibc$
composition results, while $A_1 A_2\in\Psibc^{-\infty,-\infty}(X)$ in
fact as can be seen by first of all noting that the symbolic expansion
is trivial due to the disjoint support (in the compactified sense) of
the two amplitudes, which implies $A_1
A_2,A_2A_1\in\Psibc^{-\infty,r-s+l}(X)$, but in addition in fact the
error term of the asymptotic expansions, given by an integral Taylor
series remainder term also possesses
additional decay properties. (The modulo $\Psibc^{-\infty,l}(X)$ part
of $A_2$ is easily seen to be harmless.)

The second microlocal Sobolev spaces, $\Hscb^{s,r,l}(X)$, are then
based on these pseudodifferential spaces, e.g.\ if all indices are
$\geq 0$, the space simply consists of $u\in L^2$ such that for all
$A\in\Psiscb^{s,r,l}$ (or simply for one elliptic one) $Au\in L^2$. We
refer to \cite[Section~5.3.9]{Vasy:Minicourse} for a detailed discussion.
Thus, 
$$
\Hb^{\tilde r,l}=\Hscb^{\tilde r,\tilde r+l,l}. 
$$
Also, sc-microlocally away from the zero section at $\pa
X$, i.e.\ away from the lift of $x=0$ in the b-blow up perspective, $\Hscb^{s,r,l}$ is just $\Hsc^{s,r}$, and the estimate below is
just the
estimate leading to the (strictly positive spectral parameter) limiting absorption principle in the scattering setting. However, the inclusion of $\Hscb^{s,r,l}$
into $\Hscb^{s',r',l}$ is no longer compact since there is no
improvement of the weight sc-microlocally at the zero section at $\pa
X$. Notice also that elements of $\Psiscb^{\tilde s,\tilde r,\tilde
  l}(X)$ give continuous linear maps $\Hscb^{s,r,l}\to\Hscb^{s-\tilde
  s,r-\tilde r,l-\tilde l}$.

The symbolic estimate, cf.\ Section~\ref{sec:symbolic}, is then:

\begin{prop}\label{prop:sc-b-Riem}
Suppose $S\subset[0,\infty)$ compact.
With $s,r,l$ as above, with $r\in\CI(\Ssc^*X)=\CI(\Sb^*X)$ monotone
along the $\sH_p$-flow, satisfying the threshold inequalities relative
to $-\frac{1}{2}$,, $s',r'$ arbitrary, there exists $C>0$ such that
\begin{equation}\label{eq:basic-sc-b-est-2}
\|u\|_{\Hscb^{s,r,l}}\leq
C(\|P(\sigma)u\|_{\Hscb^{s-2,r+1,l+2}}+\|u\|_{\Hscb^{s',r',l}}),\qquad\sigma\in S,
\end{equation}
provided $u\in\Hscb^{s'',r'',l}$ for some $s'',r''$ with
$r''>-\frac{1}{2}$ at the high regularity radial point (for $r$).
\end{prop}

\begin{proof}
This estimate again
arises from the `same' positive commutator argument as in the
b-algebra, with some slight modifications in the present second
microlocalized scattering algebra setting, combined with an elliptic
estimate.

Indeed, for
$\sigma$ in a fixed compact set,
the operator is elliptic near fiber infinity, so if $\tilde\varphi$ is
supported away from $0$, identically $1$ near infinity, and $B$ has
principal symbol $\tilde\varphi=\tilde\varphi(\tau^2+|\mu|^2)$ with wave front set on $\supp\tilde\varphi$ we have
elliptic estimates of the sort
\begin{equation}\label{eq:scb-elliptic}
\|\tilde Bu\|_{\Hscb^{s,r,l}}\leq C(\|P(\sigma)u\|_{\Hscb^{s-2,r,l}}+\|u\|_{\Hscb^{s',r',l'}})
\end{equation}
for any $s',r',l'$. (Notice that this in particular replaces the
elliptic estimate \eqref{eq:b-symbol-elliptic} in Section~\ref{sec:symbolic}.)

On the other hand, away from fiber infinity, as $\sH_p
(\tau^2+|\mu|^2) =0$ at $\pa X$ (and indeed one could simply use $p$
in place of $\tau^2+|\mu|^2$,
and then not just at $\pa X$), we can simply use a cutoff in $\tau^2+|\mu|^2$,
namely taking a function $\varphi$ identically $1$ near $0$, of
compact support,  and thus replace
\eqref{eq:b-symbol-choice} by
\begin{equation}\label{eq:scb-symbol-choice}
a=x^{-l-1}(\taub^2+|\mub|^2)^{(\tilde
  r-1/2)/2}\varphi(\tau^2+|\mu|^2)\psi(x),
\end{equation}
where $\tilde r=r-l$ satisfies the inequalities for $\tilde r$
discussed in Section~\ref{sec:symbolic}, adding also a regularizer as
in \eqref{eq:a-ep-regularize}. Thus, we
conclude, with $B$ as discussed around \eqref{eq:Hp-a-ep-pos} but
with a factor $\varphi(\tau^2+|\mu|^2)$ added to $b$, and with $\tilde
s,\tilde s'$ arbitrary (and irrelevant), that
\begin{equation}\label{eq:basic-sc-b-est-3}
\|Bu\|_{\Hscb^{\tilde s,r,l}}\leq
C(\|P(\sigma)u\|_{\Hscb^{\tilde s',r+1,l+2}}+\|u\|_{\Hscb^{s',r',l}}).
\end{equation}
Combining these two, with $\varphi=1-\tilde\varphi$ and $B=\Id-\tilde
B$, using
$$
\|u\|_{\Hscb^{s,r,l}}\leq \|\tilde Bu\|_{\Hscb^{s,r,l}}+\|Bu\|_{\Hscb^{s,r,l}}
$$
proves the proposition.
\end{proof}

Since microlocally near the scattering 0-section the second microlocal
Sobolev space is just the b-Sobolev space, the previous normal
operator estimate, \eqref{eq:eff-norm-op-b-est-base},
\begin{equation}
\|v\|_{\Hb^{\tilde r,l}}\leq C\|\tilde N(P(\sigma))v\|_{\Hb^{\tilde r',l+2}},
\end{equation}
can be used to improve (weaken the norm on) the error term on the
right hand side of \eqref{eq:basic-sc-b-est-2}. However, to avoid the technicalities which lengthened the
proof of Proposition~\ref{prop:impr-b-est}, we use the result of that
proposition directly.

One can then proceed as in Section~\ref{sec:normal} to prove an
analogue (and extension, since we use this proposition in the proof
below) of Proposition~\ref{prop:impr-b-est}:

\begin{prop}\label{prop:impr-scb-est}
Let $S\subset[0,\infty)$ compact. 

With $s,r,l$ as above, with $r\in\CI(\Ssc^*X)=\CI(\Sb^*X)$ monotone
along the $\sH_p$-flow, satisfying the threshold inequalities relative
to $-\frac{1}{2}$, $s',r'$ arbitrary, 
there exists $C$ such that
we have
the estimate
\begin{equation}\label{eq:impr-b-est-2}
\|u\|_{\Hscb^{s,r,l}}\leq
C(\|P(\sigma)u\|_{\Hscb^{s-2,r+1,l+2}}+\|u\|_{\Hscb^{s',
    r',l-\delta}}),\ \sigma\in S,
\end{equation}
provided $u\in\Hscb^{s'',r'',l}$ for some $s'',r''$ with
$r''>-\frac{1}{2}$ at the high regularity radial point (for $r$).
\end{prop}

\begin{proof}
First we remark that if $r=(s-1)+l$, and hence $r+1=(s-2)+(l+2)=s+l$, then the
slightly weakened version of
\eqref{eq:impr-b-est-2} is equivalent to the estimate of
Proposition~\ref{prop:impr-b-est}, since in the second microlocal
notation that states
\begin{equation}\begin{aligned}\label{eq:impr-scb-est-base-1}
\|u\|_{\Hscb^{s-1,s+l-1,l}}\sim\|u\|_{\Hb^{s-1,l}}\leq
C&(\|P(\sigma)u\|_{\Hb^{s-2,l+2}}+\|u\|_{\Hb^{r',l-\delta}})\\
&\sim C(\|P(\sigma)u\|_{\Hscb^{s-2,s+l,l+2}}+\|u\|_{\Hscb^{r',r'+l-\delta,l-\delta}})
\end{aligned}\end{equation}
with $r'$ arbitrary (but membership of $u$ in a stronger space
required). The only sense in which this is weaker than
\eqref{eq:impr-b-est-2} is that the latter is an elliptic lossless
estimate in the first order, $s$. This is easily remedied by using
Proposition~\ref{prop:sc-b-Riem}, which in combination with \eqref{eq:impr-scb-est-base-1} gives
\begin{equation}\begin{aligned}\label{eq:impr-scb-est-base-2}
\|u\|_{\Hscb^{s,r,l}}&\leq
C(\|P(\sigma)u\|_{\Hscb^{s-2,r+1,l+2}}+\|u\|_{\Hscb^{s-1,r,l}})\\
&\leq C'(\|P(\sigma)u\|_{\Hscb^{s-2,r+1,l+2}}+\|u\|_{\Hscb^{r',r'+l-\delta,l-\delta}}),
\end{aligned}\end{equation}
which is the estimate of the present proposition in this special case
as $r'$ is arbitrary.

Indeed, the same argument works if $r\leq (s-1)+l$, i.e.\ $s\geq
r-l+1$, for first let $\tilde s=r-l+1\leq s$, and apply
\eqref{eq:impr-scb-est-base-1} with $s$ replaced by $\tilde s$ to get
$$
\|u\|_{\Hscb^{\tilde s-1,\tilde s+l-1,l}}\leq
C(\|P(\sigma)u\|_{\Hscb^{\tilde s-2,\tilde s+l,l+2}}+\|u\|_{\Hscb^{r',r'+l-\delta,l-\delta}})
$$
and now by Proposition~\ref{prop:sc-b-Riem} we have
\begin{equation*}\begin{aligned}
\|u\|_{\Hscb^{s,r,l}}&\leq
C(\|P(\sigma)u\|_{\Hscb^{s-2,r+1,l+2}}+\|u\|_{\Hscb^{\tilde s-1,r,l}})\\
&\leq C'(\|P(\sigma)u\|_{\Hscb^{s-2,r+1,l+2}}+\|P(\sigma)u\|_{\Hscb^{\tilde
    s-2,r+1,l+2}}+\|u\|_{\Hscb^{r',r'+l-\delta,l-\delta}})\\
&\leq C''(\|P(\sigma)u\|_{\Hscb^{s-2,r+1,l+2}}+\|u\|_{\Hscb^{r',r'+l-\delta,l-\delta}}),
\end{aligned}\end{equation*}
proving \eqref{eq:impr-b-est-2} under the additional assumption $s\geq
r-l+1$.

Now we dualize the argument, much as in the proof of
Proposition~\ref{prop:impr-b-est}. Namely, first re-write the existing
estimate, adding a subscript $+$ to denote that the order $r$ is
monotone increasing along the Hamilton flow; the estimate then holds
if $s_+\geq r_+-l_++1$. But then the analogous estimate also holds with
the monotonicity relative to the Hamilton flow reversed, and with
$P(\sigma)$ (potentially) replaced by $P(\sigma)^*$, namely
\begin{equation}\begin{aligned}\label{eq:scb-dual-est-1}
\|u\|_{\Hscb^{s_-,r_-,l_-}}\leq C(\|P(\sigma)^*u\|_{\Hscb^{s_--2,r_-+1,l_-+2}}+\|u\|_{\Hscb^{r'_-,r_-'+l-\delta,l_--\delta}}),
\end{aligned}\end{equation}
valid if $s_-\geq r_--l_-+1$. We now apply this with $r_-=-(1+r_+)$,
$l_-=-(l_++2)$,
since we need this dual estimate with
$\Hscb^{s_-,r_-,l}=(\Hscb^{s_+-2,r_++1,l_++2})^*$. The desired choice
of $s_-$ would be $2-s_+$, but as $s_+\geq r_+-l_++1$, $2-s_+\leq
-r_++l_+-1=r_--l_-$ which cannot be satisfied if $s_-=2-s_+$ and
$s_-\geq r_--l_-+1$. Thus, we choose
$s_-=r_--l_-+1=-1-r_++l_++2+1=3-(r_+-l_++1)\geq 2-s_+$ for applying
\eqref{eq:scb-dual-est-1}, which gives us a version of the dual
estimate we actually want
\begin{equation}\label{eq:scb-adjoint-est}
\|u\|_{\Hscb^{2-s_+,-1-r_+,-l_+-2}}\leq C(\|P(\sigma)^*u\|_{\Hscb^{-s_+,-r_+,-l_+}}+\|u\|_{\Hscb^{r'_-,r_-'+l-\delta,l_--\delta}}),
\end{equation}
with the proviso that the norms in \eqref{eq:scb-dual-est-1} are
stronger than in this estimate in the sc-differential sense. However,
we argue as in the proof of Proposition~\ref{prop:impr-b-est}. Namely,
ignoring finite dimensional solvability obstacles which are handled
exactly in the same way as in Proposition~\ref{prop:impr-b-est}, the
estimate \eqref{eq:scb-dual-est-1} lets us solve $P(\sigma)u=f$ when
$f\in\Hscb^{-s_-,-r_-,-l_-}=\Hscb^{(r_+-l_++1)-3,r_++1,l_++2}$, with
the resulting
$u\in\Hscb^{2-s_-,-1-r_-,-l_--2}=\Hscb^{(r_+-l_++1)-1,r_+,l_+}$. Now,
if $f$ has the additional regularity $f\in\Hscb^{s_+-2,r_++1,l_++2}$
(note that $s_+-2\geq (r_+-l_++1)-3$), the elliptic estimates (or
directly Proposition~\ref{prop:sc-b-Riem}) imply that $u\in
\Hscb^{s_+,r_+,l_+}$, with an estimate for
$\|u\|_{\Hscb^{s_+,r_+,l_+}}$, which via a pairing argument as before \eqref{eq:impr-b-est-dual-proved} also gives the adjoint estimate \eqref{eq:scb-adjoint-est}.

We finally eliminate the restriction on $s_+$ by noting that reversing
the role of $P(\sigma)$ and $P(\sigma)^*$ (i.e.\ obtaining the
original estimate for $P(\sigma)^*$, thus dual estimate
\eqref{eq:scb-adjoint-est} for $(P(\sigma)^*)^*=P(\sigma)$ in place of
$P(\sigma)^*$, gives the desired estimate when the orders satisfy
$2-s_+\leq 2-(r_+-l_++1)\leq (-1-r_+)-(-l_+-2)+1$, which, when
relabeling $2-s_+$ as $s_-$, $-1-r_+$ as $r_-$, $-l_+-2$ as $l_-$, is
exactly the desired estimate for (a bigger range than) $s_-\leq
r_--l_-+1$. Since the estimate for monotone weights with either
direction of monotonicity are completely analogous, this completes the
proof of the proposition.
\end{proof}

Again, this has the errorless consequence if $P(0)$ has trivial
nullspace, since the error term can be dropped:

\begin{thm}\label{thm:2-micro-spaces}
Suppose that $s,r,l$ are as in Proposition~\ref{prop:impr-scb-est}.
Suppose also that $P(0)$ has trivial nullspace on
$\Hscb^{\infty,\infty,l}=\Hb^{\infty,l}$. Then there are
$\sigma_0>0$, $C>0$ such that for $\sigma\in[0,\sigma_0]$,
$$
P(\sigma)^{-1}:\Hscb^{s-2,r+1,l+2}\to\Hscb^{s,r,l}
$$
has uniform bounds:
$$
\|P(\sigma)^{-1}f\|_{\Hscb^{s-2,r+1,l+2}}\leq C\|f\|_{\Hscb^{s,r,l}},\
\qquad \sigma\in S.
$$
\end{thm}

In particular, this proves Corollary~\ref{cor:2-micro-spaces}.

\section{The Kerr setting}\label{sec:Kerr}
In this section we concentrate on the analytic
aspects of the Kerr setting that extend those of the setting of the earlier parts of this
this paper, which describe the Kerr wave operator near the Minkowski
(or Euclidean) end.
In order to deal with the full Kerr setting, given the work already
done in this paper, we need to deal with a
neighborhood of the event horizon. All the relevant features there
were handled by the microlocal Fredholm analysis of
\cite{Vasy-Dyatlov:Microlocal-Kerr}, though here we frame the operator
as one equipped with a Cauchy hypersurface inside the black hole,
rather than complex absorption, following \cite{Hintz-Vasy:Semilinear}. One then considers a manifold
with boundary $X$ as before, except that $X$ has, in addition to the
asymptotically Euclidean boundary, which we denote $\pa_+ X$, an
`artificial boundary' $\pa_-X$, given by $\mu=\mu_0$, $\mu_0<0$, where $\mu$ is a
smooth function on $X$, with $\mu-\mu_0$ a defining function of
$\pa_-X$; this acts as a Cauchy hypersurface below. We still use $x=r^{-1}$ to denote a defining function of
$\pa_+X$, and assume that near $\pa_+ X$ the same structure as before
holds. In the actual Kerr spacetime, we have $x=r^{-1}$, $\mu=1-\frac{2m}{r}$, $m$
the mass of the black hole.

As $\pa_-X$ is not considered a `real' boundary, below the notation
$\Diffb$, etc., refers to the b-structure at $\pa_+X$, and the
standard space of differential operators at $\pa_-X$. We then consider
\begin{equation}\label{eq:P-sigma-form-Kerr}
P(\sigma)=P(0)+\sigma Q-\sigma^2,\qquad P(0)\in x^2\Diffb^2(X),\ Q\in
S^{-2-\delta}\Diffb^1(X)
\end{equation}
$P(\sigma)$ is symmetric for $\sigma$ real (so $P(0)$ and $Q$
symmetric; here we explicitly make $Q$ independent of $\sigma$), and
$P(0)-\Delta_g$ in $S^{-2-\delta}\Diffb^2(X)$ near $\pa_+X$, so in
particular $P(0)$ is elliptic in $\Psibc^{2,0}$ in $x<x_0$ for some
$x_0>0$, while $P(0)$ is a wave operator near $\pa_-X$ with $\mu$
timelike satisfying microlocal hypotheses on the structure of the
Hamilton flow that we describe next.

We remark that for the actual Kerr wave operator $\Box$, for suitable positive
$\alpha\in\CI(X)$, with $\alpha-1\in x\CI(X)$, $\alpha\Box$ is of this
form ($\alpha\sim 1-2m/r=1-2mx$ near $x=0$, $m$ the mass
of the black hole, $\sim$ in the sense
that the difference is $O(x^2)$), which then remains formally
self-adjoint relative to $\alpha^{-1}$ times the metric volume
density. This form is the consequence of a simple computation near $x=0$
using Boyer-Lindquist coordinates (which need to be modified away from
$x=0$ to deal with the horizon, much as in the Kerr-de Sitter setting) in which the dual metric is
\begin{equation*}\begin{aligned}
G=-\rho^{-2}\Big(&\Delta_r\pa_r^2
+\frac{1}{\sin^2\theta}(a \sin^2\theta\,\pa_{ t}+\pa_{ \phi})^2+\pa_\theta^2-\frac{1}{\Delta_r}
((r^2+a^2)\pa_{t}+a \,\pa_{\phi})^2\Big)
\end{aligned}\end{equation*}
with $m,a$ constants, $0\leq |a|<m$,
\begin{equation*}\begin{aligned}
&\rho^2=r^2+a^2\cos^2\theta,\ \Delta_r=r^2+a^2-2mr.
\end{aligned}\end{equation*}
Indeed, expanding the squares yields, with $h^{-1}_{S^2}$ the
spherical dual metric, $k_S$ a smooth 2-tensor family (in $r^{-1}$) on the sphere,
\begin{equation*}\begin{aligned}
&((1-2m/r)^{-1}+O(r^{-2}))\pa_t^2-(1-2m/r+O(r^{-2})\pa_r^2\\
&\qquad-r^{-2}(h^{-1}_{S^2}+O(r^{-1})k_S)-O(r^{-3})\pa_t \pa_\phi,
\end{aligned}\end{equation*}
where the cancellation
between the two $\pa_t \pa_\phi$ cross terms,
$$
(r^2+a^2\cos^2\theta)^{-1}(2a\pa_t\pa_\phi-2\Delta_r^{-1}a(r^2+a^2)\pa_t\pa_\phi),
$$
with $\Delta_r^{-1}(r^2+a^2)=1+O(r)$ giving the cancellation, is
crucial in the decay in $O(r^{-3})\pa_t\pa_\phi$ that we use
here. Moreover, in the actual wave operator, the $t$-translation
invariance means that there are no first or zeroth order terms arising
from the non-decaying $(1-2m/r)^{-1}+O(r^{-2}))\pa_t^2$ term of the dual
metric.
Thus, with $\alpha$ being the reciprocal of the coefficient of $\pa_t^2$
(i.e.\ the squared metric length of $dt$), this gives the desired form
near $\pa_+X$.

The description of the hypotheses away from $\pa_+X$ follows
\cite[Section~2.2]{Vasy-Dyatlov:Microlocal-Kerr} closely (with some
minor notational changes); indeed away from $\pa_+X$ behaves much like
a Kerr-de Sitter metric.
In order to deal with the structure of the operator away from $\pa_+X$
we also assume
that the principal symbol $p$ of $P(0)$
has a non-degenerate zero set $\Sigma=\Sigma_+\cup\Sigma_-$ with the
subscript denoting two connected components (i.e.\ $dp$ does not vanish when $p$
does), and has non-degenerate, in the sense of
\cite[Section~2.2]{Vasy-Dyatlov:Microlocal-Kerr}, generalized radial sets $\Lambda^-_\pm$ which are
submanifolds of $\Tb^*_{X^\circ}X$ (of course as $\Lambda^-$ lies
over the interior of $X$, there is no distinction between $\Tb^*$ and
$T^*$ near $\Lambda^-$, i.e.\ the `b' nature is irrelevant) (in the actual Kerr space these
would be the two halves of $N^*Y$, where $Y\subset X$ is the event
horizon), and $\Lambda^-_\pm$ are (normal) sources ($-$)/sinks ($+$) for $\sH_p$ within
$\Tb^*X$, with $\Lambda^-_\pm\subset\Sigma_\mp$, with the
superscript standing for the black hole end. We write $L^-$ for the
image of $\Lambda^-$, with the appropriate subscripts, in the
cosphere bundle $\Sb^*X$. The subprincipal symbol of $P(\sigma)$ at $L_\pm^-$ also
plays a role below; we assume that, with principal symbols considered
as homogeneous functions (of degree $1$) on $\Tb^*X\setminus o$, and
$\rho_\infty$ being a defining function of fiber infinity (with
asymptotic degree $-1$ homogeneity, and with the
particular choice not playing any role)
$$
\sigma_1\Big(\frac{1}{2i}(P(\sigma)-P(\sigma)^*)\Big)|_{\Lambda^-_\pm}=(\im\sigma)\sigma_1(Q)=\kappa(\im\sigma)
\rho_\infty^{-1}\sH_p \rho_\infty|_{\Lambda^-_\pm}
$$
for a function $\kappa>0$; for notational simplicity we take $\kappa$ constant.

Such a
Cauchy hypersurface plus radial point plus the b-infinity setting
is called {\em non-trapping} if (null)-bicharacteristics of $p$ have
the following behavior:
\begin{enumerate}
\item
All bicharacteristics in $\Sigma_+\setminus L_-^-$ tend to $L_-^-$ in the backward 
direction and to $\pa_- X$ (meaning $\Tb^*_{\pa_-X}X$) in the forward direction. 
\item
All bicharacteristics in $\Sigma_-\setminus L_+^-$ tend to $L_+^-$ in the forward 
direction and to $\pa_-X$ in the backward direction. 
\end{enumerate}
Thus, in $\Sigma_+$, $L^-$ can only be reached in the backward
direction and $\pa_- X$ in the forward direction; in $\Sigma_-$ forward and backward are reversed. (Microlocal analysis in $X^\circ$ is only at fiber
infinity, i.e.\ at infinite momentum, $\pa\overline{T^*}_{X^\circ}
X=S^*_{X^\circ} X$.) This reversal corresponds to the need for propagating
estimates forward in $\Sigma_+$ and backward in $\Sigma_-$ for the
operator, with the reverse for the adjoint, corresponding to the
`causal' (Cauchy problem) inverse we need.

Thus, with $x=r^{-1}$ as before, $\chi_-$ supported in $(0,\infty)$,
identically $1$ on $(x_1,\infty)$,
$x_1>0$, $\tilde\chi_-$ supported in $(0,\infty)$, identically $1$ on
$\supp\chi_-$, and propagating estimates forward along $\sH_p$ in
$\Sigma_+$, backwards in $\Sigma_-$, one has estimates
$$
\|\chi_- u\|_{\bar H^s}\leq C(\|\tilde\chi_- P(\sigma)u\|_{\bar
  H^{s-1}}+\|u\|_{\bar H^{-N}})
$$
holding if $s>\frac{1}{2}-\kappa\im\sigma$, and dually, reversing the
direction of propagation,
$$
\|\chi_- u\|_{\dot H^{\tilde s}}\leq C(\|\tilde\chi_-
P(\sigma)^*u\|_{\dot H^{\tilde s-1}}+\|u\|_{\dot H^{-N}}),
$$
holding if $\tilde s<\frac{1}{2}-\kappa\im\sigma$; for the Fredholm estimate
the relevant value of $\tilde s$ is $-(s-1)=1-s$, so the inequalities
for $s$ and $\tilde s$ are equivalent. Here $\bar H$ and $\dot H$ are,
respectively, the Sobolev spaces of extendible and supported
distributions, with this property referring to the Cauchy hypersurface; see \cite[Section~2.1]{Hintz-Vasy:Semilinear} in a
context closely related to the present one.

As the function spaces are simpler, we first consider the b-, rather than the second microlocal,
estimate on the asymptotically Euclidean end. Thus, we let $\tilde r$ be a
variable order as in the main theorem,
with $\beta>0$,
\begin{equation*}
\tilde
r=\frac{1}{2}-(l+1)\pm\beta\frac{\taub}{(\taub^2+|\mub|^2)^{1/2}},\qquad |l+1|<\frac{n-2}{2}.
\end{equation*}
The estimate on the asymptotically Euclidean end,
where $\chi_+$ can be taken to be $1-\chi_-$, while $\tilde\chi_+$ is
supported in $[0,x_0)$, identically $1$ on $\supp\chi_+$, is, with $0<K<\beta$,
$$
\|\chi_+ u\|_{\Hb^{\tilde r,l}}\leq C(\|\tilde\chi_+
P(\sigma)u\|_{\Hb^{\tilde r-1,l+2}}+\|u\|_{\Hb^{\tilde r-K,l'}})
$$
and dually
$$
\|\chi_+ u\|_{\Hb^{1-\tilde r,-l-2}}\leq C(\|\tilde\chi_+
P(\sigma)^*u\|_{\Hb^{-\tilde r,-l}}+\|u\|_{\Hb^{\tilde r-K,l'}})
$$
where $|l+1|<\frac{n-2}{2}$, and $\tilde r$ a variable order.

So let
$s$ be a variable order satisfying $s>\frac{1}{2}-\kappa\im\sigma$ near the
radial sets at the event horizon, indeed for convenience $s$ constant
outside $x<x_0$, and $s=\tilde r$ near $\pa_-X$, we can combine these
two (simply use a partition of unity: in the transitional region
$P(\sigma)$ is elliptic, so the order is irrelevant) to obtain, with $0<K<\beta$,
$$
\|u\|_{\bHb^{s,l}}\leq
C(\|P(\sigma)u\|_{\bHb^{s-1,l+2}}+\|u\|_{\bHb^{s-K,l'}}),
$$
and dually
$$
\|u\|_{\dHb^{1-s,-l-2}}\leq
C(\|P(\sigma)^*u\|_{\dHb^{-s,-l}}+\|u\|_{\dHb^{s-K,l'}}).
$$
We can then argue exactly as in Section~\ref{sec:normal} in the proof of
Proposition~\ref{prop:remove-cpt-error} to remove the compact error
term. We state the result both for $P(\sigma)$ and for its adjoint
$P(\sigma)^*$ since now the dual spaces are not of the same kind due
to the artificial boundary, as extendible and supported distributions
are interchanged in dualization:

\begin{thm}
Suppose $P(0):\bHb^{s,l}\to\bHb^{s-2,l+2}$ has trivial nullspace. Then
there exist $\sigma_0>0$, $C>0$ such that for $|\sigma|<\sigma_0$,
$$
\|u\|_{\bHb^{s,l}}\leq C\|P(\sigma)u\|_{\bHb^{s-1,l+2}}.
$$

An analogous statement holds for $P(\sigma)^*$ on the dual spaces,
provided $P(0):\dHb^{1-s,-l-2}\to\dHb^{-s,-l}$ has trivial nullspace.

As a consequence, if both $P(0)$ and $P(0)^*$ have trivial nullspace on the
respective spaces, then there is $\sigma_0>0$ such that for $|\sigma|<\sigma_0$
$$
P(\sigma):\{u\in \bHb^{s,l}: P(\sigma)u\in
\bHb^{s-1,l+2}\}\to \bHb^{s-1,l+2}
$$
is invertible, with the inverse possessing uniform bounds.
\end{thm}

Finally the second microlocal version is in some sense simpler, in
that the order $s$ can be taken to be constant. Namely, with
$s>\frac{1}{2}-\kappa\im\sigma$, and with
$$
r=-\frac{1}{2}\pm\beta\frac{\taub}{(\taub^2+|\mub|^2)^{1/2}}=-\frac{1}{2}\pm\beta\frac{\tau}{(\tau^2+|\mu|^2)^{1/2}},
$$
by Proposition~\ref{prop:sc-b-Riem}
one has, for $u\in\Hscb^{s',r-K,l'}$, $0<K<\beta$, $s',r',l'$ arbitrary
$$
\|\chi_+ u\|_{\Hscb^{s,r,l}}\leq C(\|\tilde\chi_+
P(\sigma)u\|_{\Hscb^{s-1,r+1,l+2}}+\|u\|_{\Hscb^{s',r',l'}})
$$
and dually
$$
\|\chi_+ u\|_{\Hscb^{1-s,-1-r,-l-2}}\leq C(\|\tilde\chi_+
P(\sigma)^*u\|_{\Hscb^{-s,-r,-l}}+\|u\|_{\Hscb^{s',r',l'}})
$$
where $|l+1|<\frac{n-2}{2}$, which together with the $\chi_-$
estimates above yields, for $u\in\bHscb^{s',r-K,l'}$, $0<K<\beta$, with $P(\sigma)u\in \bHscb^{s-1,r+1,l+2}$
$$
\|u\|_{\bHscb^{s,r,l}}\leq C(\|P(\sigma)u\|_{\bHscb^{s-1,r+1,l+2}}+\|u\|_{\bHscb^{-N,-N',l'}}),
$$
and dually (under analogous conditions)
$$
\|u\|_{\dHscb^{1-s,-1-r,-l-2}}\leq C(\|P(\sigma)^*u\|_{\dHscb^{-s,-r,-l}}+\|u\|_{\dHscb^{-N,-N',l'}}).
$$
Again, arguing exactly as in the proof of
Proposition~\ref{prop:remove-cpt-error} we deduce:

\begin{thm}
Suppose $P(0):\bHscb^{s,r,l}\to\bHscb^{s-2,r,l+2}$ has trivial nullspace. Then
there exist $\sigma_0>0$, $C>0$ such that for $|\sigma|<\sigma_0$,
$$
\|u\|_{\bHscb^{s,r,l}}\leq C\|P(\sigma)u\|_{\bHscb^{s-1,r+1,l+2}}.
$$

An analogous statement holds for $P(\sigma)^*$ on the dual spaces,
provided $P(0)^*:\dHscb^{1-s,-r-1,-l-2}\to\dHscb^{-s,-r,-l}$ has trivial
nullspace.

As a consequence, if both $P(0)$ and $P(0)^*$ have trivial nullspace on the
respective spaces, then there is $\sigma_0>0$ such that for $|\sigma|<\sigma_0$
$$
P(\sigma):\{u\in \bHscb^{s,r,l}: P(\sigma)u\in
\bHscb^{s-1,r+1,l+2}\}\to \bHscb^{s-1,r+1,l+2}
$$
is invertible.
\end{thm}

\begin{rem}\label{rem:Kerr-perturb}
As noted in the elliptic setting in Remark~\ref{rem:bd-state-stable},
if $P_b(\sigma)$ is a family
of operators of the form discussed above for $P(\sigma)$ continuously
depending on a parameter $b\in\RR^k$, and $P_0(0)$ is invertible, then
there is $\ep>0$ such that
the same holds for $P_b(0)$ for $|b|<\ep$, i.e.\ the hypothesis of the
theorem is perturbation stable. In this setting this is an immediate consequence
of the stability discussion of
\cite[Section~2.7]{Vasy-Dyatlov:Microlocal-Kerr} for non-elliptic
Fredholm problems.

In particular, the Kerr family is a perturbation of the Schwarzschild
family in this sense. Thus, the invertibility of $P_0(0)$ for a
Schwarzschild metric of a given black hole mass $m>0$ implies the invertibility of
$P_b(0)$, $b=(m',a')$,
for Kerr metrics of black hole mass $m'$ close to $m$ and angular momentum $a'$
close to $0$. If $P_b(\sigma)$ is the $-t_*$-Fourier transformed wave
operator on scalar functions, $P_0(0)$ is well-known to have trivial
nullspace (scalar `mode stability' of Schwarzschild), going back to
Regge and Wheeler \cite{Regge-Wheeler:Schwarzschild}, and the the same
holds for the adjoint. The simplest way to see this is that by
a regularization (in terms of decay) argument one shows that for
functions in the relevant nullspace, $0=\langle
P_0(0)u,u\rangle_{r>r_0}=\|du\|^2_{r>r_0}$); here $r_0$ is the event horizon,
and there are no boundary terms at the event horizon since the
operator is characteristic there. (Notice that $du$ is in $L^2$
automatically since $d$ is a scattering differential operator, so in
terms of decay as a b-operator, it is order $-1$, and the relevant
nullspace is in $\Hb^{\infty,-\frac{1}{2}+\ep}$ for all $\ep>0$.) A different way of proceeding, using
spherical symmetry, is that the problem reduces to a one-dimensional
Schr\"odinger operator on $(r_0,\infty)$,
with a non-negative potential that decays inverse quadratically as
$r\to\infty$ (and smooth at the event horizon), and the relevant nullspace is
on functions that are smooth at $r_0$ and are in the b-Sobolev spaces discussed in this paper
at the other end. Thus, by the perturbation stability corresponding
results also hold for slowly rotating Kerr black holes. In particular, in this slowly rotating case
it is not necessary to use Whiting's transformation \cite{Whiting:Mode} and its extension by
Shlapentokh-Rothman \cite{Shlapentokh-Rothman:Quantitative}.
\end{rem}

\section{Bound states and half-bound states}\label{sec:resonances}
Finally we consider the case of $P(0)$ having a non-trivial nullspace,
at first in the elliptic, in the scattering sense, setting. Before
proceeding, we mention the references \cite{Jensen-Kato:Spectral} and
\cite[Section~3.3]{Dyatlov-Zworski:Scattering} for the precise
description of the resolvent near zero energy in potential scattering
for sufficiently rapidly decaying, resp.\ spatially compactly
supported, perturbations of the Euclidean Laplacian.

What
happens in this setting is highly dimension dependent for the following reason.
For the Fredholm theory, we considered
$P(\sigma)$ as a map
$$
\{u\in\Hb^{\tilde r,l}:\ P(\sigma)u\in \Hb^{\tilde r-1,l+2}\}=\cX_\sigma\to\Hb^{\tilde r-1,l+2},
$$
and the domain here depends on $\sigma$ (though the estimates employed utilize
standard variable order Sobolev spaces). Now, in general, in standard meromorphic Fredholm
theory, or indeed simply Fredholm theory depending continuously on a
parameter, one would decompose the domain as the direct sum of the
nullspace of $P(0)$ and a complementary space, and analogously the target space
would be decomposed into the range of $P(0)$ and a complementary
space, and consider $P(\sigma)$ as a `block matrix' with respect to
this decomposition. Here the choice of the complementary space is
flexible, but the nullspace and the range play a key role. Under our
assumptions, if $u\in\Ker P(0)$, then for $\sigma\neq 0$,
$u\in\cX_\sigma$ if and only if $P(\sigma)u\in\Hb^{\tilde r-1,l+2}$,
which is the case if and only if $\sigma^2 u\in \Hb^{\tilde r-1,l+2}$,
i.e.\ $ u\in \Hb^{\tilde r-1,l+2}$. Here the differential order is not
an issue, since elements of the nullspace are in $\Hb^{\infty,l'}$ for
all $l'<-1+\frac{n-2}{2}$, but decay is. Since we need
$l>-1-\frac{n-2}{2}$, the statement $u\in\cX_\sigma$ amounts to
$1-\frac{n-2}{2}<-1+\frac{n-2}{2}$,
i.e.\ $n>4$. Thus, for $n\geq 5$, which is exactly when it is
guaranteed that there are no
half-bound states, one can proceed rather directly with perturbation
theory in a Fredholm setting. Namely, choose a complementary subspace to $\Ker P(0)$ in
$\Hb^{\tilde r,l}$ (e.g.\ the orthocomplement), and a
complementary subspace to $\Ran P(0)$ in $\Hb^{\tilde
  r-1,l+2}$; consider the complement, resp.\ $\Ker P(0)$, intersected with $\cX_\sigma$ as
the actual decomposition of the domain, keeping in mind that $\Ker
P(0)\subset\cX_\sigma$ as we already noted (with $l$ chosen
appropriately, namely $l<-3+\frac{n-2}{2}$, which is possible for
$n\geq 5$). Assuming that $P(\sigma)=P(0)-\sigma^2$ with $P(0)=P(0)^*$ for simplicity,
with the decompositions being $\Ker P(0)^\perp\oplus\Ker P(0)$ and
$\Ran P(0)\oplus\Ran P(0)^\perp$, with $\perp$ relative to the
$L^2_g$-inner product to take advantage of the formal self-adjointness, using
that $\Ran P(0)^\perp=\Ker P(0)$ (as $\Ker
P(0)\subset\Hb^{\infty,l'}$, the $L^2$-orthocomplement makes
sense in $\Hb^{\tilde r-1,l+2}$ as $l+2+l'$ is taken $\geq 0$ for
$l'+2<1+\frac{n-2}{2}$, arbitrarily close to equality, so even with
$l$ barely greater than $-1-\frac{n-2}{2}$, the non-negativity of the
sum can be arranged),
$P(\sigma)$ becomes the `block matrix'
$$
\begin{pmatrix} P(0)-\sigma^2&0\\0&-\sigma^2\end{pmatrix}.
$$
Now the top left entry is invertible for sufficiently small $\sigma$
for the same reasons as in the case $\Ker P(0)$ and $\Ker P(0)^*$ were
trivial, namely the arguments of Section~\ref{sec:normal} apply. The
actual inverse is thus
$$
\begin{pmatrix} (P(0)-\sigma^2)^{-1}&0\\0&-\sigma^{-2}\end{pmatrix},
$$
and thus on the complement of $\Ker P(0)$, the resolvent remains
uniformly bounded as $\sigma\to 0$, while on $\Ker P(0)$ it has the
expected $O(|\sigma|^{-2})$ asymptotic behavior.

Notice that what we really needed above is not that $n\geq 5$, rather
that the nullspace of $P(0)$ is sufficiently decaying, so even if
$n\geq 3$ merely, and $\Ker P(0)\subset \Hb^{\infty,l+2}$ for some
$l>-1-\frac{n-2}{2}$, which is a strengthening of the a priori decay
order by $4-n+\ep$, $\ep>0$ arbitrarily small, i.e.\ a strengthening
by just a bit more than one order of decay if $n=3$, all the above arguments go through unchanged. Recall that
the asymptotic behavior of elements of the nullspace is given by an
expansion in terms of resonances, which means that the asymptotics are
given by
powers $\frac{n-2}{2}+\sqrt{(\frac{n-2}{2})^2+\lambda_j}$ of $x$ (plus
natural numbers, with possible logarithms), where $\lambda_j$ are the
eigenvalues of $\Delta_Y$, which is to say that these terms lie in
$$
\cap_{\ep>0}\Hb^{\infty,-1+\sqrt{(\frac{n-2}{2})^2+\lambda_j}-\ep}.
$$
For instance, for $n=3$ and $Y$ the standard sphere,
$\lambda_j=j(j+1)$, $j\geq 0$ integer, and the exponent is $-\frac{1}{2}+j-\ep$, which
satisfies the desired inequality for $j\geq 2$ (which is stronger than
the requirement for being in $L^2$: $j\geq 1$). Thus, the only
`non-trivial' part of the nullspace from this perspective is the one
with non-vanishing $j=0,1$ terms. We note that this also limits the
dimension of the more complicated part of the nullspace to the sum of
the dimension of the resonant states (thus of the eigenspaces of the
boundary Laplacian) for $j=0,1$.

In order to extend this result, we proceed from a somewhat different
perspective which is common in scattering theory. We consider a perturbation $\tilde P(\sigma)$ of $P(\sigma)$ such
that $\tilde P(\sigma)-P(\sigma)\in S^{-2-\delta}\Diffb^2(X)$ and such that
$\tilde P(0)$ is invertible as a map $\Hb^{\tilde r,l}\to\Hb^{\tilde
  r,l+2}$, $|l+1|<\frac{n-2}{2}$. Note that such $\tilde P(\sigma)$ exists
and
$$
V(\sigma)=\tilde P(\sigma)-P(\sigma) \in S^{-2-\delta}\Diffb^2(X)
$$
may even be taken to be compactly supported
(and even order $-\infty$) in $X^\circ$, one way of doing this is
approximating elements of the nullspaces of $P(0)$ and its adjoint by
compactly supported $\CI$ functions, another is to add to $P(0)$ a
positive `potential' which is sufficiently large in a sufficiently
large compact set. Notice that (even under the weaker assumption),
$\cX_\sigma$ is both the domain for $P(\sigma)$ and for $\tilde
P(\sigma)$ for all $\sigma$ (as Fredholm operators). Moreover, by our
main results, $\tilde P(\sigma)$ is invertible for sufficiently small
$|\sigma|$ (and $\im\sigma$ of the correct indefinite sign). Then to work on a
domain independent of $\sigma$ we proceed in a standard manner in
scattering theory, and consider $P(\sigma)\tilde
P(\sigma)^{-1}:\cY\to\cY$; inverting this is equivalent to
inverting $P(\sigma)$ in the desired manner.

At first we assume that
$$
V=V(\sigma)
$$
is independent of $\sigma$; later we will relax this assumption.
Now,
$$
P(\sigma)\tilde
P(\sigma)^{-1}=(\tilde
P(\sigma)-V) \tilde P(\sigma)^{-1}=\Id-V\tilde P(\sigma)^{-1} ,
$$
and $\Id-V\tilde P(\sigma)^{-1}$ is now a bounded, indeed compact, family of operators,
continuous in the weak operator topology. Moreover, its nullspace for $\sigma=0$ is
the image of that of $P(0)$ under $\tilde P(0)$. Thus, elements $v$
of the nullspace satisfy $0=P(0)\tilde P(0)^{-1} v=v-V\tilde
P(0)^{-1}v$, i.e.\ $v=V\tilde P(0)^{-1}v$, so $v$ is in fact compactly
supported and smooth if $V$ is arranged to be such. On the other hand,
the $L^2$-orthocomplement (annihilator) of the range is the nullspace
of $P(0)^*$ in $\cY^*$. We can now consider the decomposition of
$\Id-V\tilde P(\sigma)^{-1} $ with respect to the orthogonal
decomposition of $\cY$ into $\tilde P(0)\Ker P(0)$ and its
orthocomplement on the domain side (written in the reverse order,
compatibly with the preceding discussion), and $\Ran P(0)=(\Ker P(0)^*)^\perp$ and
$\Ker P(0)^*$ on the target space side. With respect to this
decomposition, the $00$ entry (from the orthocomplement of $\tilde
P(0)\Ker P(0)$ to $\Ran P(0)=(\Ker P(0)^*)^\perp$) of this operator is
invertible for $\sigma=0$, and thus for $|\sigma|$ small, with a
uniform bound as $\sigma\to 0$. For all other entries, either the
domain is restricted to $\tilde P(0)\Ker P(0)$ or the target space is
restricted to $\Ker P(0)^*$. In the former case, using that $P(0)u=0$
means $\tilde P(0)u=Vu$, this amounts to considering the
operator
\begin{equation}\begin{aligned}\label{eq:P-zero-Ker-side}
P(\sigma)\tilde P(\sigma)^{-1}\tilde P(0)|_{\Ker P(0)}&=\tilde
P(0)-V\tilde P(\sigma)^{-1}\tilde P(0)|_{\Ker P(0)}\\
&=V(\Id-\tilde P(\sigma)^{-1} \tilde P(0))|_{\Ker P(0)}\\
&=V(\tilde P(0)^{-1}-\tilde P(\sigma)^{-1}) V |_{\Ker P(0)}.
\end{aligned}\end{equation}
Similarly, taking adjoints, in the latter case, using that
$P(0)^*\phi=0$ means $\tilde P(0)^*\phi=V^*\phi$, we are considering
\begin{equation}\begin{aligned}\label{eq:P-zero-star-Ker-side}
(\tilde P(\sigma)^{-1})^*P(\sigma)^*|_{\Ker P(0)^*}&=\Id-(\tilde
P(\sigma)^{-1})^*V^*|_{\Ker P(0)^*}\\
&=((\tilde P(0)^{-1})^*-(\tilde P(\sigma)^{-1})^*)V^*|_{\Ker P(0)^*}
\end{aligned}\end{equation}

The key point is that one would like to say that $\tilde
P(0)^{-1}-\tilde P(\sigma)^{-1}$ is small, namely the same size as
$\tilde P(0)-\tilde P(\sigma)$. Since
$$
\tilde P(0)-\tilde P(\sigma)-\sigma^2
$$
is a smooth $\sigma$ dependent family of operators in
$S^{-2-\delta}\Diffb^2(X)$ vanishing at $\sigma=0$
by assumption, it is convenient to assume that $\tilde P(\sigma)$ is
chosen so that
\begin{equation}\label{eq:tilde-P-sigma-dep}
\tilde P(0)-\tilde P(\sigma)=\sigma^2,
\end{equation}
which is thus of $O(|\sigma|^2)$. (Together with the assumption that
$V$ is independent of $\sigma$, this means that we are making an
assumption on the $\sigma$ dependence of $P(\sigma)$; as already
indicated, we will remove this assumption on $V$ in due course, while
keeping the assumption \eqref{eq:tilde-P-sigma-dep} on $\tilde P(\sigma)$.) However, the expectation for $\tilde
P(0)^{-1}-\tilde P(\sigma)^{-1}$ is not always true
for domain reasons. Namely, the `resolvent identity'
\begin{equation}\begin{aligned}\label{eq:formal-res-id}
\tilde
P(0)^{-1}-\tilde P(\sigma)^{-1}&=\tilde P(0)^{-1}(\tilde
P(\sigma)-\tilde P(0))\tilde P(\sigma)^{-1}\\
&=-\sigma^2 \tilde 
P(0)^{-1}\tilde P(\sigma)^{-1}=-\sigma^2 \tilde 
P(\sigma)^{-1}\tilde P(0)^{-1}
\end{aligned}\end{equation}
holds for $\sigma\neq 0$, with the composition justified e.g.\ as $\tilde
P(\sigma)$ is elliptic near the 0-section of the scattering cotangent
bundle, but this is of course not a uniform statement as $\sigma\to
0$, hence this expression fails to be $O(|\sigma|^2)$ in
general. Indeed, the issue with the composition is that acting on
$\Hb^{\tilde r-1,l+2}$, $\tilde P(\sigma)^{-1}$ (resp.\ $\tilde P(0)^{-1}$) produces an element of
$\Hb^{\tilde r,l}$ (uniformly bounded in $\sigma$), but as
$l<-1+\frac{n-2}{2}$, this is in the domain of $\tilde P(0)^{-1}$ (resp.\ $\tilde P(\sigma)^{-1}$) only
if $-1+\frac{n-2}{2}>1-\frac{n-2}{2}$ (i.e.\ $n>4$), in which case
one can indeed make sense of the composition in a uniformly bounded
way. Notice that for $n=4$, the failure of the composition making
sense is just the lack of an additional $\ep$ order decay, while for
$n=3$ the lack of an additional $1+\ep$ order of decay, where $\ep>0$ can be taken
arbitrarily small.

However:

\begin{lemma}
For $0\leq s\leq 2$, $l$, $\tilde r$ as before, the operator $\sigma^2\tilde P(\sigma)^{-1}$ is bounded  by
$C|\sigma|^{s}$ as a map
$$
\Hb^{\tilde r-1,l+2}\to\Hb^{\tilde
  r-2+s,l+2-s},\ 0\leq s\leq 2.
$$
\end{lemma}

\begin{rem}
Here the key part is the decay order as that is what prevents the
composition $\tilde P(0)^{-1}\tilde P(\sigma)^{-1}$ from being
uniformly bounded; the point is that we can improve the mapping
property (the target weight) by making $s$ smaller, i.e.\ by giving up
on decay as $|\sigma|\to 0$. Furthermore, since we are interested in
this expression either when applied to elements of $V\Ker P(0)$, or
paired with elements of $V^*\Ker P(0)^*$, both of which have infinite
b-differential regularity, the differential orders never matter for us.
\end{rem}

\begin{proof}
We have
\begin{equation}\label{eq:tilde-P-sigma-identity}
\sigma^2\tilde P(\sigma)^{-1}=(\tilde P(0)-\tilde P(\sigma))\tilde P(\sigma)^{-1}=-\Id+\tilde P(0) \tilde P(\sigma)^{-1}.
\end{equation}
Now, our results for $\tilde P(\sigma)$ imply that the left hand side
is bounded by $C|\sigma|^2$ as a map $\Hb^{\tilde r-1,l+2}\to\Hb^{\tilde
  r,l}$, while the right hand side is bounded by $C$ as a map
$\Hb^{\tilde r-1,l+2}\to\Hb^{\tilde r-2,l+2}$. By interpolation, this
proves the lemma since the interpolation spaces between $\Hb^{\tilde
  r,l}$ and $\Hb^{\tilde r-2,l+2}$ are exactly $\Hb^{\tilde
  r-2+s,l+2-s}$, $0\leq s\leq 2$. Recall that this can instead be
rephrased using a holomorphic family of operators which are an
isomorphism between $\Hb^{\tilde
  r-2+s,l+2-s}$, $0\leq s\leq 2$, and the base case, $s=0$. Following
a suggestion of Hintz, one can
achieve this particularly simply by {\em not} using complex powers of a
single operator rather a family of the form $x^{2z}L^z$, $L\in\Diffb^2(X)$ is formally self-adjoint, elliptic with normal
operator invertible on the reals, e.g.\ the Laplacian of a Riemannian
b-metric (asymptotically cylindrical metric)
plus $1$, with boundedness of $x^{2z}L^z \sigma^{-2z} \tilde P(\sigma)$
as a map $\Hb^{\tilde r-1,l+2}\to\Hb^{\tilde r-2,l+2}$ with $\re
z=0,1$ clear, and thus interpolation proving it for all $z$, and thus
the lemma for all $s$.
\end{proof}

This lemma gives, for $n=4$, that
\eqref{eq:formal-res-id} is bounded by $C|\sigma|^{2-\ep}$, while for
$n=3$ by $C|\sigma|^{1-\ep}$, $\ep>0$ arbitrary, between appropriate
spaces. In view of the presence of $V$ or $V^*$ in
\eqref{eq:P-zero-Ker-side} and \eqref{eq:P-zero-star-Ker-side}, the
precise spaces in between which we have the composition bound makes
little difference, provided that $V$ maps the full range of
$\cX_\sigma\subset\Hb^{\tilde r,l}$ to the full range of
$\cY=\Hb^{\tilde r-1,l'+2}$ of unrelated decay order $l'$ (in the
acceptable range of decay orders, so $|l'+1|<\frac{n-2}{2}$), that is
provided that $V\in S^{-\alpha}\Diffb^2(X)$ with $\alpha\geq (1+\frac{n-2}{2})-(-1-\frac{n-2}{2})=n$.

For later use, we note that assuming that we restrict to a subspace of
$\Ker P(0)$ or $\Ker P(0)^*$ with better decay properties, the
statement can be strengthened. For instance, writing
$$
(\tilde P(0)^{-1}-\tilde P(\sigma)^{-1})V|_{\Ker P(0)}=-\sigma^2\tilde
P(\sigma)^{-1}\tilde P(0)^{-1}V|_{\Ker P(0)}=-\sigma^2\tilde
P(\sigma)^{-1}|_{\Ker P(0)},
$$
if we restrict to a subspace of $\Ker P(0)$ which lies in
$\cap_{\ep>0}\Hb^{\infty,(n-4)/2+j-\ep}$, then we can reduce the loss
of order of vanishing by $j$ (but not to faster than quadratic
vanishing), i.e.\ when $n=3$ and $Y$ is the standard sphere we have an a priori bound of $C|\sigma|^{2-\ep}$ for
$j=1$, and (consistent with how $j\geq 2$ directly fit into
the perturbation theory discussion) a $C|\sigma|^2$ bound for $j\geq 2$.

From now on we concentrate on the most interesting case of $n=3$, also
because it is more singular than the $n=4$ case. For the sake of
definiteness, we will only consider
the case when $Y$ is the standard sphere. It is convenient to further
refine the block matrix decomposition by decomposing the 11 block into
the $j=0$ and the $j\geq 1$ parts, with the precise meaning that the
the $j\geq 1$ part lies in $\cap_{\ep>0}\Hb^{\infty,(n-4)/2+1-\ep}$,
while the $j=0$ is complementary to this in $\Ker P(0)$ (or $\Ker P(0)^*$). {\em From now on we change the
  indexing accordingly, the operator will be regarded as a 3-by-3
  block matrix, with the previous column, resp.\ row, 1 will be
  replaced by 2 columns, resp.\ rows, denoted by 1 and 2, with 1
  corresponding to $j=0$, 2 corresponding to $j\geq 1$.}

Our argument so far proves:

\begin{prop}\label{prop:a-priori-block-bounds}
Suppose that $V=V(\sigma)$ is independent of $\sigma$ and $n=3$.

The entries of the block matrix of $P(\sigma)\tilde
P(\sigma)^{-1}$ in the above sense have bounds
$$
\begin{pmatrix} O(1)&O(|\sigma|^{1-\ep})&O(|\sigma|^{2-\ep})\\O(|\sigma|^{1-\ep})&O(|\sigma|^{1-\ep})&O(|\sigma|^{2-\ep})\\O(|\sigma|^{2-\ep})&O(|\sigma|^{2-\ep})&O(|\sigma|^{2-\ep})\end{pmatrix},
$$
with $\ep>0$ arbitrary,
and furthermore the 00 entry has a bounded inverse.
\end{prop}

We next strengthen this if $\tilde P(\sigma)$ is the spectral family
of the Euclidean Laplacian:

\begin{prop}\label{prop:refined-block-bounds}
Suppose that $V=V(\sigma)$ is independent of $\sigma$, and  
that $\tilde P(\sigma)=\Delta_{\RR^n}-\sigma^2$, 
$n=3$. 

Consider the block matrix of
Proposition~\ref{prop:a-priori-block-bounds}. Suppose that the $L^2$
pairing between the $j\geq 1$ part of $\Ker P(0)$ and the
corresponding part of $\Ker P(0)^*$ is non-degenerate, and that the
$L^2(Y)$ pairing between the leading asymptotic terms of the $j=0$
parts of $\Ker P(0)$ and $\Ker P(0)^*$ is also non-degenerate.

Then for $\sigma\neq 0$, the 11, resp.\ 22 entries are invertible with
bounds $C|\sigma|$, resp.\ $C|\sigma|^{2}$, and with the respected
inverses bounded by $C|\sigma|^{-1}$, resp.\ $C|\sigma|^{-2}$, while
the 12 and 21 entries are bounded by $C|\sigma|^2$.

Here the lower right 2-by-2-block in fact has entries that are 
multiples of the indicated power of $\sigma$ up to a term with an
extra $O(|\sigma|^{1-\ep})$ vanishing, and the diagonal entries
have inverses that are multiples of the indicated power of
$\sigma^{-1}$, again up to a term with an
extra $O(|\sigma|^{1-\ep})$ vanishing. 
\end{prop}

Note that in particular the non-degeneracy of the pairing of $\Ker P(0)$ and $\Ker P(0)^*$ holds if $P(0)=P(0)^*$.

\begin{proof}
Notice that this proposition is
purely a statement involving columns and rows 1 and 2, thus in both
the `input' and the `output' slots we are restricted to elements of
(the image of) $\Ker P(0)$, resp. $\Ker P(0)^*$. 

The 11, 12, 21 and 22 entries are still higher (but finite) rank in general; we
evaluate them on an element $u$ of $\Ker P(0)$ paired with an element
$\phi$ of
$\Ker P(0)^*$. Thus, we need to evaluate
\begin{equation}\label{eq:Eucl-res-diff}
\langle (\tilde P(0)^{-1}-\tilde P(\sigma)^{-1})Vu,V^*\phi\rangle.
\end{equation}

Now $\tilde P(0)^{-1}-\tilde P(\sigma)^{-1}$ is a Fourier multiplier
by
$$
|\xi|^{-2}-(|\xi|^2-(\sigma\pm
i0)^2)^{-1}=-\sigma^2|\xi|^{-2}(|\xi|^2-(\sigma\pm i0)^2)^{-1},
$$
and thus \eqref{eq:Eucl-res-diff} is, with the hat denoting the
Fourier transform,
$$
-(2\pi)^{-3}\int_{\RR^3} \sigma^2|\xi|^{-2}(|\xi|^2-(\sigma\pm i0)^2)^{-1} \widehat{Vu}(\xi)\widehat{V^*\phi}(\xi)\,d\xi.
$$
Changing variables to $\eta=\xi/|\sigma|$, we obtain, with $\hat\sigma=\frac{\sigma}{|\sigma|}$,
$$
-(2\pi)^{-3}\sigma^2|\sigma|^{-1}\int_{\RR^3}
|\eta|^{-2}(|\eta|^2-(\hat\sigma\pm i0)^2)^{-1} \widehat{Vu}(|\sigma|\eta)\widehat{V^*\phi}(|\sigma|\eta)\,d\eta.
$$
Now, as $\eta\mapsto |\eta|^{-2}(|\eta|^2-(\hat\sigma\pm i0))^{-1}$ is $L^1$
away from $|\eta|=1$, and near $|\eta|=1$ it can be considered a
distribution of compact support (technically one uses a partition of
unity to combine these two results), the integral is uniformly
bounded as $\sigma\to 0$, and indeed converges as $\sigma\to 0$ to
$$
\widehat{Vu}(0)\widehat{V^*\phi}(0) \int_{\RR^3}
|\eta|^{-2}(|\eta|^2-(\hat\sigma\pm i0)^2)^{-1} \,d\eta.
$$
This new integral can be computed, for the sake of definiteness for
the $+$ sign in $\pm$ (the other case is the complex conjugate) by writing it as
\begin{equation}\begin{aligned}\label{eq:explicit-Euclidean-comp-2}
&\int_{\RR^3}
|\eta|^{-2}(|\eta|^2-(\hat\sigma+ i0)^2)^{-1} \,d\eta=4\pi\int_0^\infty(\rho^2-(\hat\sigma+ i0)^2)^{-1}
\,d\rho\\
&=2\pi\hat\sigma^{-1}\int_{0}^\infty ((\rho-(\hat\sigma+
i0))^{-1}-(\rho+(\hat\sigma+
i0))^{-1})\,d\rho\\
&=2\pi\hat\sigma^{-1}\log\frac{\rho-(\hat\sigma+i0)}{\rho+(\hat\sigma+i0)}\Big|_0^\infty=-2\pi(-\pi
i)\hat\sigma^{-1}=2\pi^2 i\hat\sigma^{-1},
\end{aligned}\end{equation}
which is non-zero.
Now, $\widehat{Vu}(0)$ is a non-vanishing multiple of the leading
asymptotic coefficient, namely that of $x=r^{-1}$, of $\tilde P(0)^{-1} Vu$ since the latter is
the inverse Fourier transform of $|\xi|^{-2}\widehat{Vu}$ (so the
multiple is actually $4\pi$ in view of the inverse Fourier transform of $|\xi|^{-2}$), but as
$Vu=\tilde P(0) u$, this is exactly $\tilde P(0)^{-1} Vu=u$.
Thus, the 11 entry is actually a constant ($\frac{1}{4\pi i} (4\pi)^2$) times $\sigma$ times the pairing between the leading
(constant!) coefficient of the expansion of $\Ker P(0)$ and $\Ker
P(0)^*$. Thus, the 11 entry is invertible for $\sigma\neq 0$, and the
inverse is bounded by $C|\sigma|^{-1}$. Correspondingly, in the
absence of the $j\geq 1$ parts (column and row 2 of the matrix)
$P(\sigma)\tilde P(\sigma)^{-1}$, and thus $P(\sigma)$, is invertible
for $\sigma\neq 0$, with a bound $C|\sigma|^{-1}$, and in the block
decomposition the inverse is
$$
\begin{pmatrix} O(1)&O(|\sigma|^{-\ep})\\O(|\sigma|^{-\ep})&O(|\sigma|^{-1})\end{pmatrix},
$$
with the leading term of the 11 entry being $4\pi i\sigma^{-1}$ if $u$ is
normalized to have leading term $\frac{x}{4\pi}$.

This computation also shows that if either $u$ or $\phi$ have
vanishing leading term, then indeed the pairing in
$O(|\sigma|^2)$. Hence it remains to compute the 22 block.
For this we note that
\begin{equation}\begin{aligned}\label{eq:res-pairing-22-block-compute}
&\langle (\tilde P(0)^{-1}-\tilde P(\sigma)^{-1})Vu,V^*\phi\rangle\\
&=-\sigma^2\langle \tilde P(\sigma)^{-1}Vu,(\tilde
P(0)^{-1})^*V^*\phi\rangle\\
&=-\sigma^2\langle \tilde P(0)^{-1}Vu,\phi\rangle-\sigma^2\langle (\tilde
P(\sigma)^{-1}-\tilde P(0))^{-1}Vu,\phi\rangle\\
&=-\sigma^2\langle u,\phi\rangle+\sigma^2\langle \sigma^2\tilde
P(\sigma)^{-1}\tilde P(0)^{-1}Vu,\phi\rangle\\
&=-\sigma^2\langle u,\phi\rangle+\sigma^2\langle \sigma^2\tilde
P(\sigma)^{-1}u,\phi\rangle
\end{aligned}\end{equation}
Now, the preceding discussion shows that the second term is
$O(|\sigma|^{3-\ep})$, so the 22 entry of the block matrix is simply
$-\sigma^2$ times the pairing between $\Ker P(0)$ and $\Ker P(0)^*$,
modulo $O(|\sigma|^2)$, so as long as this pairing is non-degenerate,
the block is invertible with an $O(|\sigma|^{-2})$ inverse, completing
the proof.
\end{proof}

From this proposition a standard
argument, using (block-)Gaussian elimination, immediately shows:

\begin{thm}\label{thm:asymp-Eucl-resonances}
Under the assumptions of Proposition~\ref{prop:refined-block-bounds},
the whole block
matrix of $P(\sigma)\tilde P(\sigma)^{-1}$ is invertible for
$\sigma\neq 0$ with inverse having block matrix with bounds
$$
\begin{pmatrix} O(1)&O(|\sigma|^{-\ep})&O(|\sigma|^{-\ep})\\O(|\sigma|^{-\ep})&O(|\sigma|^{-1})&O(|\sigma|^{-1})\\O(|\sigma|^{-\ep})&O(|\sigma|^{-1})&O(|\sigma|^{-2})\end{pmatrix}.
$$
Consequently, $P(\sigma)^{-1}$ has the same structure.

Here the lower right 2-by-2-block in fact has entries that are smooth 
multiples of the indicated power of $\sigma^{-1}$ up to a term with an
extra $O(|\sigma|^{1-\ep})$ vanishing.
\end{thm}

Note that this theorem can be
applied if $P(\sigma)=P(0)-\sigma^2$ is a spectral family on
$X=\overline{\RR^n}$, with $Y$ the standard Riemannian sphere, {\em with no additional assumptions on $P(\sigma)$}.
A key consequence of this for applications is that when applied to
elements of the complement of $\Ker P(0)$, $P(\sigma)^{-1}$ is bounded
by $O(|\sigma|^{-\ep})$, and the same holds when $P(\sigma)^{-1}$ is
applied to any element of the domain, provided the result is projected
to $\Ran P(0)$.

Finally we consider the case of $P(\sigma)$ being a
Kerr-type operator, i.e.\ $X$ has two boundary hypersurfaces $\pa_+ X=Y$
and $\pa_-X$ as in the previous section, with the former being the
Euclidean end, and with $V(\sigma)$ not assumed to be independent of
$\sigma$. Notice that with this arrangement one may keep $\tilde
P(\sigma)-\tilde P(0)=-\sigma^2$, even if $P(\sigma)$ equals the Kerr
`spectral family' near the Euclidean end $\pa_+X$ and one wants $\tilde
P(\sigma)$ to be the spectral family of the Euclidean Laplacian near
$\pa_+X$.

We prove:

\begin{thm}\label{thm:Kerr-resonances}
Suppose that $\tilde P(\sigma)$ equals $\Delta_{\RR^n}-\sigma^2$, 
$n=3$, near $\pa_+X$.

Consider the block matrix of
Proposition~\ref{prop:a-priori-block-bounds}. Suppose that the $L^2$
pairing between the $j\geq 1$ part of $\Ker P(0)$ and the
corresponding part of $\Ker P(0)^*$ is non-degenerate, and that the
$L^2(Y)$ pairing between the leading $\pa_+X$-asymptotic terms of the $j=0$
parts of $\Ker P(0)$ and $\Ker P(0)^*$ is also non-degenerate.

Suppose also that $V(\sigma)-V(0)$ annihilates $\Ker P(0)$ and of $\Ker P(0)^*$.

Then
the whole block
matrix of $P(\sigma)\tilde P(\sigma)^{-1}$ is invertible for
$\sigma\neq 0$ with inverse having block matrix with bounds
$$
\begin{pmatrix} O(1)&O(|\sigma|^{-\ep})&O(|\sigma|^{-\ep})\\O(|\sigma|^{-\ep})&O(|\sigma|^{-1})&O(|\sigma|^{-1})\\O(|\sigma|^{-\ep})&O(|\sigma|^{-1})&O(|\sigma|^{-2})\end{pmatrix}.
$$
Consequently, $P(\sigma)^{-1}$ has the same structure.

Here the lower right 2-by-2-block in fact has entries that are smooth 
multiples of the indicated power of $\sigma^{-1}$ up to a term with an
extra $O(|\sigma|^{1-\ep})$ vanishing.
\end{thm}

\begin{rem}
As follows from the proof below, it suffices to make the assumption on
$V(\sigma)-V(0)$ for the $j\geq 1$ part of $\Ker P(0)$ and of $\Ker
P(0)^*$, provided $V(\sigma)-V(0)$ has sufficiently small coefficients.
\end{rem}

\begin{rem}
Notice that if, within our framework, one perturbs an operator for which the non-degeneracy
of the pairings is known, the same non-degeneracy will hold for
sufficiently small perturbations.
\end{rem}

\begin{rem}
If one has a family of operators $P_b(\sigma)$ as in
Remark~\ref{rem:bd-state-stable}, cf.\ Remark~\ref{rem:Kerr-perturb},
and performs the block matrix decomposition for $P_0(0)$, then the 00
block remains invertible for $P_b(0)$, $|b|$ sufficiently small, by
perturbation stability. On the other hand, in general, $\Ker P_b(0)$ and
$\Ker P_b(0)^*$ may become lower dimensional. However, if one a priori
knows that there are subspaces of $\Ker P_b(0)$ and $\Ker P_b(0)^*$
which vary continuously with $b$, and for $b=0$ are $\Ker P_0(0)$ and
$\Ker P_0(0)^*$, then these are necessarily all of $\Ker P_b(0)$ and
$\Ker P_b(0)^*$ since the dimension of the latter spaces (for $b\neq
0$, $|b|$ small) is bounded by that
of that of the former spaces (by block-Gaussian elimination). In addition, in this case, the pairings
in Theorem~\ref{thm:Kerr-resonances}, themselves being continuous in
$b$, are non-degenerate as well for $b\neq 0$, i.e.\ the theorem is
equally applicable then.
\end{rem}

\begin{proof}
We start the proof with a simpler setting, when
$X=\overline{\RR^n}$ for now with $Y$ the standard sphere, but we
allow $V$ to depend on $\sigma$, so $V(\sigma)=\tilde
P(\sigma)- P(\sigma)$. (So this in particular already incorporates the
behavior of the Kerr spectral family near $\pa_+X$.)
In this case $\tilde P(0)^{-1}V(0)$
is the identity on $\Ker P(0)$. Moreover,
\begin{equation}\begin{aligned}\label{eq:P-zero-Ker-side-mod}
&P(\sigma)\tilde P(\sigma)^{-1}\tilde P(0)|_{\Ker P(0)}\\
&=\tilde
P(0)-V(0)\tilde P(\sigma)^{-1}\tilde P(0)-(V(\sigma)-V(0))\tilde P(\sigma)^{-1}\tilde P(0)|_{\Ker P(0)}\\
&=V(0)(\Id-\tilde P(\sigma)^{-1} \tilde P(0)) -(V(\sigma)-V(0))\tilde P(\sigma)^{-1}\tilde P(0)|_{\Ker P(0)}\\
&=V(0)(\tilde P(0)^{-1}-\tilde P(\sigma)^{-1}) V(0)
-(V(\sigma)-V(0))\tilde P(\sigma)^{-1} V(0)|_{\Ker P(0)}.
\end{aligned}\end{equation}
Now the first term behaves exactly as before, when $V$ was assumed to
be independent of $\sigma$, while for the second
term we can write
\begin{equation}\begin{aligned}\label{eq:P-zero-Ker-side-mod-2nd-term}
&(V(\sigma)-V(0))\tilde P(\sigma)^{-1} V(0)|_{\Ker P(0)}\\
&=(V(\sigma)-V(0))+(V(\sigma)-V(0)) (\tilde P(\sigma)^{-1}-\tilde P(0)^{-1}) V(0)|_{\Ker P(0)}.
\end{aligned}\end{equation}
The second term of this expression can be analyzed as before, with the
prefactor $V(\sigma)-V(0)$ giving an extra $\sigma$ vanishing, thus it
contributes $O(|\sigma|^{2-\ep})$ in all cases, indeed
$O(|\sigma|^{3-\ep})$ in the $j\geq 1$ input case (i.e.\ column 2). On the other hand, in
general, the first term has a non-trivial $O(|\sigma|)$ contribution.

Similarly, taking adjoints, in the latter case, using that
$P(0)^*\phi=0$ means $\tilde P(0)^*\phi=V(0)^*\phi$, we are considering
\begin{equation}\begin{aligned}\label{eq:P-zero-star-Ker-side-mod}
&(\tilde P(\sigma)^{-1})^*P(\sigma)^*|_{\Ker P(0)^*}=\Id-(\tilde
P(\sigma)^{-1})^*V(\sigma)^*|_{\Ker P(0)^*}\\
&=((\tilde P(0)^{-1})^*-(\tilde P(\sigma)^{-1})^*)V(0)^*-(\tilde
P(\sigma)^{-1})^*(V(\sigma)-V(0))^*|_{\Ker P(0)^*}
\end{aligned}\end{equation}
Again, the first term is as before, while for the second
\begin{equation}\begin{aligned}\label{eq:P-zero-star-Ker-side-mod-2nd-term}
&(\tilde
P(\sigma)^{-1})^*(V(\sigma)-V(0))^*|_{\Ker P(0)^*}\\
&=(\tilde
P(0)^{-1})^*(V(\sigma)-V(0))^*+(\tilde
P(\sigma)^{-1}-\tilde P(0)^{-1})^*(V(\sigma)-V(0))^*|_{\Ker P(0)^*},
\end{aligned}\end{equation}
and one again has an $O(|\sigma|^{2-\ep})$ estimate for the new second
term, while the first is $O(|\sigma|)$.

Now, $O(|\sigma|^{2-\ep})$ entries were negligible in our preceding discussion
{\em except} for the 22 entry, where $o(|\sigma|^2)$ entries are
negligible, thus $O(|\sigma|^{3-\ep})$ is in particular always negligible. Thus, the second term of
\eqref{eq:P-zero-Ker-side-mod-2nd-term} and
\eqref{eq:P-zero-star-Ker-side-mod-2nd-term} are both negligible for
our purposes, and the only potential issue is the behavior of the
first term in both of these, since it is only $O(|\sigma|)$. Such terms were negligible in our preceding discussion
{\em except} for the 02,20,11,12,21 and 22 entries, and in fact {\em small}
$O(|\sigma|)$ entries are negligible even there {\em except} for the
12, 21 (where small $O(|\sigma|^{3/2})$ works) and 22 entries (where
we need small $O(|\sigma|^2)$), by the Gaussian
elimination argument (though the inverse will have larger entries
in the 02, 12, 20 and 21 slots).

So if we assume that $V(\sigma)-V(0)$ annihilates the $j\geq 1$ part
of $\Ker P(0)$ and its adjoint of $\Ker P(0)^*$, and is small on the $j=0$ part (the 11 slot), we have
exactly the same result as if $V$ is independent of $\sigma$, namely
the analogue of Proposition~\ref{prop:refined-block-bounds}
holds. This proves Theorem~\ref{thm:Kerr-resonances} under the
stronger assumption on $X$.

It remains to deal with the Cauchy hypersurface, $\pa_-X$ and to
ensure that all of our arguments go through in this case as
well. All of the discussion concerning $V(\sigma)$ above remain valid,
i.e.\ the new terms with $V(\sigma)-V(0)$ can be handled as above. Moreover,
the computation leading to \eqref{eq:res-pairing-22-block-compute}
remains valid even in this case as well. Thus, in order to show the analogue
of Proposition~\ref{prop:refined-block-bounds} it remains to show that
even in this more general setting
the contribution of the first term of the right hand sides of
\eqref{eq:P-zero-Ker-side-mod} and \eqref{eq:P-zero-star-Ker-side-mod}
to the 12
and 21 blocks are $O(|\sigma|^2)$ and compute the 11 block.
For this we introduce an auxiliary differential operator $\tilde P_0(\sigma)$  that equals $\tilde
P(\sigma)$ near the Euclidean end, $\pa_+ X$, so that if $\chi\in\CI(X;\RR)$
identically $1$ near $\pa_+X$, supported away from $\pa_-X$, then
$\chi P(\sigma)=\chi \tilde P(\sigma)$. Here we can consider $\tilde
P_0(\sigma)$ as an operator on a different manifold; concretely, in
the case of interest, of $Y$ being the standard sphere, we replace $X$
by $\tilde X=\overline{\RR^n}$, and assume that $\tilde P_0(\sigma)$
is an operator on this space, and indeed that it is the spectral
family of the Euclidean Laplacian.

Then
$$
\tilde P_0(\sigma)\chi\tilde P(\sigma)^{-1}V(\sigma)u=(\chi +[\tilde P_0(\sigma),\chi]\tilde P(\sigma)^{-1})V(\sigma)u
$$
in $\cY$, thus
$$
\chi\tilde P(\sigma)^{-1}V(\sigma)u=\tilde P_0(\sigma)^{-1}(\chi +[\tilde P_0(\sigma),\chi]\tilde P(\sigma)^{-1})V(\sigma)u,
$$
and similarly for $\chi\tilde P(0)^{-1}V(\sigma)u$. Notice that
\begin{equation}\label{eq:tilded-Ps-same-asymp}
\tilde
P(\sigma)^{-1}V(\sigma)u,\ \chi\tilde P(\sigma)^{-1}V(\sigma)u,\ \tilde
P_0(\sigma)^{-1}(\chi +[\tilde P_0(\sigma),\chi]\tilde
P(\sigma)^{-1})V(\sigma)u
\end{equation}
thus have the same asymptotic behavior near $\pa_+ X$, since they are
actually equal there.

Now,
\begin{equation*}\begin{aligned}
&\langle (\tilde P(0)^{-1}-\tilde
P(\sigma)^{-1})V(\sigma)u,V(\sigma)^*\phi\rangle\\
&=-\sigma^2\langle \tilde 
P(\sigma)^{-1}V(\sigma)u,(\tilde P(0)^{-1})^*V(\sigma)^*\phi\rangle\\
&=-\sigma^2\langle \chi\tilde 
P(\sigma)^{-1}V(\sigma)u,\chi(\tilde
P(0)^{-1})^*V(\sigma)^*\phi\rangle\\
&\qquad-\sigma^2\langle (1-\chi^2)\tilde 
P(\sigma)^{-1}V(\sigma)u,(\tilde P(0)^{-1})^*V(\sigma)^*\phi\rangle\\
\end{aligned}\end{equation*}
Due to $1-\chi^2$ vanishing near $\pa_+X$, the composition $\tilde
P(0)^{-1}(1-\chi^2)\tilde P(\sigma)$ remains uniformly bounded as
$|\sigma|\to 0$, so the second term is $O(|\sigma|^2)$, so it remains
to analyze the first term, which is
$$
-\sigma^2\langle \tilde 
P_0(\sigma)^{-1}f,(\tilde P_0(0)^{-1})^*\psi\rangle
$$
with $f=(\chi +[\tilde P_0(\sigma),\chi]\tilde P(\sigma)^{-1})V(\sigma)u$,
etc. But now $\tilde P_0(\sigma)$ is the spectral family of the
Euclidean Laplacian, so our computation from the proof of Proposition~\ref{prop:refined-block-bounds} is applicable,
replacing $V(\sigma)u$ by $f$, and $V(\sigma)^*\phi$ by $\psi$. Now, as remarked after
\eqref{eq:explicit-Euclidean-comp-2}, the 11 entry is the pairing
between the leading asymptotic coefficients of $\tilde 
P_0(\sigma)^{-1}f$ and $(\tilde P_0(0)^{-1})^*\psi$, thus, in view of
\eqref{eq:tilded-Ps-same-asymp}, of $\tilde 
P(\sigma)^{-1}V(\sigma)u$ and $(\tilde P(0)^{-1})^*V(\sigma)^*\phi$, i.e.\ in terms of
the asymptotics, it is given
by exactly the same expression as beforehand! Similarly, the 12 and 21
entries are $O(|\sigma|^2)$ since this leading term vanishes. In
summary, the structure of our block matrix is the same even in the
presence of Cauchy hypersurfaces $\pa_-X$, provided the pairing between the
leading $\pa_+X$-asymptotic terms of the
$j=0$ resonant states and dual states is non-degenerate, and provided
that the $L^2$-pairing between the $j\geq 1$ resonant states and dual
states is non-degenerate, completing the proof.
\end{proof}

\begin{rem}
A simple extension of the proof of this theorem shows that if $\Ker
P(0)$ has a non-trivial $j=0$ component, but $\Ker P(0)^*$ does not,
choosing a non-degenerate dual decomposition of $\Ker P(0)^*$
(assuming the existence of this), the
structure of $P(\sigma)\tilde P(\sigma)^{-1}$
is
$$
\begin{pmatrix} O(1)&O(|\sigma|^{1-\ep})&O(|\sigma|^{2-\ep})\\O(|\sigma|^{2-\ep})&O(|\sigma|^{2})&O(|\sigma|^{2})\\O(|\sigma|^{2-\ep})&O(|\sigma|^{2})&O(|\sigma|^{2})\end{pmatrix},
$$
where non-degeneracy means that the 2-by-2 lower right block is
$\sigma^2$ times invertible. Notice that in this form the 10 and 11 entries vanish to one
order higher than before. Then $P(\sigma)\tilde P(\sigma)^{-1}$
is invertible for
$\sigma\neq 0$, and the inverse has block matrix with bounds
$$
\begin{pmatrix} O(1)&O(|\sigma|^{-1-\ep})&O(|\sigma|^{-\ep})\\O(|\sigma|^{-\ep})&O(|\sigma|^{-2})&O(|\sigma|^{-2})\\O(|\sigma|^{-\ep})&O(|\sigma|^{-1})&O(|\sigma|^{-2})\end{pmatrix}.
$$

This structure can happen in interesting examples. For instance, for
the wave equation on
1-forms on Kerr space, 0-resonances can be read off from
\cite[Theorem~4.4, Lemma~4.6, Section~5]{Hintz-Vasy:Kerr-forms} (while
this paper is in the Kerr-de Sitter setting,
appropriate explicit resonant and dual resonant states persist in
Kerr, with the appropriateness coming from only considering solutions
with desired asymptotics as $r\to\infty$). In particular, for
Schwarzschild there is a  resonant state which is a suitable linear
combination of $dr$ and $dt$, with non-vanishing $O(r^{-1})$ leading term, but the dual state is a delta
distribution on the horizon, thus is identically zero near the
Euclidean end. Indeed, this delta distributional nature of the dual
state persists for slowly rotating Kerr space, see
\cite[Section~5]{Hintz-Vasy:Kerr-forms}, so in fact this structure is
stable within the family.
\end{rem}

\bibliographystyle{plain}
\bibliography{sm}

\end{document}